\pdfoutput=1

 \documentclass[twoside]{scrartcl} 
 \usepackage{savesym} 
 \usepackage{amsmath}
 \savesymbol{openbox}
 \usepackage{stmaryrd} 
 \usepackage{soul}
 \usepackage{cite}  
 \usepackage{amsfonts} 
 \usepackage{amssymb}
 \usepackage{txfonts}
\restoresymbol{txfonts}{openbox}
 \usepackage{mathtools} 
 \usepackage[inline]{enumitem} 
 \usepackage[english]{babel} 

 \usepackage{amsthm}

\makeatletter
\newcommand{\specialcell}[1]{\ifmeasuring@#1\else\omit$\displaystyle#1$\ignorespaces\fi}
\makeatother

 \DeclareFontFamily{OT1}{pzc}{}
 \DeclareFontShape{OT1}{pzc}{m}{it}%
               {<-> s * [1.15] pzcmi7t}{} 
 \DeclareMathAlphabet{\mathpzc}{OT1}{pzc}%
                                  {m}{it}

 \usepackage{tikz}
 \usetikzlibrary{matrix,arrows}
 \usetikzlibrary{calc}
 \usetikzlibrary{positioning}

\newcommand{\Rrrightarrow}{%
~
\begingroup
\tikzset{every path/.style={}}%
\tikz \draw (0,1.5pt) -- ++(.6em-1pt,0) (0,.5pt) -- ++(.6em+.5pt,0)
(0,-.5pt) -- ++(.6em+.5pt,0) (0,-1.5pt) -- ++(.6em-1pt,0)
(.6em-2pt,2.5pt) to (.6em+1pt,0)  to (.6em-2pt,-2.5pt)
(.6em-2pt,2.5pt) to (.6em+.5pt,0)  to (.6em-2pt,-2.5pt)
;
\endgroup
~
}

\makeatletter
\newcommand\sec@label[1]{%
\@bsphack
  \protected@write\@auxout{}%
         {\string\newlabel{sec@#1}{{\thesubsection}{\thepage}}}%
  \@esphack}

\def\label@in@display#1{%
    \ifx\df@label\@empty\else
        \@amsmath@err{Multiple \string\label's:
            label '\df@label' will be lost}\@eha
    \fi
    \sec@label{#1}\gdef\df@label{#1}%
}
\def\secref#1{\expandafter\@setref\csname r@sec@#1\endcsname\@firstoftwo{#1}}
\def\secpageref#1{\expandafter\@setref\csname
  r@sec@#1\endcsname\@secondoftwo{#1}}

\let\origlabel\label 
\def\label#1{\origlabel{#1}\sec@label{#1}} 
\makeatother

 \usepackage[center]{titlesec}

\sodef\kapsperrt{}{.02em}{.6em plus1em}{1em plus.1em minus.1em}

  \titleformat{\section}
  {  \vspace{1ex} %
   \normalfont\fillast\Large\sc}{\thesection}{1em}{}
     \titlespacing{\section}
     {3pc}{*3}{*2}[3pc]
 \titleformat{\subsection}
 {\vspace{1ex} %
 \titlerule
 \vspace{.8ex}%
 \normalfont
}
 {\textbf\thesubsection}{3em}{\textbf  }

 \usepackage[automark]{scrpage2}
 \clearscrheadings
 \clearscrplain

 \automark[section]{section}

\setkomafont{pagehead}{\normalfont}
\cehead[]{\MakeUppercase{Homomorphisms of Gray-categories as pseudo algebras}}
\cohead[]{\MakeUppercase{Homomorphisms of Gray-categories as pseudo algebras}}

 \setkomafont{pagenumber}{%
 \normalfont\rmfamily
 }

 \rohead[\pagemark]{\pagemark} 

 \lehead[\pagemark]{\pagemark}

 \catcode`\@=11 

 \def\endamsmathequationenvironment{
   \endmathdisplay{equation}%
   \mathdisplay@pop
   \ignorespacesafterend
 }

 \def\endamsmathequationenvironmentunnumbered{
 \endmathdisplay{equation*}%
 \mathdisplay@pop
 \ignorespacesafterend
 }
 \def\beginamsmathequationenvironment{
   \incr@eqnum
   \mathdisplay@push
   \st@rredfalse \global\@eqnswtrue
   \mathdisplay{equation}%
 }

 \renewenvironment{equation}{%
   \incr@eqnum
   \mathdisplay@push
   \st@rredfalse \global\@eqnswtrue
   \mathdisplay{equation}%
 }{%
 \endamsmathequationenvironment
 }
 \renewenvironment{equation*}{%
   \mathdisplay@push
   \st@rredtrue \global\@eqnswfalse
   \mathdisplay{equation*}%
 }{%
 \endamsmathequationenvironmentunnumbered
 }

 \catcode`\@=12 

 \catcode`\@=11 
 \def\puncteq#1{~#1}
 \def\puncttikz[#1]#2{
 \pgfnodealias{final_node}{#1};
 \node (punctuation) [base right=of final_node,xshift=-3em]{\,#2};
 }
 \catcode`\@=12 

 \catcode`\@=11 

 \def\addtopunct#1{\expandafter\let\csname punct@\meaning#1\endcsname\let}
 \addtopunct{.}    \addtopunct{,}    \addtopunct{?}
 \addtopunct{!}    \addtopunct{;}    \addtopunct{:}

 \let\seveendformula\endamsmathequationenvironment
 \def\PunctAndEndFormula #1{~~#1\seveendformula}
 \def\endamsmathequationenvironment{\futurelet\punctlet\checkpunct@i}
 \def\checkpunct@i{\expandafter\ifx\csname punct@\meaning\punctlet\endcsname\let  
        \expandafter\PunctAndEndFormula 
        \else \expandafter\seveendformula\fi}

 \catcode`\@=12 

\sodef\titelsperrt{}{.07em}{.7em plus1em}{1em plus.1em minus.1em}
\sodef\namsperrt{}{}{.3em plus1em}{1em plus.1em minus.1em}

\begin{document}

 \def\co{\colon\thinspace}
 \def\oc{\thinspace\colon\mkern-5.5mu}
\def\op{^{\mathrm{op}}}
\def\astpar{$(^\ast\hspace{-.55pt})$}

 \newtheorem{theorem}{Theorem}
 \newtheorem{lemma}{Lemma}
 \newtheorem{proposition}{Proposition}
 \newtheorem{corollary}{Corollary}
 \theoremstyle{definition}
 \newtheorem{definition}{Definition}
 \newtheorem{example}{Example}
\newtheorem*{example*}{Example}
 \newtheorem{remark}{Remark}

 \thispagestyle{empty}

\begin{tikzpicture}[remember picture, overlay]  
\node[yshift=-3.42cm] at (current page.north)  (preprintNummern){%
    \begin{minipage}{\textwidth}
        \begin{flushright}
   { {\footnotesize  ZMP-HH/14-13} \\ {\footnotesize  Hamburger Beitr\"age  zur Mathematik Nr. 511}
   }\\
{ {\footnotesize August 2014}}
 \end{flushright}
    \end{minipage}
};
   \end{tikzpicture}

\vspace{-0.5cm}
 \begin{center}
 
\begin{Large}{\sc{Homomorphisms of Gray-categories as pseudo algebras}}\end{Large}

 \vspace{0.5cm}

    \normalsize{\sc{Lukas Buhn\'e}}

\vspace{0.5cm}
\footnotesize{\textit{Fachbereich Mathematik, Universit\"at Hamburg,
    }}

\footnotesize{\textit{Bereich Algebra und Zahlentheorie,
    }}

\footnotesize{\textit{
Bundesstraße 55, D-20146 Hamburg, Germany}}
 \end{center}

\vspace{0.2cm}
\noindent\rule{\textwidth}{0.4pt}
 \begin{abstract}\noindent {\sc Abstract.}
\small{
Given $\mathpzc{Gray}$-categories $\mathpzc{P}$ and $\mathpzc{L}$, 
there is a $\mathpzc{Gray}$-category
$\mathpzc{Tricat}_{\mathrm{ls}}(\mathpzc{P},\mathpzc{L})$ of locally strict
trihomomorphisms with domain
$\mathpzc{P}$ and codomain $\mathpzc{L}$, tritransformations,
trimodifications,
and perturbations. If the domain $\mathpzc{P}$ is small and the
codomain  $\mathpzc{L}$ is cocomplete, we show
that this
$\mathpzc{Gray}$-category is isomorphic as a
$\mathpzc{Gray}$-category to the $\mathpzc{Gray}$-category $\mathrm{Ps}\text{-}T\text{-}\mathrm{Alg}$ of pseudo
algebras, pseudo functors, transformations, and 
modifications for a $\mathpzc{Gray}$-monad $T$ derived from left Kan
extension. 

Inspired by a similar situation in two-dimensional monad theory, we
apply the coherence theory of three-dimensional monad theory and prove
 that the the inclusion of the functor category in the enriched sense into this $\mathpzc{Gray}$-category
of locally strict trihomomorphisms has a left adjoint such that the components of the unit
of the adjunction are internal biequivalences. 
This proves that any
locally strict trihomomorphism
between $\mathpzc{Gray}$-categories with small domain and cocomplete
codomain is biequivalent to a $\mathpzc{Gray}$-functor. Moreover, the
hom $\mathpzc{Gray}$-adjunction gives an isomorphism
of the hom $2$-categories of
tritransformations between a locally strict trihomomorphism and a
$\mathpzc{Gray}$-functor with the corresponding hom $2$-categories in the functor $\mathpzc{Gray}$-category.
A notable example is given by locally strict $\mathpzc{Gray}$-valued presheafs with
small domain.
Our results have applications in three-dimensional descent theory and point into the direction of a Yoneda lemma for tricategories.
}
 \end{abstract}

 \vspace{-0.3cm}

 \noindent\rule{\textwidth}{0.4pt}
                         
 \selectlanguage{english}


\vspace{0.5cm}

\section{Introduction}
Three-dimensional monad theory is the study of $\mathpzc{Gray}$-monads
and their different kinds of algebras.
While three-dimensional monad theory is by now
well-developed, see \cite[Part III]{gurskicoherencein} and
\cite{powerthreedimensional}, there 
are only few examples. 
This paper is based on the insight that one example from
two-dimensional monad theory can be transferred to the
three-dimensional context. This
is in fact nontrivial since the amount of computation is considerably
higher than in the two-dimensional context. On the other hand, we have been
in need of exactly this three-dimensional
result for applications in three-dimensional descent theory. 

We here provide the details of the reformulation.
We
show how under suitable conditions on domain and codomain the locally
strict trihomomorphisms between
$\mathpzc{Gray}$-categories $\mathpzc{P}$ and $\mathpzc{L}$ correspond to pseudo algebras for a
$\mathpzc{Gray}$-monad $T$ on the functor $\mathpzc{Gray}$-category $[\mathrm{ob}\mathpzc{P},\mathpzc{L}]$
derived from left Kan extension, where $\mathrm{ob}\mathpzc{P}$ is the underlying
set of objects of $\mathpzc{P}$ considered as a discrete
$\mathpzc{Gray}$-category.  The conditions are that the domain
$\mathpzc{P}$ is a small and that the codomain
$\mathpzc{L}$ is a cocomplete $\mathpzc{Gray}$-category. 
In fact, we prove that the $\mathpzc{Gray}$-categories 
$\mathrm{Ps}\text{-}T\text{-}\mathrm{Alg}$ and
$\mathpzc{Tricat}_{\mathrm{ls}}(\mathpzc{P},\mathpzc{L})$ are isomorphic as
$\mathpzc{Gray}$-categories, which extends the local result mentioned in
\cite[Ex. 3.5]{powerthreedimensional}.
 On the other hand, the Eilenberg-Moore object
 $[\mathrm{ob}\mathpzc{P},\mathpzc{L}]^T$ for this monad 
 is given by the functor
$\mathpzc{Gray}$-category $[\mathpzc{P},\mathpzc{L}]$, and there is an
obvious inclusion of $[\mathpzc{P},\mathpzc{L}]$ into
$\mathrm{Ps}\text{-}T\text{-}\mathrm{Alg}$.
The relation of these two $\mathpzc{Gray}$-categories was studied 
locally by Power \cite{powerthreedimensional} and by Gurski using codescent objects
\cite{gurskicoherencein}. This mimics the situation one categorical
dimension below, which started with Blackwell et al.'s paper
\cite{blackwell} and was later refined by Lack using codescent objects
\cite{lackcodescent2002}.
We readily show that a corollary of Gurski's central coherence theorem
\cite[Th. 15.13]{gurskicoherencein} applies to the
$\mathpzc{Gray}$-monad on $[\mathrm{ob}\mathpzc{P},\mathpzc{L}]$:
The inclusion of the Eilenberg-Moore object
$[\mathrm{ob}\mathpzc{P},\mathpzc{L}]^T$ into the
$\mathpzc{Gray}$-category $\mathrm{Ps}\text{-}T\text{-}\mathrm{Alg}$
of pseudo $T$-algebras has a left adjoint such that the components of
the unit of this adjunction are internal biequivalences.

This corresponds to the fact from two-dimensional monad theory that
given a small $2$-category $P$ and a cocomplete $2$-category $L$, the
inclusion of the functor category $[P,L]$ (in the sense of
$\mathpzc{Cat}$-enriched category theory) into the $2$-category $\mathpzc{Bicat}(P,L)$ of functors, pseudonatural
transformations and modifications has a left adjoint such that the
components of the unit of this adjunction are internal equivalences.
While this is not explicitly stated, it follows from Lack's coherence
theorem \cite[Th. 4.10]{lackcodescent2002}. Explicit partial results may be found in 
\cite[Ex. 4.2]{powergeneral} and  
\cite[Ex. 6.6]{blackwell}.

We now go on to expand on the monad $T$, its properties, and the
identification of $\mathrm{Ps}\text{-}T\text{-}\mathrm{Alg}$.
For this purpose we will need the unique $\mathpzc{Gray}$-functor
$H\co\mathrm{ob}\mathpzc{P}\rightarrow \mathpzc{P}$ which is the
identity on objects, and the $\mathpzc{Gray}$-functor $[H,1]\co
[\mathpzc{P},\mathpzc{L}]\rightarrow
[\mathrm{ob}\mathpzc{P},\mathpzc{L}]$ from enriched category theory,
which is given on objects by
precomposition with $H$. This functor sends any
cell of $[\mathpzc{P},\mathpzc{L}]$ such as a $\mathpzc{Gray}$-functor or a $\mathpzc{Gray}$-natural
transformation to its family of values and components in $\mathpzc{L}$ respectively.
By
the theorem of Kan adjoints, left Kan extension $\mathrm{Lan}_H$ along
$H$ provides a left adjoint to $[H,1]$, and $T$ is the
$\mathpzc{Gray}$-monad corresponding to this adjunction.
This is all as in the $2$-dimensional context, and the story is then usually told as follows: The enriched Beck's
monadicity theorem shows that $[H,1]\co [\mathpzc{P},\mathpzc{L}]\rightarrow
[\mathrm{ob}\mathpzc{P},\mathpzc{L}]$ is strictly monadic. That is,
the Eilenberg-Moore object $[\mathrm{ob}\mathpzc{P},\mathpzc{L}]^T$ is
isomorphic to the functor category $[\mathpzc{P},\mathpzc{L}]$ such that the forgetful
functor factorizes through this isomorphism and $[H,1]$.
On the other hand, the monad has an obvious explicit description.
In fact, by the description of the left Kan extension in terms of
tensor products and coends we must have:
\begin{align}\label{introexlicitLan}
  (TA)Q=
 \int^{P\in\mathrm{ob}\mathpzc{P}}\mathpzc{P}(P,Q)\otimes AP
\puncteq{}
\end{align}
where $A$ is a $\mathpzc{Gray}$-functor
$\mathrm{ob}\mathpzc{P}\rightarrow \mathpzc{L}$ and where
$Q$ is an object of $\mathpzc{P}$.
Recall that the tensor product is a special indexed colimit. For enrichment in a general symmetric monoidal closed category $\mathpzc{V}$, it is
characterized by an appropriately natural
isomorphism in $\mathpzc{V}$:
\begin{align}\label{introtensor}
  \mathpzc{L}(\mathpzc{P}(P,Q)\otimes AP,AQ)\cong
  [\mathpzc{P}(P,Q),\mathpzc{L}(AP,AQ)]
\puncteq{,}
\end{align}
where $[-,-]$ denotes the internal hom of $\mathpzc{V}$. Thus, in the
case of enrichment in $\mathpzc{Gray}$, \eqref{introtensor} is an
isomorphism of $2$-categories. In fact, the tensor product gives rise
to a $\mathpzc{Gray}$-adjunction, and its
hom $\mathpzc{Gray}$-adjunction is given by
\eqref{introtensor}.

To achieve the promised identification of
$\mathrm{Ps}\text{-}T\text{-}\mathrm{Alg}$ with
$\mathpzc{Tricat}_{\mathrm{ls}}(\mathpzc{P},\mathpzc{L})$, we have to determine
how the data and $\mathpzc{Gray}$-category structure of
$\mathrm{Ps}\text{-}T\text{-}\mathrm{Alg}$ transforms under the
adjunction of the tensor product. Indeed, one can also identify the
Eilenberg-Moore object with the functor $\mathpzc{Gray}$-category in
this fashion.
For example, an object of $[\mathrm{ob}\mathpzc{P},\mathpzc{L}]^T$ is
an algebra for the monad $T$. The definition of such an algebra is just as
for an ordinary monad. Thus, it consists of a $1$-cell $a\co TA\rightarrow
A$ subject to two algebra axioms. According to equation
\eqref{introexlicitLan}, the $\mathpzc{Gray}$-natural transformation $a$ is determined by components \mbox{$a_{PQ}\co \mathpzc{P}(P,Q)\otimes AP\rightarrow
AQ$}. These are objects in the $2$-category
\mbox{$\mathpzc{L}(\mathpzc{P}(P,Q)\otimes AP,AQ)$} . The internal hom of
$\mathpzc{Gray}$ is given by the $2$-category of strict functors, pseudonatural
transformations, and modifications.
Thus $a_{PQ}$ corresponds under the hom adjunction \eqref{introtensor} to a strict functor \mbox{$A_{PQ}\co
\mathpzc{P}(P,Q)\rightarrow \mathpzc{L}(AP,AQ)$}, and the axioms of an algebra imply that this gives $A$ the structure of a
$\mathpzc{Gray}$-functor $\mathpzc{P}\rightarrow \mathpzc{L}$.

Given a $\mathpzc{Gray}$-monad on a $\mathpzc{Gray}$-category
$\mathpzc{K}$, the notions of pseudo algebras, pseudo functors, 
transformations, and modifications are all given by cell data
of the $\mathpzc{Gray}$-category $\mathpzc{K}$. In the case that
$\mathpzc{K}=[\mathrm{ob}\mathpzc{P},\mathpzc{L}]$, this means that the
data consists of families of cells in the target $\mathpzc{L}$. Parts
of these data transform under
the adjunction of the tensor product into families of cells in
the internal hom, that is, families of strict functors of
$2$-categories, pseudonatural transformations of those, and
modifications of those.
This already shows that we only have a chance to recover locally
strict trihomomorphisms from pseudo algebras because a general
trihomomorphism might consist of nonstrict functors of
$2$-categories. This is in contrast to the two-dimensional context
where pseudo algebras correspond precisely to possibly nonstrict
functors of $2$-categories.

We now give a short overview of how this paper is organized.
In Section \ref{preliminaries}, we describe the symmetric monoidal closed
category $\mathpzc{Gray}$ and extend some elementary results on the correspondence of 
cubical functors and strict functors on Gray products.

In Section \ref{pseudoalgebras}, we reproduce Gurski's definition of
$\mathrm{Ps}\text{-}T\text{-}\mathrm{Alg}$ and prove that two lax
algebra axioms are redundant for a pseudo algebra.

In Section \ref{sectionmonad}, we introduce the monad $T$ on
$[\mathrm{ob}\mathpzc{P},\mathpzc{L}]$ in \ref{subsectionmonad} and
describe it explicitly in \ref{subsectionexplicit}. In
\ref{someproperties} we expand on tensor products and derive the two
rather involved Lemmata \ref{asabovebynatofML} and \ref{noticethesymmetry},
which play a critical role in the bulk of our technical calculations.

In Section \ref{sectionstrict}, we explicitly identify the Eilenberg-Moore object
$[\mathrm{ob}\mathpzc{P},\mathpzc{L}]^T$ in the general situation
where $\mathpzc{V}$ is a complete and cocomplete locally small
symmetric monoidal closed category.

In Section \ref{sectionGrayhom}, we establish the identification of
$\mathrm{Ps}\text{-}T\text{-}\mathrm{Alg}$ and
$\mathpzc{Tricat}_{\mathrm{ls}}(\mathpzc{P},\mathpzc{L})$, on
which we will now comment in more detail.
To characterize how  $\mathrm{Ps}\text{-}T\text{-}\mathrm{Alg}$ transforms under the adjunction of the tensor
product, in \ref{homomorphismsofgraycategories} we introduce the
notion of homomorphisms of $\mathpzc{Gray}$-categories, Gray transformations,
Gray modifications, and Gray perturbations. With the help of Lemmata
\ref{asabovebynatofML} and \ref{noticethesymmetry} from \secref{noticethesymmetry}, these are seen to
be exactly the transforms of pseudo algebras, pseudo functors, 
transformations, and modifications respectively. This also equips the
Gray data with the structure of a $\mathpzc{Gray}$-category.

Elementary observations in \ref{subsectionls} then give that a homomorphism of $\mathpzc{Gray}$-categories is the same thing as a locally strict trihomomorphism,
much like a $\mathpzc{Gray}$-category is the same thing as a strict,
cubical tricategory. Similarly, the notion of a Gray transformation
corresponds exactly to a tritransformation between locally strict
functors, and Gray modifications and perturbations correspond exactly
to trimodifications and perturbations of those. 
This follows from the general correspondence, mediated by Theorem
\ref{internal2isobicatcmulti} from \secref{internal2isobicatcmulti}, of data for the
cubical composition functor and the cartesian product on the hand and data for
the composition law of the $\mathpzc{Gray}$-category and the Gray
product on the other hand.
The only thing left to check is that the axioms correspond to
each other. Namely, the Gray notions being the transforms of the
pseudo notions of three-dimensional monad theory, the axioms are
equations of modifications, while the axioms for the tricategorical
constructions are equations involving the components of modifications. That these
coincide is mostly straightforward, less transparent is only the
comparison of interchange cells.
Gurski's coherence theorem then gives
that the inclusion of the full sub-$\mathpzc{Gray}$-category of
$\mathpzc{Tricat}(\mathpzc{P},\mathpzc{L})$ determined by the locally
strict functors, denoted by
$\mathpzc{Tricat}_{\mathrm{ls}}(\mathpzc{P},\mathpzc{L})$, into the
functor $\mathpzc{Gray}$-category $[\mathpzc{P},\mathpzc{L}]$ has a
left adjoint and the components of the unit of this adjunction are
internal biequivalences. 

\section*{Acknowledgements} 
The author thanks Christoph Schweigert for his constant encouragement
and help with the draft
and Nick Gurski for very valuable discussions and comments on the draft.
Thanks also to Richard Garner for discussions on an extended version of this paper. Support by the Research Training Group 1670
''Mathematics inspired by string theory and quantum field theory`` is
gratefully acknowledged.

\section{Preliminaries}\label{preliminaries}
We assume familiarity with enriched category theory. Regarding enriched category theory, we stay
notationally close to Kelly's book \cite{kelly}, from which we shall
cite freely. We also assume a fair amount of bicategory theory, see
for example \cite{benabou67} or the short
\cite{leinsterbasic98}. Since the notions of tricategory theory are
intrinsically involved, and since we consider various slight variations of
these, we refrain from supplying all of them. The appropriate references are
the original paper by Gordon, Power, and Street
\cite{gordonpowerstreet}, Gurski's thesis \cite{gurskialgebraic}, and
his later book \cite{gurskicoherencein}, which includes much of the
material first presented in \cite{gurskialgebraic}. Since this is also
our primary reference for three-dimensional monad theory, we will
usually cite from \cite{gurskicoherencein}.
 In fact, the tricategories considered are all $\mathpzc{Gray}$-categories, and we will describe them in terms of enriched
notions to the extent possible. We do supply definitions in terms of enriched notions that correspond precisely  to locally strict trihomomorphisms,
tritransformations, trimodifications, and perturbations, but in describing this correspondence we assume knowledge of
the tricategorical definitions. As a matter of fact the sheer amount of
notational
translation between the $\mathpzc{Gray}$-enriched and the
tricategorical context can be challenging at times.

Only basic knowledge of the general theory of
monads in a $2$-category is required cf. \cite{streetformal72}. For
monads in enriched category theory see also \cite{dubuc}.
\subsection{Conventions}
Horizontal composition in a bicategory is generally denoted by the
symbol $\ast$, while vertical
composition is denoted by the symbol $\diamond$.
We use the term functor for what is elsewhere called  pseudofunctor or
weak functor or homomorphism of bicategories and shall indicate
whether the functor is strict where it is not clear from context. By
an isomorphism we always mean an honest isomorphism, e.g. an
isomorphism on objects and hom objects in enriched category theory.
The symbol $\otimes$ is reserved both for a monoidal structure and
tensor products in the sense of enriched category theory. If not
otherwise stated, $\mathpzc{V}$ denotes a locally small symmetric
monoidal closed category with monoidal structure $\otimes$;
associators and unitors $a,l,$ and $r$;
internal hom $[-,-]$; unit $d$ and counit or evaluation $e$.
We shall usually use the prefix $\mathpzc{V}$- to emphasize when the
$\mathpzc{V}$-enriched notions are meant, although this is occasionally
dropped
where it would otherwise seem overly redundant. The composition law of
a $\mathpzc{V}$-category $\mathpzc{K}$ is denoted by $M_{\mathpzc{K}}$. The unit
at the object $K\in \mathpzc{K}$ is denoted by $j_K$ or occasionally
$1_K$, for example when it shall be emphasized that it is also the
unit at $K$ in the underlying category $\mathpzc{K}_0$. The
identification of $\mathpzc{V}_0(I,[X,Y])$ and $\mathpzc{V}_0(X,Y)$ induced from the closed structure for
objects $X,Y\in\mathpzc{V}$
 of $\mathpzc{V}$ is to
be understood and usually implicit.
 We use the terms
indexed limit and indexed colimit for what is elsewhere also called
weighted limit and weighted colimit. The concepts of
ordinary and extraordinary $\mathpzc{V}$-naturality
cf. \cite[Ch. 1]{kelly} and the corresponding composition calculus are to be understood, and we freely use the
underlying ordinary and extraordinary naturality too. 

 Composition in a monoidal category is generally denoted by
juxtaposition.
Composition of $\mathpzc{V}$-functors is in general also denoted
by juxtaposition.
For cells of a $\mathpzc{Gray}$-category, juxtaposition is used as
shorthand for the application of its composition law.

\subsection{The category $\mathpzc{2Cat}$}\label{2cat}
Let $\mathpzc{Cat}$ denote the category of small categories and
functors. It is well-known that $\mathpzc{Cat}$ is complete and
cocomplete: it clearly has products and equalizers, thus is
complete. Coproducts are given by disjoint union, and there is a
construction for coequalizers in 
\cite[I,1.3, p. 25]{gray} (due to Wolff).
In fact, the same strategy applies to the category of
$\mathpzc{V}$-enriched categories and \mbox{$\mathpzc{V}$-functors} in
general, where $\mathpzc{V}$ is a complete and cocomplete symmetric
monoidal closed category. 
 Products are given by the cartesian product
of the object sets and the cartesian product of the hom
objects. Equalizers of $\mathpzc{V}$-functors are given by the
equalizer of the maps on objects and the equalizer of the hom
morphisms.
Coproducts are given by the coproduct of the object sets i.e. the
disjoint union and by the coproduct of the hom objects.
The construction of coequalizers in  \cite[I,1.3, p. 25]{gray} 
can easily be transferred to this context.
In particular, the category  $\mathpzc{2Cat}$ of small $2$-categories
and strict functors is  complete and cocomplete.

\subsection{The symmetric monoidal closed category $\mathpzc{Gray}$}
We now describe the symmetric monoidal closed category in which we will usually enrich. Its underlying category is $\mathpzc{2Cat}$, which has a symmetric monoidal closed structure given by the Gray product. We can only provide a brief description of the Gray product here. For
 details, the
reader is referred to \cite[3.1, p. 36ff.]{gurskicoherencein} and to
\cite[I,4.9, p. 73ff.]{gray} for a lax variant.

 The Gray product  of $2$-categories $X$
and $Y$ is a $2$-category denoted by
$X\otimes Y$.
Rather than giving a complete explicit description of the Gray
product, we mention the following characterization (see
\cite[Ch. 3]{gurskicoherencein}).
Considering the sets of objects $\mathrm{ob}X$ and $\mathrm{ob}Y$ as discrete
$2$-categories, we denote by  $X\oblong Y$ the pushout  in
$2\mathpzc{Cat}$ of the diagram below, where $\times$ denotes the cartesian product of $2$-categories, and
where the morphisms are given by products of the inclusions $\mathrm{ob}
X\rightarrow X$ and $\mathrm{ob}Y\rightarrow Y$ and identity functors respectively.
  \begin{equation}\label{funnytensor}
    \begin{tikzpicture}
      \matrix (a) [matrix of math nodes, row sep=4em, column sep=6em,
      text height=1.5ex, text depth=0.25ex]{  \mathrm{ob}X\times
  \mathrm{ob}Y &  X\times \mathrm{ob}Y \\
        \mathrm{ob}X\times Y  & \\};
      
 \path[->] (a-1-1) edge
      node[right]{$\scriptstyle  $} (a-2-1); 
\path[->] (a-1-1)
      edge node[above]{$\scriptstyle$}
      (a-1-2); 
\puncttikz[a-2-1]{}
    \end{tikzpicture}
  \end{equation}
By the
universal property of the pushout, the products of the inclusions and
identity functors, $\mathrm{ob} X\times
Y\rightarrow X\times Y$ and
$X\times \mathrm{ob}Y\rightarrow X\times Y$, induce a strict functor
$j\co X\oblong Y\rightarrow X\times Y$.

It is well-known that there is an orthogonal
  factorization system on $\mathpzc{2Cat}$ with left class the
  strict functors which are bijective on objects and $1$-cells and
  right class the strict functors which are locally fully faithful,
  see for example\cite[Corr. 3.20, p. 51]{gurskicoherencein}.
The Gray product $X\otimes Y$ may be characterized by factorizing $j$ with respect
to this factorization system. More precisely, $X\otimes Y$ is uniquely characterized (up to unique isomorphism
in $\mathpzc{2Cat}$)
by the fact that there is a 
strict functor $m\co X\oblong Y \rightarrow X\otimes Y$ which is an isomorphism on the underlying categories i.e. bijective on objects
and $1$-cells and a strict functor $i\co X\otimes Y\rightarrow X\times
Y$ which is locally fully faithful such that $j= im$.

There is an obvious explicit description of $X\otimes Y$ in terms of
generators and relations, which can be used to construct a functor $\otimes \co \mathpzc{2Cat}\times
\mathpzc{2Cat}\rightarrow \mathpzc{2Cat}$.
Clearly, $X\otimes Y$ has
the same objects as $X\times Y$, and we have the images of the
$1$-cells and $2$-cells from
$X\times \mathrm{ob}Y$ and $\mathrm{ob}X\times Y$, for which we use
the same name in $X\otimes Y$. That is, there are $1$-cells $(f,1)\co (A,B)\rightarrow (A',B)$
for $1$-cells $f\co A\rightarrow A'$ in $X$ and objects $B$ in
$Y$, and there are $1$-cells $(1,g)\co (A,B)\rightarrow (A,B')$ for objects $A$ in $X$
and $1$-cells $g\co B\rightarrow B'$ in $Y$. All $1$-cells in
$X\otimes Y$ are up to the obvious relations generated by horizontal
strings of those $1$-cells, the identity $1$-cell being $(1,1)$.
Apart from the obvious $2$-cells $(\alpha,1)\co (f,1)\Rightarrow (f',1)\co
(A,B)\rightarrow (A',B)$ and $(1,\beta)\co (1,g)\Rightarrow (1,g')\co
(A,B)\rightarrow (A,B')$, there must be unique invertible interchange $2$-cells 
$\Sigma_{f,g}\co (f,1)\ast(1,g)\Rightarrow (1,g)\ast(f,1)$ mapping to the
identity of $(f,g)$ under $j$ because the latter is fully faithful---domain and codomain cleary both map to $(f,g)$ under $j$. In
particular, by uniqueness 
i.e. because $j$ is locally fully faithful, these must be the
identity if either $f$ or $g$ is the identity.
There are various relations on horizontal and vertical composites of
those cells, all rather obvious from the characterization above. We omit
those as well as the details how equivalence classes and horizontal
and vertical composition are defined.

For functors $F\co X\rightarrow X'$ and $G\co Y\rightarrow Y'$, it is
not hard to give a functorial definition of the functor $F\otimes G\co X\otimes
Y\rightarrow X'\otimes Y'$. We confine ourselves with the
observation that on interchange $2$-cells,
\begin{align}\label{FGSigma}
  (F\otimes
  G)_{(A,B),(A'B')}(\Sigma_{f,g})=\Sigma_{F_{A,A'}(f),G_{B,B'}(g)}
\puncteq{.}
\end{align}

From the characterization above it is then clear how to define
associators and unitors for $\otimes$. We only mention here that
\begin{align}
  a(\Sigma_{f,g},1)=\Sigma_{f,(g,1)}\label{aSigmafg1}
\end{align}
and
\begin{align}
  a(\Sigma_{(f,1),h})=\Sigma_{f,(1,h)}\label{aSigmaf1h}
\end{align}
and
\begin{align}\label{aSigma1gh}
  a(\Sigma_{(1,g),h})=(1,\Sigma_{g,h})
\puncteq{.}
\end{align}

We omit the details that this gives a monoidal
structure on $\mathpzc{2Cat}$ (pentagon and triangle identity follow
from pentagon and triangle identity for the cartesian product and
$\oblong$).

There is an obvious symmetry $c$ for the Gray product, which on
interchange cells is given by
\begin{align}\label{cSigma}
  c(\Sigma_{f,g})=\Sigma^{-1}_{g,f}
\puncteq{.}
\end{align}

As for any two bicategories (by all means for small domain), there is a functor
bicategory $\mathpzc{Bicat}(X,Y)$
given by functors of bicategories, pseudonatural
transformations, and modifications. As the codomain $Y$ is a $2$-category,
this is in fact again a $2$-category. We denote by $[X,Y]$ the full
sub-$2$-category of $\mathpzc{Bicat}(X,Y)$ given by the strict functors.
One can show that this gives $\mathpzc{2Cat}$ the structure of
a symmetric monoidal closed category with internal hom $[X,Y]$:

\begin{theorem}\cite[Th. 3.16]{gurskicoherencein}
  The category $\mathpzc{2Cat}$ of small $2$-categories and strict functors
  has the structure of a symmetric monoidal closed category. As such,
  it is referred
  to as $\mathpzc{Gray}$. The monoidal structure is given by the Gray
  product and the terminal $2$-category as the unit object, the
  internal hom is given by the functor $2$-category of strict
  functors, pseudonatural transformations, and modifications.
\end{theorem}
\begin{remark}
  In fact, we will not have to specify the closed structure of
  $\mathpzc{Gray}$ apart from the fact that its evaluation is (partly)
  given by taking components. 
  This is because our ultimate goal is to compare definitions from
  three-dimensional i.e. $\mathpzc{Gray}$-enriched monad theory to
  definitions from the theory of tricategories, and we do so in the case where all
  tricategories are in fact $\mathpzc{Gray}$-categories, that is,
  equivalently, strict, cubical tricategories. These definitions
  will only formally involve the cubical composition functor, which
  relates to the composition law of the $\mathpzc{Gray}$-category --
  we will usually not have to specify the composition. Of course,
  one can explicitly identify the enriched notions, and then there are
  alternative explicit arguments. 
 However, we think that the formal
  argumentation is more adequate.
The closed structure is worked out in \cite[3.3]{gurskicoherencein},
and the enriched notions usually turn out to be just as one would
expect.
We spell out a few explicit prescriptions below the following lemma,
but in fact we just need a few consequences of these, for example equation
\eqref{TenSigma} below. 

\end{remark}

Next recall that a locally small symmetric monoidal closed category $\mathpzc{V}$ can be considered as
 a category enriched in itself i.e. as a $\mathpzc{V}$-category. Also
 recall that if the underlying category $\mathpzc{V}_0$ of $\mathpzc{V}$ is complete
 and cocomplete, $\mathpzc{V}$ is complete and cocomplete considered as
 a $\mathpzc{V}$-category. This means it has any small indexed limit
 and any small indexed colimit. For the concept of an indexed limit
 see \cite[Ch. 3]{kelly}. In fact, completeness follows from the fact
 that a limit is given by an end, and if the limit is small, this end
 exists and is given by an equalizer in $\mathpzc{\mathpzc{V}_0}$, see
 \cite[(2.2)]{kelly}.
 It is
 cocomplete because, $\mathpzc{V}$ being complete,
 $\mathpzc{V}\op$ is tensored and thus also admits small conical limits
 because $\mathpzc{V}_0$ is cocomplete, hence $\mathpzc{V}$ admits small coends because it is also tensored,  but then since by
 \cite[(3.70)]{kelly} any small colimit is given by a small coend over
 tensor products, it is cocomplete.

Recall that the underlying category $\mathpzc{2Cat}$ of $\mathpzc{Gray}$ is complete and cocomplete cf. \ref{2cat}.
Thus in particular, we have the following
:
 \begin{lemma}\label{graycocomplete}
   The $\mathpzc{Gray}$-category $\mathpzc{Gray}$ is complete and
   cocomplete. \qed
 \end{lemma}

The composition law of the $\mathpzc{Gray}$-category $\mathpzc{Gray}$ is given by strict functors $[Y,Z]\otimes [X,Y]\rightarrow
[X,Y]$, where $X,Y$ and $Z$ are $2$-categories . It is given on objects by 
composition of strict functors. On $1$-cells of the form $(\theta,1)\co
(F,G)\rightarrow (F',G)$  it is given by the pseudonatural
transformation denoted $G^\ast\theta $ with components $\theta_{Gx}$
and naturality $2$-cells $\theta_{Gf}$. On $1$-cells of the form
$(1,\sigma)\co (F,G)\rightarrow (F,G')$ it is given by the
pseudonatural transformation denoted $F_\ast \sigma$ with components
$F_{Gx,G'x}(\sigma_x)$ and naturality $2$-cells
$F_{Gx,G'x'}(\sigma_f)$. Similarly, on $2$-cells of the form
$(\Gamma,1)\co (\theta,1)\Rightarrow (\theta',1)\co (F,G)\rightarrow
(F',G)$ it is given by the modification denoted $G^\ast\Gamma$ with
components $\Gamma_{Gx}$, and on $2$-cells of the form $(1,\Delta)\co
(1,\sigma)\Rightarrow (1,\sigma')\co (F,G)\rightarrow (F,G')$ it is
given by the modification denoted $F^\ast \Delta$ with components
$F_{Gx,G'x}(\Delta_X)$.
Finally, on interchange cells of the form $\Sigma_{\theta,\sigma}$, it
is given by the naturality $2$-cell $\theta_{\sigma_x}$ of $\theta$ at
$\sigma_x$, hence,
\begin{align}\label{MGraySigma}
  (M_{\mathpzc{Gray}}(\Sigma_{\theta,\sigma}))_x= \theta_{\sigma_x}  \co \theta_{G'x}\ast
  F_{Gx,G'x}(\sigma_x)\Rightarrow F'_{Gx,G'x}(\sigma_x)\ast \theta_{Gx}
\puncteq{.}
\end{align}
This follows from the general form of $M_{\mathpzc{V}}$ in enriched
category theory by inspection of the closed structure of $\mathpzc{Gray}$
cf. \cite[Prop. 3.10]{gurskicoherencein}.

Also recall that there  is a functor $\mathrm{Ten}\co \mathpzc{V}\otimes
\mathpzc{V}\rightarrow \mathpzc{V}$ which is given on objects by the
monoidal structure.
For $\mathpzc{V}=\mathpzc{Gray}$, its strict hom functor
\begin{align*}
  \mathrm{Ten}_{(X,X'),(Y,Y')}\co [X,X']\otimes [Y,Y']\rightarrow
  [X\otimes Y,X'\otimes Y']
\end{align*}
sends an object $(F,G)$ to the functor $F\otimes G$. It sends a transformation $(\theta,1_G)\co
(F,G)\Rightarrow (F',G)$ to the
transformation with component the $1$-cell $(\theta_x,1_{Gy})$ in 
$X'\otimes Y'$ at the object $(x,y)$ in  $X\otimes
Y$; and naturality $2$-cells $(\sigma_f,1_{Gg})$ and $\Sigma_{\theta_{x},Gg}$ at $1$-cells $(f,1_y)$ and $(1_x,g)$ respectively.
Its effect on a transformation $(1_F,\iota)\co (F,G)\Rightarrow
(F,G')$ is analogous. It sends a
modification $(\Gamma,1_{1_G})\co (\theta,1_G)\Rrightarrow
(\theta',1_G)$ to the modification with component the $2$-cell
$(\Gamma_X,1_{1_{Gy}})$ in $X'\otimes Y'$ at $(x,y)$ in $X\otimes
Y$. Its effect on a modification $(1_{1_F},\Delta)\co (1_F,\iota)\Rrightarrow
(1_F,\iota')$ is analogous. Finally, it sends the interchange $2$-cell
$\Sigma_{\theta,\iota}$ to the modification with component the
interchange $2$-cell $\Sigma_{\theta_x,\iota_y}$, hence,
\begin{align}\label{TenSigma}
  (\mathrm{Ten}_{(X,X'),(Y,Y')}(\Sigma_{\theta,\iota}))_{x,y}=\Sigma_{\theta_x,\iota_y}
\puncteq{.}
\end{align}
All of this again follows from
inspection of the closed structure of $\mathpzc{Gray}$,
cf. \cite[Prop. 3.10]{gurskicoherencein}. See also equation \eqref{Tenhom} below.

\subsection{Cubical functors}\label{subsectioncubical}

Given $2$-categories $X,Y,Z$, recall that a cubical functor in two
variables is a functor $\hat F\co
X\times Y\rightarrow Z$ 
such that for all $1$-cells
$(f,g)$ in $X\times Y$, the composition constraint
\begin{align*}
  \hat F_{(1,g),(f,1)}\co  \hat F(1,g)\ast \hat F(f,1)\Rightarrow
  \hat F(f,g)
\puncteq{,}
\end{align*}
is the identity $2$-cell,
and such that for all composable 
$1$-cells $(f',1)$, $(f,1)$ in $X\times Y$, 
\begin{align*}
  \hat F_{(f',1),(f,1)}\co \hat F(f',1)\ast \hat
  F(f,1)\Rightarrow \hat F(f'\ast f,1)
\puncteq{,}
\end{align*}
is the identity $2$-cell, 
and such that for all composable $1$-cells  $(1,g')$, $(1,g)$ in
$X\times Y$,
\begin{align*}
  \hat F_{(1,g'),(1,g)}\co \hat F(1,g')\ast \hat F(1,g)\Rightarrow  \hat F(1,g'\ast g)
\puncteq{,}
\end{align*}
is the identity $2$-cell.  For composable $(1,g'),(f,g)$ and $(f',g'),(f,1)$, the
constraint cells are then automatically identities by compatibility of
$\hat F$
with associators i.e. a functor axiom for $\hat F$; it also automatically
preserves identity $1$-cells.

We start with the following elementary result, which extends the
natural $\mathpzc{Set}$-isomorphism in
\cite[Th. 3.7]{gurskicoherencein} to a $\mathpzc{Cat}$-isomorphism.
\begin{proposition}\label{internal2isobicatc}
Given $2$-categories $X,Y,Z$, there is a universal cubical functor $C\co
X\times Y\rightarrow X\otimes Y$ natural in $X$ and $Y$ such that
precomposition with $C$ induces a natural isomorphism of $2$-categories
(i.e. a $\mathpzc{Cat}$-isomorphism)
\begin{align*}
 [X\otimes Y, Z]\cong \mathpzc{Bicat}_c(X,Y;Z)
\puncteq{,}
\end{align*}
where $\mathpzc{Bicat}_c(X,Y;Z)$ denotes the full sub-$2$-category of
$\mathpzc{Bicat}(X\times Y,Z)$ determined by the cubical functors.
\end{proposition}
\begin{proof}
The functor $C$ is determined by the requirements that it be the identity
on objects, that $C(f,1)=(f,1)$, $C(\alpha,1)=(\alpha,1)$,
$C(1,g)=(1,g)$, $C(1,\beta)=(1,\beta)$, and that it be a cubical
functor.
In particular, observe that this means that $C(f,g)=(1,g)\ast(f,1)$
and that  the constraint $C_{(f,1),(1,g)}$ is given by the interchange
cell $\Sigma_{f,g}$.

As for an arbitrary functor of bicategories, precomposition with $C$ induces a strict functor
\begin{align*}
C^\ast\co \mathpzc{Bicat}(X\otimes Y,Z)\rightarrow \mathpzc{Bicat}(X\times Y,Z)
\puncteq{.}
\end{align*}
It sends a functor $G\co X\otimes Y\rightarrow Z$ to the composite functor $GC\co X\times Y\rightarrow C$. In fact, if $F$ is a strict functor $X\otimes Y\rightarrow Z$,
recalling the
  definition of the composite of two functors of bicategories, a
  moment's reflection affirms that $\hat F\coloneqq F C$ is a
  cubical functor with constraint $\hat F_{(f,1),(1,g)}=F(\Sigma_{f,g})$. Thus by restriction, $C^\ast$ gives
  rise to  a functor $ [X\otimes Y, Z]\rightarrow
  \mathpzc{Bicat}_c(X,Y;Z)$ which we also denote by
  $C^\ast$.

If $\sigma\co F\Rightarrow G\co X\otimes Y\rightarrow Z$ is a pseudonatural transformation,  $ C^\ast \sigma\co FC\Rightarrow GC$ is the pseudonatural transformation with component
\begin{align*}
  (C^\ast \sigma)_{(A,B)}=\sigma_{C(A,B)}=\sigma_{(A,B)}
\end{align*}
 at an object $(A,B)\in X\times Y$, and naturality $2$-cell
\begin{align*}
  (C^\ast \sigma)_{(f,g)}=\sigma_{C(f,g)}=\sigma_{(f,1)\ast
    (1,g)}=(\sigma_{(f,1)}\ast 1)\diamond(1\ast\sigma_{(1,g)})
\end{align*}
 at  a $1$-cell $(f,g)\in
X\times Y$, where the last equation is by respect for composition of $\sigma$. 
If $\sigma$ is the identity pseudonatural transformation, it is
immediate that the same applies to $C^\ast \sigma$.
Given another
pseudonatural transformation of strict functors $\tau\co G\Rightarrow
H$, we maintain that \mbox{$(C^\ast \tau) \ast (C^\ast \sigma) = C^\ast(\tau\ast
\sigma)$}. It is manifest that the components coincide: both are given by
$\tau_{(A,B)}\ast\sigma_{(A,B)}$ at the object $(A,B)\in X\times
Y$. That the naturality $2$-cells at a $1$-cell $(f,g)\in X\times Y$
coincide,
\begin{align*}
  (\tau_{C(f,g)}\ast 1)\diamond (1\ast \sigma_{C(f,g)})=(\tau\ast \sigma)_{C(f,g)}
\end{align*}
is simply the defining equation for the naturality $2$-cell of the
horizontal composite $\tau\ast \sigma$. 

If $\Delta\co \sigma\Rrightarrow \pi$ is a modification of
pseudonatural transformations $F\Rightarrow G$ of strict functors $X\otimes
Y\rightarrow Z$, then $ C^\ast \Delta $ is 
the modification
$ C^\ast\sigma\Rrightarrow C^\ast \pi $ with component 
\begin{align*}
  (C^\ast \Delta)_{(A,B)}=\Delta_{C(A,B)}=\Delta_{(A,B)}
\end{align*}
at  an object $(A,B)\in X\times Y$,
and this prescription clearly strictly
preserves identities and vertical composition of modifications.
Given another modification $\Lambda\co \tau\Rrightarrow \rho\co
G\Rightarrow H$ where $H$ is strict, one
readily checks that $(C^\ast\Lambda)\ast (C^\ast
\Delta)=C^\ast(\Lambda\ast \Delta)$ both having component
$\Lambda_{(A,B)}\ast \Delta_{(A,B)}$ at an object $(A,B)\in X\times
Y$.
Thus, we have shown that $C^\ast$ is indeed a strict
functor.

As a side note, we remark that because we only consider $2$-categories, $C^\ast$ is the same as the functor $\mathpzc{Bicat}(C,Z)$ induced by the composition of the tricategory $\mathpzc{Tricat}$ of bicategories, functors, pseudonatural transformations, and modifications.

Let $\hat F\co X\times Y\rightarrow Z$ be an arbitrary cubical
functor, then the prescriptions $F(f,1)=\hat F(f,1)$,
$F(\alpha,1)=\hat F(\alpha,1)$, $F(1,g)=\hat
F(1,g)$, $F(1,\beta)=\hat F(1,\beta)$, and $F(\Sigma_{f,g})=\hat 
F_{(f,1),(1,g)}$, provide a strict functor $F\co X\otimes Y\rightarrow
Z$ such that $FC= \hat F$. The latter equation and the requirement
that $F$ be strict, clearly determine
$F$ uniquely. That this is well-defined e.g. that it
respects the various relations for the interchange cells is by compatibility of
$\hat F$ with associators and
naturality of
\begin{align}\label{functoriality,,}
  \hat F_{(A,B),(A',B'),(A'',B'')}\co \ast_Z (\hat
  F_{(A',B'),(A'',B'')}\times \hat F_{(A,B),(A',B')} )\Rightarrow
 \hat F_{(A,B),(A'',B'')}\ast_{X\times Y}
\puncteq{,}
\end{align}
where $\ast$ denotes the corresponding horizontal composition functors. For example, for the relation
\begin{align}\label{Sigmaf'fg}
  \Sigma_{f'\ast f,g}\sim (\Sigma_{f',g}\ast (1_f,1))\diamond
  ((1_{f'},1)\ast \Sigma_{f,g})
\end{align}
one has to use that axiom twice giving
\begin{align*}
  \hat F_{(f'\ast f,1),(1,g)}=\hat F_{(f',1)\ast(f,1),(1,g)}= \hat
  F_{(f',1),(f,g)} \diamond (\hat F(1_{f'},1)\ast \hat F_{(f,1),(1,g)})
\end{align*}
and $\hat F_{(f',1),(f,g)}=
\hat F_{(f',1),(1,g)\ast(f,1)}=\hat F_{(f',1),(g,1)}\ast \hat 
F(1_f,1)$. Another way to see this, is to use coherence for the
functor $\hat F$---then any relation in the Gray product must clearly
be mapped to an identity in $Z$ because the constraints in
$\mathpzc{F}_{\hat F}Z$ are mapped to
identities in $\mathpzc{F}_{2C}Z$, where these are the corresponding free constructions on the underlying category-enriched graphs cf. \cite[2.]{gurskicoherencein}.

Now let $\hat \sigma\co \hat F\Rightarrow \hat G$ be an arbitrary pseudonatural
transformation of cubical functors. We have already shown that $\hat F$ and $\hat G$ have the form $FC$ and $GC$ respectively, where $F$ and $G$ were determined above. We maintain that there is a unique pseudonatural transformation
$\sigma\co F\Rightarrow G$ such that $\hat \sigma =C^\ast \sigma$. By
the above, the latter equation uniquely determines both the components,
$\sigma_{(A,B)}=\hat\sigma_{(A,B)}$, and the
naturality $2$-cells of $\sigma$, namely
$\sigma_{(f,1)}=\hat\sigma_{(f,1)}$ and
$\sigma_{(1,g)}=\hat\sigma_{(1,g)}$, and thus $\sigma$ is uniquely
determined by respect for composition. That this is
 compatible with 
 the relations $(f',1)\ast (f,1)\sim (f'\ast f,1)$ and $(1,g')\ast
 (1,g)\sim (1,g'\ast g)$ in the Gray
 product follows from the fact that respect for composition is in
 this case tantamount to respect for composition of $\hat\sigma$
 because the constraints are identities here due to the axioms
 for cubical functors.
Hence, what
is left to prove is that 
this is indeed a pseudonatural transformation. First observe that the
prescriptions for $\sigma$ have been determined by the requirement
that it respects composition, and respect for units is tantamount to
respect for units of $\hat\sigma$.
 Naturality with respect to $2$-cells of the form $(\alpha,1)$ and
 $(1,\beta)$ is tantamount to the corresponding naturality condition
 for $\hat \sigma$. 
 Naturality with respect to an interchange cell
 $\Sigma_{f,g}$, i.e.
 \begin{align*}
(G(\Sigma_{f,g})  \ast 1_{\sigma_{(A,B)}})\diamond  \sigma_{(f,1)\ast (1,g)}= \sigma_{(1,g)\ast (f,1)}\diamond (1_{\sigma_{(A',B')}}\ast F(\Sigma_{f,g}))
 \end{align*}
is---by the requirement that $\sigma$ respects composition:
\begin{align*}
  \sigma_{(f,1)\ast (1,g)}=(1_{G(f,1)}\ast
  \sigma_{(1,g)})\diamond(\sigma_{(f,1)}\ast 1_{F(1,g)})= (1_{\hat G(f,1)}\ast
\hat\sigma_{(1,g)})\diamond(\hat\sigma_{(f,1)}\ast 1_{\hat F(1,g)})
\end{align*}
and by respect for composition of $\hat\sigma$:
\begin{align*}
  \sigma_{(1,g)\ast (f,1)}= (1_{G(1,g)}\ast
  \sigma_{(f,1)})\diamond(\sigma_{(1,g)}\ast 1_{F(f,1)})= (1_{\hat
    G(1,g)}\ast \hat\sigma_{(f,1)})\diamond(\hat\sigma_{(1,g)}\ast
  1_{\hat F(f,1)}) = \hat \sigma_{(f,g)}
\end{align*}
(the constraints are trivial here)---tantamount to respect for composition of $\hat \sigma$:
\begin{align*}
(\hat G_{(f,1),(1,g)}  \ast 1_{\hat \sigma_{(A,B)}})\diamond
((1_{\hat G(f,1)}\ast
\hat\sigma_{(1,g)})\diamond(\hat\sigma_{(f,1)}\ast 1_{\hat F(1,g)}))= \hat \sigma_{(f,g)}
\diamond (1_{\hat\sigma_{(A',B')}}\ast \hat F_{(f,1),(1,g)})
\puncteq{.}
 \end{align*}
Notice that in general, naturality with respect to a vertical composite is
implied by naturality with respect to the individual factors. Similarly,
naturality with respect to a horizontal composite is implied by functoriality
of $F_{,,}$ and $G_{,,}$ (cf. \eqref{functoriality,,}), respect for composition, and naturality
with respect to the individual factors. 

Finally, let $\hat \Delta \co C^\ast\sigma  \Rrightarrow C^\ast\pi \co FC\Rightarrow GC$
be an arbitrary modification. Then we maintain that there is a unique modification
$\Delta\co \sigma\Rightarrow \pi$ such that $\hat \Delta =C^\ast\Delta
$. By the above, the latter equation uniquely determines
$\Delta$'s components, $\Delta_{(A,B)}=\hat \Delta_{(A,B)}$ and thus $\Delta$ itself, but we have to show that $\Delta$ exists i.e. that this gives $\Delta$ the structure of a modification.  
The modification axiom for $1$-cells of the form $(f,1)$  is tantamount to the modification axiom for $\hat
\Delta$ and the corresponding $1$-cell in $X\times Y$ of the same
name. The same applies to the modification axiom for $1$-cells of the
form $(1,g)$. This proves that $\sigma$ is a modification because the
modification axiom for a horizontal composite is implied by  respect for composition of $\sigma$ and $\pi$, and the
modification axiom for the individual factors. 

\end{proof}

Given $2$-categories $X_1,X_2,X_3$, it is an easy observation that
\begin{align}\label{aCC1}
  a(C(C\times 1))=C(1\times C)a_{\times}\co X_1\times X_2\times X_3\rightarrow
  X_1\otimes (X_2\otimes X_3)
\puncteq{,}
\end{align}
where $a_{\times}$ is the associator of the cartesian product.

It is well-known that a strict, cubical tricategory is the same thing as a $\mathpzc{Gray}$-category. To prove this, one has to replace the cubical composition functor by the composition law
of a $\mathpzc{Gray}$-category. This
uses the underlying $\mathpzc{Set}$-isomorphism of
Proposition \ref{internal2isobicatc}. 
In the same fashion, in order to compare locally strict trihomomorphisms between
$\mathpzc{Gray}$-categories with Gray homomorphisms as it is done in
Theorem \ref{graygleichls} in \secref{graygleichls} below,
we need the following
many-variable version of Proposition \ref{internal2isobicatc} to
replace adjoint equivalences and modifications of cubical functors by
adjoint equivalences and modifications of the corresponding strict functors on Gray products.
 
\begin{theorem}\label{internal2isobicatcmulti}
Given a natural number $n$ and $2$-categories $Z,X_1,X_2,...,X_n$, 
composition with
\begin{align*}
  C(C(C(...)\times 1_{X_{n-1}})\times 1_{X_n})\co X_1\times X_2\times
  ...\times X_n\rightarrow (...((X_1\otimes X_2)\otimes X_3)\otimes
  ...)\otimes X_n
\puncteq{,}
\end{align*}
where $C$ is the universal cubical functor,  induces a natural isomorphism of $2$-categories
(i.e. a $\mathpzc{Cat}$-isomorphism)
\begin{align*}
  [(...((X_1\otimes
X_2)\otimes X_3)\otimes ...)\otimes X_n, Z]\cong \mathpzc{Bicat}_c(X_1,X_2,...,X_n;Z)
\puncteq{,}
\end{align*}
where $\mathpzc{Bicat}_c(X_1,X_2,...,X_n;Z)$ denotes the full sub-$2$-category of
$\mathpzc{Bicat}(X_1\times X_2\times ...\times X_n,Z)$ determined by
the cubical functors in $n$ variables. The same is of course true for
any other combination of universal cubical functors (mediated by the
unique isomorphism 
 in terms of associators for the Gray product).
\end{theorem}
\begin{proof}
Recall that the composition ($F\circ(F_1\times ...\times F_k)$) of
cubical functors is again a cubical functor. This shows that
the restriction of $(C(C(C(...)\times 1_{X_{n-1}})\times 1_{X_n}))^\ast$
to \mbox{$[(...((X_1\otimes
X_2)\otimes X_3)\otimes ...)\otimes X_n, Z]$} does indeed factorize through
$\mathpzc{Bicat}_c(X_1,X_2,...,X_n;Z)$. The proof that this gives an
isomorphism as wanted is then a straightforward extension of the
two-variable case. There are Gray product relations on combinations of
interchange cells, which correspond to relations for
the constraints holding by coherence. Note that indeed any
diagram of interchange cells commutes because these map to identities
in the cartesian product.

 For example, a cubical functor in three variables is determined by compatible
 partial cubical functors in two variables and a relation on their
 constraints cf. the diagram in \cite[Prop. 3.3, p. 42]{gurskicoherencein}.
This 
corresponds to a combination of 
the Gray product relation
\begin{align}\label{Sigmafg'g}
  \Sigma_{f,g'\ast g}\sim((1_{g'},1)\ast
  \Sigma_{f,g})\diamond (\Sigma_{f,g'}\ast
  (1,1_{g})) 
\end{align}
for $f=f_1$ and $g'= (f_2,1)$ and $g=(1,f_3)$ and $g'= (1,f_3)$ and
$g=(f_2,1)$ respectively, 
and the Gray product relation
\begin{align*}
  ((1,\beta)\ast (\alpha,1))\diamond \Sigma_{f,g}\sim
  \Sigma_{f',g'}\diamond ((\alpha,1)\ast (1,\beta))
\end{align*}
for $f=f_1=f'$, $\alpha=1_{f_1}$,
$g=(f_2,1)\ast(1,f_3)$, $g'=(1,f_3)\ast (f_2,1)$,
 and $\beta=\Sigma_{f_2,f_3}$, which reads
\begin{align*}
 & ((1,\Sigma_{f_2,f_3})\ast (1_{f_1},1))\diamond ((1_{(f_2,1)},1)\ast
  \Sigma_{f_1,(1,f_3)})\diamond (\Sigma_{f_1,(f_2,1)}\ast
  (1,1_{(1,f_3)}))\\
\sim &
 ((1_{(1,f_3)},1)\ast
  \Sigma_{f_1,(f_2,1)})\diamond (\Sigma_{f_1,(1,f_3)}\ast
  (1,1_{(f_2,1)})) \diamond ((1_{f_1},1)\ast
  (1,\Sigma_{f_2,f_3}))
\puncteq{.}
\end{align*}

For pseudonatural transformations and modifications, the arguments are
entirely analogous to the two-variable case.
\end{proof}

\subsection{$\mathpzc{V}$-enriched monad theory}

Recall from enriched category theory that there is a
$\mathpzc{V}$-functor $\mathrm{Hom}_{\mathpzc{L}}\co
\mathpzc{L}\op\otimes \mathpzc{L}\rightarrow \mathpzc{V}$, which sends
and object $(M,N)$ in the tensor product of $\mathpzc{V}$-categories
$\mathpzc{L}\op\otimes \mathpzc{L}$ to the hom object
$\mathpzc{L}(M,N)$ in $\mathpzc{V}$.
As is common, the corresponding partial  $\mathpzc{V}$-functors are
denoted $\mathpzc{L}(M,-)$ and $\mathpzc{L}(-,N)$ with hom morphisms
determined by the equations
\begin{align}\label{covapartialhom}
  e^{\mathpzc{L}(M,N)}_{\mathpzc{L}(M,N')}(\mathpzc{L}(M,-)_{N,N'}\otimes
  1) = M_{\mathpzc{L}}
\end{align}
and
\begin{align}\label{contrapartialhom}
  e^{\mathpzc{L}(M',N)}_{\mathpzc{L}(M,N)}(\mathpzc{L}(-,N)_{M,M'}\otimes
  1) = M_{\mathpzc{L}}c
\puncteq{.}
\end{align}

For an element
$f\co N\rightarrow N'$ in $\mathpzc{L}(N,N')$ i.e. a morphism in the
underlying category of $\mathpzc{L}$, we usually denote by
$\mathpzc{L}(M,f)\co \mathpzc{L}(M,N)\rightarrow \mathpzc{L}(M,N')$
the morphism in $\mathpzc{V}$ corresponding to the image of $f$ under
the underlying functor of $\mathpzc{L}(M,-)$ with respect to the identification
$\mathpzc{V}_0(I,[X,Y])\cong \mathpzc{V}_0(X,Y)$ induced from the
closed structure of $\mathpzc{V}$ for objects $X,Y\in\mathpzc{V}$.
Also note that we will occasionally write $\mathpzc{L}(1,f)$ instead e.g. if the functoriality of
\begin{align}
  \mathrm{hom}_{\mathpzc{L}}\co \mathpzc{L}\op_{0}\times
  \mathpzc{L}_{0}\xrightarrow{\phantom{(\mathrm{Hom}_{\mathpzc{L}})_0}}(\mathpzc{L}\op\otimes
  \mathpzc{L})_0\xrightarrow{(\mathrm{Hom}_{\mathpzc{L}})_0} \mathpzc{V}_0
\end{align}
is to be emphasized, where the first arrow is the canonical comparison functor \cite{kelly}.
This notation is obviously extended in the case that
$\mathpzc{V}=\mathpzc{Gray}$, e.g. given a $1$-cell $\alpha\co
f\rightarrow g$ in the $2$-category $\mathpzc{L}(N,N')\cong
[I,\mathpzc{L}(N,N')]$, then $\mathpzc{L}(M,\alpha)$ denotes the
pseudonatural transformation $\mathpzc{L}(M,f)\Rightarrow
\mathpzc{L}(M,g)$ given by $\mathpzc{L}(M,-)_{N,N'}(\alpha)$ i.e. the
$1$-cell in $[\mathpzc{L}(M,N),\mathpzc{L}(M,N')]$.

Recall that the $2$-category $\mathpzc{V}\text{-}\mathpzc{CAT}$ of
$\mathpzc{V}$-categories, $\mathpzc{V}$-functors, and
$\mathpzc{V}$-natural transformations is a symmetric monoidal
$2$-category with monoidal structure the
tensor product of $\mathpzc{V}$-categories  and unit object
the unit $\mathpzc{V}$-category $\mathcal{I}$ with a
single object $0$ and hom object $I$.
Recall that the 
$2$-functor
$(-)_0=\mathpzc{V}\text{-}\mathpzc{CAT}(\mathcal{I},-)\co
\mathpzc{V}\text{-}\mathpzc{CAT} \rightarrow \mathpzc{CAT}$ sends a
$\mathpzc{V}$-category to its underlying category, a
$\mathpzc{V}$-functor to its underlying functor, and a
$\mathpzc{V}$-natural transformation to its underlying natural transformation.

Let $T$ be a $\mathpzc{V}$-monad on a $\mathpzc{V}$-category
$\mathpzc{M}$. Recall that this means that $T$ is a monad in the
$2$-category $\mathpzc{V}\text{-}\mathpzc{CAT}$.
Thus $T$ is a $\mathpzc{V}$-functor $\mathpzc{M}\rightarrow
\mathpzc{M}$, and its multiplication and unit are \mbox{$\mathpzc{V}$-natural}
transformations $\mu\co TT\Rightarrow T$ and $\eta\co
1_{\mathpzc{M}}\Rightarrow T$ respectively such that 
\begin{align}\label{monadaxioms}
  \mu (\mu T)=\mu (T\mu)\quad\quad \mbox{and}\quad \quad\mu(\eta T)=  1_T=\mu(T\eta)\puncteq{,}
\end{align}
where $\mu T$ and $T\mu$ are as usual the $\mathpzc{V}$-natural transformations with component
\begin{align*}
  \mu_{TM}\co I\rightarrow \mathpzc{M}(TTTM,TM) \quad\mbox{and}\quad
  T_{TTM,M}\mu_M\co I\rightarrow
  \mathpzc{M}(TTTM,TM)
\end{align*}
respectively at the object
$M\in\mathpzc{M}$, and similarly for $\eta T$ and $T\eta$.

Under the assumption that $\mathpzc{V}$ has equalizers e.g. if
$\mathpzc{V}$ is complete, the Eilenberg-Moore object $\mathpzc{M}^T$
exists and has an explicit description,  on which we expand below.
For now, recall that the Eilenberg-Moore object is formally characterized by the
existence of an isomorphism
\begin{align}
  \mathpzc{V}\text{-}\mathpzc{Cat}(\mathpzc{K},\mathpzc{M}^T)\cong \mathpzc{V}\text{-}\mathpzc{Cat}(\mathpzc{K},\mathpzc{M})^{T_\ast}
\end{align}
of categories which is $\mathpzc{Cat}$-natural in $\mathpzc{K}$ and where $T_\ast$
is the ordinary monad induced by composition with $T$.

In particular, putting $\mathpzc{K}=\mathcal{I}$ shows that the
underlying category $\mathpzc{M}^T$ of the Eilenberg-Moore object in
$\mathpzc{V}\text{-}\mathpzc{CAT}$ is isomorphic to the
Eilenberg-Moore object for the underlying monad $T_0$ on the
underlying category $\mathpzc{M}_0$ of $\mathpzc{M}$.

Thus an object of $\mathpzc{M}^T$ i.e. a $T$-algebra is the same thing
as a $T_0$ algebra. This means, that it is given by a pair $(A,a)$ where $A$ is an object of
$\mathpzc{M}$ and $a$ is an element $I\rightarrow \mathpzc{M}(TA,A)$
such that the two algebra axioms hold true:
\begin{align}\label{algebraaxioms}
  M_{\mathpzc{M}}(a, T_{TA,A}a)= M_{\mathpzc{M}}(a,\mu_{A})
  \quad\quad\mbox{and}\quad\quad 1_A = M_{\mathpzc{M}}(a,\eta_A)\puncteq{.}
\end{align}
Here, the notation is already suggestive for the situation for
$\mathpzc{V}=\mathpzc{Gray}$. Namely, $(a,T_{TA,A}a)$ is considered as an
element of the underlying set $V(\mathpzc{M}(TA,A)\otimes \mathpzc{M}(TTA,TA))$, and we apply
the underlying function $VM_{\mathpzc{M}}$ to this, where
$V$ is usually dropped because for
$\mathpzc{V}=\mathpzc{Gray}$ the equations in \eqref{algebraaxioms}
make sense as equations of the values of strict functors on objects in
the Gray product.

Given $T$-algebras $(A,a)$ and $(B,b)$, the hom object of
$\mathpzc{M}^T$ is given by the following equalizer
\begin{equation}\label{equalizerhomEM}
\begin{tikzpicture}[baseline=(current  bounding  box.center)]
\matrix (a) [matrix of math nodes, row sep=3em, column sep=4.5em, text
height=1.5ex, text depth=0.25ex]{
  \mathpzc{M}^T((A,a),(B,b)) & \mathpzc{M}(A,B) & 
\mathpzc{M}(TA,B) \\ };
\draw[->]
($ (a-1-1.east) + (0,.0) $)
-- node [above] {$\scriptstyle (U^T)_{(A,a),(B,b)}$}($ (a-1-2.west) + (0,.0) $);

\draw[->]
($ (a-1-2.east) + (0,.075) $)
-- node [above] {$\scriptstyle \mathpzc{M}(a,1)$}($ (a-1-3.west) + (0,.075) $);
\draw[->]
($ (a-1-2.east) + (0,-.075) $)
-- node [below] {$\scriptstyle \mathpzc{M}(1,b)T_{A,B}$}($ (a-1-3.west) + (0,-.075) $);	
\puncttikz[a-1-3]{.}		
\end{tikzpicture}
\end{equation}
In fact, it is not hard to show that the composition law
$M_{\mathpzc{M}}$ and the
units $j_A$ of $\mathpzc{M}$ induce a $\mathpzc{V}$-category structure
on $\mathpzc{M}^T$ such that $U^T$ is a faithful
$\mathpzc{V}$-functor $\mathpzc{M}^T\rightarrow \mathpzc{M}$, which we call the forgetful functor.
The explicit arguments may be found in \cite{lintonrelative69}.

In the case that $\mathpzc{V}=\mathpzc{Gray}$ and $\mathpzc{K}$ is a
$\mathpzc{Gray}$-category with a $\mathpzc{V}$-monad $T$ on it, Gurski identifies $\mathpzc{K}^T$ explicitly in
\cite[13.1]{gurskicoherencein}. This is also
what the equalizer description gives when it is
spelled out:

\begin{proposition}\label{propstrictalgebras}
  The $\mathpzc{Gray}$-category of algebras for a
  $\mathpzc{Gray}$-monad $T$ on  a $\mathpzc{Gray}$-category $K$,
  i.e. the Eilenberg-Moore object $\mathpzc{K}^T$, can
  be described in the following way.
Objects are $T$-algebras: they are given by an object $X$ in
$\mathpzc{K}$ and a $1$-cell $x\co TX\rightarrow X$ i.e. an object in
$\mathpzc{K}(TX,X)$ satisfying $
M_{\mathpzc{K}}(x,Tx)=M_{\mathpzc{K}}(x,\mu_X)$ and $1_X=M_{\mathpzc{K}}(x,\eta_X)$.
These algebra axioms are abbreviated by $xTx=x\mu_X$ and $1_X=x\eta_X$ respectively.

An algebra $1$-cell $f\co(X,x)\rightarrow (Y,y)$ is given by a
$1$-cell $f\co X\rightarrow Y$ i.e. an object in $\mathpzc{K}(X,Y)$
such that $M_{\mathpzc{K}}(f,x)=M_{\mathpzc{K}}(y,Tf)$, which is
abbreviated by
$fx=yTf$. An algebra $2$-cell $\alpha\co f\Rightarrow g\co
(X,x)\rightarrow (Y,y)$ is given by a $2$-cell $\alpha\co f\rightarrow
g$ i.e. a $1$-cell in $\mathpzc{K}(X,Y)$ such that
$M_{\mathpzc{K}}(1_y,T\alpha)=M_{\mathpzc{K}}(\alpha,1_x)$, which is abbreviated
by $1_y T\alpha=\alpha 1_x$.
An algebra $3$-cell $\Gamma\co \alpha\Rrightarrow \beta\co
f\Rightarrow g\co (X,x)\rightarrow (Y,y)$ is given by a $3$-cell
$\Gamma\co \alpha\Rrightarrow \beta$ i.e. a $2$-cell in
$\mathpzc{K}(X,Y)$ such that
$M_{\mathpzc{K}}(1_{1_y},T\Gamma)=M_{\mathpzc{K}}(\Gamma,1_{1_x})$,
which is
abbreviated $1_{1_y}T\Gamma=\Gamma 1_{1_x}$.
The compositions are induced from the $\mathpzc{Gray}$-category
structure of $\mathpzc{K}$. \qed
\end{proposition}
Observe here that the common notation $xTx=x\mu_X$ for equations of
(composites of) morphisms in the underlying categories has been
obviously extended for $\mathpzc{V}=\mathpzc{Gray}$ to
$2$-cells and $3$-cells i.e. $1$- and $2$-cells in the hom
$2$-categories, where juxtaposition now denotes application of the
composition law of $\mathpzc{K}$, and the axioms for algebra $2$-
and $3$-cells are whiskered equations with respect to this composition on
$2$-cells and $3$-cells in $\mathpzc{K}$.

\section{The $\mathpzc{Gray}$-category $\mathrm{Ps}\text{-}T\text{-}\mathrm{Alg}$ 
  of pseudo algebras}\label{pseudoalgebras}

Let again $T$ be a $\mathpzc{Gray}$-monad on a
$\mathpzc{Gray}$-category $\mathpzc{K}$.
Since the underlying category $\mathpzc{2Cat}$ of the $\mathpzc{Gray}$-category
$\mathpzc{Gray}$ is complete, it has equalizers in particular, so we
have a convenient description of the
$\mathpzc{Gray}$-category $\mathpzc{K}^T$ of $T$-algebras in terms of
equalizers as in \secref{equalizerhomEM}.

Recall that for enrichment in $\mathpzc{Cat}$, there is a pseudo and a
lax version of the $2$-category of algebras with obvious
inclusions of the stricter into the laxer ones respectively. 
 Under suitable
conditions on the monad and its (co)domain, there are two coherence
results relating those different kinds of algebras.  First, each of the inclusions has a left
adjoint. Second, each component of the unit of the adjunction is an
internal equivalence.
The primary references for these results are \cite{blackwell} and
\cite{lackcodescent2002}. In particular, in the second, Lack provides
an analysis of
the coherence problem by use of codescent objects.
In the case of enrichment in $\mathpzc{Gray}$, there are partial results along these lines
by Power \cite{powerthreedimensional}, and a local version of the
identification of pseudo notions for the monad
\eqref{introexlicitLan} from the Introduction with tricategorical structures is mentioned in
\cite[Ex. 3.5, p. 319]{powerthreedimensional}.
A perspective similar to Lack's treatment is given by Gurski in
\cite[Part III]{gurskicoherencein}.

For a $\mathpzc{Gray}$-monad $T$ on a $\mathpzc{Gray}$-category $\mathpzc{K}$, Gurski gives a definition of lax algebras, lax functors of lax
algebras, transformations of lax functors, and modifications of those,
and shows that these assemble into a $\mathpzc{Gray}$-category $\mathrm{Lax}\text{-}T\text{-}\mathrm{Alg}$.
Further, he defines pseudo algebras, pseudo functors of
pseudo algebras, and shows that these, together with transformations of
pseudo functors and modifications of those, form a
$\mathpzc{Gray}$-category $\mathrm{Ps}\text{-}T\text{-}\mathrm{Alg}$, which embeds as a locally full
sub-$\mathpzc{Gray}$-category 
in the $\mathpzc{Gray}$-category $\mathrm{Lax}\text{-}T\text{-}\mathrm{Alg}$ of lax
algebras. Finally, there is an obvious $2$-locally full inclusion of
the
$\mathpzc{Gray}$-category   $\mathpzc{K}^T$  of  algebras into  $\mathrm{Ps}\text{-}T\text{-}\mathrm{Alg}$ and $\mathrm{Lax}\text{-}T\text{-}\mathrm{Alg}$.

\subsection{Definitions and two identities }
We reproduce here Gurski's definition of
$\mathrm{Ps}\text{-}T\text{-}\mathrm{Alg}$ in equational form. In
Section \ref{sectionGrayhom} we will identify this
$\mathpzc{Gray}$-category for a particular monad on the functor
$\mathpzc{Gray}$-category $[\mathrm{ob}\mathpzc{P},\mathpzc{L}]$
where $\mathpzc{P}$ is a small and $\mathpzc{L}$ is a cocomplete
$\mathpzc{Gray}$-category. Namely, we show that it is isomorphic as a
$\mathpzc{Gray}$-category to the full sub-$\mathpzc{Gray}$-category
$\mathpzc{Tricat}_{\mathrm{ls}}(\mathpzc{P},\mathpzc{L})$ determined by the
locally strict trihomomorphisms.

\begin{definition}\cite[Def. 13.4, Def. 13.8]{gurskicoherencein}
   A pseudo $T$-algebra consists of
   \begin{itemize}
   \item an object $X$ of $\mathpzc{K}$;
   \item a $1$-cell $x\co TX\rightarrow X$ i.e. an object in
     $\mathpzc{K}(TX,X)$;
   \item $2$-cell adjoint equivalences\footnote{For adjunctions in a
       $2$-category see \cite[\S2.]{kellystreet74}.} $(m,m^\bullet)\co M_{\mathpzc{K}}(x,
     Tx)\rightarrow M_{\mathpzc{K}}(x, \mu_X)$ or abbreviated $(m,m^\bullet)\co x
     Tx\rightarrow x\mu_X$
     and $(i,i^\bullet)\co 1_X\rightarrow M_{\mathpzc{K}}(x, \eta_X)$ or abbreviated $(i,i^\bullet)\co
     1\rightarrow x\eta_X$ i.e. $1$-cells in
     $\mathpzc{K}(T^2X,X)$ and $\mathpzc{K}(X,X)$ respectively which
     are adjoint equivalences;
   \item and three invertible $3$-cells $\pi,\lambda,\mu$ as in \textbf{(PSA1)}-\textbf{(PSA3)}  subject to
     the 
     four axioms  \textbf{(LAA1)}-\textbf{(LAA4)} of a lax $T$-algebra:
   \end{itemize} 
\noindent\textbf{(PSA1)}\; An invertible $3$-cell $\pi$ given by an
invertible $2$-cell in $\mathpzc{K}(T^3X,X)$:
     \begin{align*}
       \pi\co (m1_{\mu_{TX}})\ast(m1_{T^2x})\Rightarrow (m
       1_{T\mu_X})\ast (1_xTm)
\puncteq{,}
     \end{align*}
which is shorthand for
     \begin{align*}
       \pi\co M_{\mathpzc{K}}(m,1_{\mu_{TX}})\ast
       M_{\mathpzc{K}}(m,1_{T^2x}) \Rightarrow
       M_{\mathpzc{K}}(m,1_{T\mu_X})\ast M_{\mathpzc{K}}(1_x,Tm)
\puncteq{;} 
     \end{align*}
where the horizontal factors on the left compose due to
$\mathpzc{Gray}$-naturality of $\mu$ and the codomains match by the  monad axiom $\mu
(\mu T)=\mu (T\mu)$.

\noindent\textbf{(PSA2)}\; An invertible $3$-cell $\lambda$ given by an
invertible $2$-cell in  $\mathpzc{K}(TX,X)$:
\begin{align*}
  \lambda\co (m 1_{\eta_{TX}})\ast (i1_x)\Rightarrow 1_x
\puncteq{,}
\end{align*}
which is shorthand for
     \begin{align*}
       \lambda\co M_{\mathpzc{K}}(m,1_{\eta_{TX}})\ast
       M_{\mathpzc{K}}(i,1_{x})\Rightarrow 1_x
\puncteq{,}
     \end{align*}
where the horizontal factors compose due to
$\mathpzc{Gray}$-naturality of $\eta$ and the codomains match by the monad axiom $\mu
(\eta T)=1_T$.

\noindent\textbf{(PSA3)}\; An invertible $3$-cell $\rho$ given by an
invertible $2$-cell in  $\mathpzc{K}(TX,X)$:
\begin{align*}
  \rho\co (m1_{T\eta_X})\ast (1_x Ti)\Rightarrow 1_x
\puncteq{,}
\end{align*}
which is shorthand for
     \begin{align*}
       \rho \co M_{\mathpzc{K}}(m,1_{T\eta_X})\ast
       M_{\mathpzc{K}}(1_x,Ti)\Rightarrow 1_x
\puncteq{,}
     \end{align*}
where the codomains match by the monad axiom $\mu (T\eta)=1_T$.

The four lax algebra axioms are:

\noindent\textbf{(LAA1)}\; The following equation in $\mathpzc{K}(T^4X,X)$ of vertical
composites of whiskered $3$-cells is required:
\begin{align*}
  &((\pi 1)\ast 1_{11T^2m})\diamond (1_{m11}\ast \Sigma^{-1}_{m,T^2m})\diamond ((\pi 1)\ast 1_{m11})
  \\ & = (1_{m11}\ast (1T\pi ))\diamond ((\pi 1) \ast 1_{1Tm1}) \diamond
  (1_{m11}\ast (\pi 1))
\puncteq{,}
\end{align*}
where  $\Sigma^{-1}_{m,T^2m}$ is shorthand for
$M_{\mathpzc{K}}(\Sigma^{-1}_{m,T^2m})$. A careful inspection shows
that the horizontal and vertical factors do indeed compose. Note that any mention of the object
$X$ has been omitted, e.g. $T\mu_T$ stands for $T\mu_{TX}$. We refer
to this axiom as the \emph{pentagon-like axiom} for $\pi$.

\noindent\textbf{(LAA2)}\; The following equation in $\mathpzc{K}(T^2X,X)$ of vertical
composites of whiskered $3$-cells is required:
\begin{align*}
  ((\rho 1)\ast 1_{m1})\diamond (1_{m11}\ast
  \Sigma_{m,T^2i})= (1_{m11}\ast (1_{x}T\rho))\diamond ((\pi
  1)\ast 1_{11T^2i})
\puncteq{.}
\end{align*}

\noindent\textbf{(LAA3)}\; The following equation in $\mathpzc{K}(T^2X,X)$ of vertical
composites of whiskered $3$-cells is required:
\begin{align*}
  1_{m11}\ast (\lambda 1)=((\lambda 1)\ast
  1_{1m})\diamond(1_{m11}\ast\Sigma^{-1}_{i,m})  \diamond((\pi 1)\ast 1_{i11})
\puncteq{.}
\end{align*}

\noindent\textbf{(LAA4)}\; The following equation in $\mathpzc{K}(T^4X,X)$ of vertical
composites of whiskered $3$-cells is required:
\begin{align*}
  ( 1_{m11}\ast (1T\lambda))\diamond ((\pi 1)\ast  1_{1Ti1}) =  1_{m11}\ast
  (\rho 1)
\puncteq{.}
\end{align*}
We refer to this as the \emph{triangle-like axiom} for $\lambda$, $\rho$, and
$\pi$. 
Diagrams for these axioms may be found in Gurski's definition.
\end{definition}

   \begin{remark}
     In the shorthand notation juxtaposition stands for an application
     of $M_{\mathpzc{K}}$, an instance of a power of $T$ in an index
     refers to its effect on the object $X$, any other
     instance of a power of $T$ is shorthand for a hom $2$-functor
     and only
     applies to the cell directly following it. Notice that this
     notation is possible due to the functor axiom for $T$.
   \end{remark}
This definition of a pseudo algebra is derived from the
definition of a lax algebra by requiring the $2$-cells $m$ and $i$ to
be adjoint equivalences 
and the $3$-cells $\pi,\lambda$, $\rho$ to be invertible. In fact,
under these circumstances we do not need all of the axioms. This is proved in the
following proposition, which is central for the comparison of
trihomomorphisms of $\mathpzc{Gray}$-categories with pseudo algebras. Namely, there are only two axioms for a
trihomomorphism, while there are four in the definition of a lax
algebra. Proposition \ref{kelly2axioms} shows that, in general, two of
the axioms suffice for a pseudo algebra.
 
   \begin{proposition}\label{kelly2axioms}
     Given a pseudo $T$-algebra, the pentagon-like axiom
     \textup{\textbf{(LAA1})} and the
     triangle-like axiom \textup{\textbf{(LAA4})} imply the other two
     axioms \textup{\textbf{(LAA2)}}-\textup{\textbf{(LAA3)}}, i.e. these are redundant.
   \end{proposition}
   \begin{proof}  
We proceed analogous to Kelly's classical proof that the two
corresponding axioms in
MacLane's original definition of a monoidal category are redundant
\cite{kellyonmaclanes64}. The associators and unitors in Kelly's proof here correspond
to $\pi,\lambda$, and $\rho$. Commuting naturality squares for
associators have to be replaced by
instances of the middle four interchange law, and there is an
additional complication due to the appearance of interchange cells --
these have no counterpart in Kelly's proof, so that it is gratifying that the
strategy of the proof can still be applied. 
We only show here the proof for the axiom \textbf{(LAA3)} involving $\pi$ and $\lambda$, the
one for the axiom \textbf{(LAA2)} involving $\pi$ and $\rho$ is entirely analogous.

The general idea of the proof is to transform the equation of the
axiom \textbf{(LAA3)} into an equivalent form, namely
\eqref{lefthandrighthand} below, which we can manipulate
by use of the pentagon-like and triangle-like axiom. This is probably
easier to see from the diagrammatic form of the axioms, where it means
that we adjoin 
\begin{equation*}
\begin{tikzpicture}[remember picture, every node/.style={scale=0.77}]]
\matrix (a) [matrix of math nodes, row sep=3em, column sep=1em, text
height=1.5ex, text depth=0.25ex]{ x1Tx T^2x &
xTx T\eta Tx T^2x  \\ 
x1Tx T\mu & xTx T\eta Tx T\mu   \\
 xTxT\mu T\eta_{T}T\mu & xTxT\mu T^2\mu T\eta_{T^2} \\
 };
\path[->] (a-1-1) edge 	node[above]{$\scriptstyle 1Ti11$} 
						  	(a-1-2);
\path[->] (a-1-1) edge 	node[left]{$\scriptstyle 11Tm$} 
						  	(a-2-1);
\path[->] (a-1-2) edge 	node[left]{$\scriptstyle 111Tm$} 
										(a-2-2);
\path[->] (a-2-1) edge 	node[above]{$\scriptstyle 1Ti11$} 
						  	(a-2-2);
\path[->] (a-2-2) edge 	node[left]{$\scriptstyle 1Tm11$} 
										(a-3-2);
\draw[-]
($ (a-2-1.south) + (.04,.0) $)
-- node [above] {}($ (a-3-1.north) + (.04,.0) $);
\draw[-]
($ (a-2-1.south) + (-.04,.0) $)
-- node [above] {}($ (a-3-1.north) + (-.04,.0) $);
\draw[-]
($ (a-3-1.east) + (.0,.04) $)
-- node [above] {}($ (a-3-2.west) + (.,.04) $);
\draw[-]
($ (a-3-1.east) + (.0,-.04) $)
-- node [above] {}($ (a-3-2.west) + (.,-.04) $);
\node at ($ 1/2*(a-2-1) + 1/2*(a-1-2) + (-.4,0) $) {$ \Downarrow
  \Sigma^{-1}_{1Ti,Tm} $};
\node at ($ 1/2*(a-3-1) + 1/2*(a-2-2) + (-.4,0) $) {$ \Downarrow
  1 T\lambda 1 $};
\puncttikz[a-2-2]{}
\end{tikzpicture}
\end{equation*}
to the right hand side of some image of the pentagon-like axiom
\textbf{(LAA1)}, and then a diagram equivalent to the right hand side
of \textbf{(LAA3)} can be
identified as a subdiagram of this.

Since $i\co 1_X\rightarrow x\eta_X$ is an equivalence in
$\mathpzc{L}(X,X)$, and since
\begin{align*}
  \mathpzc{L}(X',1_X)\co \mathpzc{L}(X',X)\rightarrow
  \mathpzc{L}(X',X)
\end{align*}
is the identity for arbitrary $X'\in\mathrm{ob}\mathpzc{L}$, we
have that $\mathpzc{L}(X',x\eta_X)$ is
equivalent to the identity functor. 
In particular, it is $2$-locally fully faithful i.e. a bijection on the
sets of $2$-cells. On the other hand, by naturality of $\eta$ we have:
\begin{align*}
  \mathpzc{L}(X',x\eta_X)=\mathpzc{L}(X',x)\mathpzc{L}(X',\eta_X)=\mathpzc{L}(X',x)\mathpzc{L}(\eta_{X'},TX)T_,
\puncteq{,}
\end{align*}
where the subscript of $T$ on the
right indicates a hom morphism of $T$. 
This means that the equation of \textbf{(LAA3)}  is equivalent to its image
under
\begin{align*}
  \mathpzc{L}(T^2X,x)\mathpzc{L}(\eta_{T^2X},TX)T_,=\mathpzc{L}(\eta_{T^2
    X},X)\mathpzc{L}(T^3X,x)T_,
\puncteq{,}
\end{align*}
where we have used the underlying
functoriality of $\mathrm{hom}_{\mathpzc{L}}$.
We will actually show that the image of the equation under
$\mathpzc{L}(T^3X,x)T_,$ holds, which of course implies that the image of the
equation under $\mathpzc{L}(\eta_{T^2
  X},X)\mathpzc{L}(T^3X,x)T_,$ holds.

Applying $\mathpzc{L}(T^3X,x)T_,$ to the lax algebra axiom \textbf{(LAA3)} gives
\begin{align}\label{zwischen0}
  1_{1Tm11}\ast  (1 T\lambda 1)=((1 T\lambda 1)\ast
  1_{11Tm})\diamond(1_{1Tm11}\ast (1T\Sigma^{-1}_{i,m}))
  \diamond((1T\pi 1)\ast 1_{1Ti11})
\puncteq{.}
\end{align}
Observe that since $\Sigma^{-1}_{i,m}$ is shorthand for
$M_{\mathpzc{L}}(\Sigma^{-1}_{i,m})$, we have
$1T\Sigma^{-1}_{i,m}=1\Sigma^{-1}_{Ti,Tm}$ by the
functor axiom for $T$ and equation \eqref{FGSigma} from \secref{FGSigma}, which is in fact shorthand for
\begin{align*}
  M_{\mathpzc{L}}((1,M_{\mathpzc{L}}(\Sigma^{-1}_{Ti,Tm}))) & =
  M_{\mathpzc{L}} (M_{\mathpzc{L}}\otimes
  1)(a^{-1}(1,\Sigma^{-1}_{Ti,Tm})) & \quad \mbox{(by a
    $\mathpzc{Gray}$-category axiom)}\\
& =M_{\mathpzc{L}}(M_{\mathpzc{L}}\otimes
  1)(\Sigma^{-1}_{(1,Ti),Tm}) & \quad \mbox{(by eq. \eqref{aSigma1gh} from \secref{aSigma1gh})} \\
& =M_{\mathpzc{L}}(\Sigma_{1Ti,Tm}^{-1})\puncteq{,} & \quad \mbox{(by
  eq. \eqref{FGSigma} from \secref{FGSigma})}
\end{align*}
for which the corresponding shorthand is just $\Sigma^{-1}_{1Ti,Tm}$. 

Next, equation \eqref{zwischen0} is clearly equivalent to the one whiskered with
$m111$ on the
left  because $m111$  is an (adjoint) equivalence
by the definition of $m$, i.e. to
\begin{align}\label{zwischen1}
  & 1_{m111}\ast  1_{1Tm11}\ast(1 T\lambda 1) \nonumber\\
  & =(1_{m111}\ast(1 T\lambda 1)\ast
  1_{11Tm})\diamond(1_{m111}\ast1_{1Tm11}\ast
  (\Sigma^{-1}_{1Ti,Tm}))  \diamond(1_{m111}\ast (1T\pi 1)\ast
  1_{1Ti11}) 
\puncteq{.}
\end{align}
Here we have used functoriality of $\ast$ i.e. the middle four
interchange law, and it is understood that because horizontal composition is
associative, we can drop parentheses.

Now observe that $1_{m111}\ast (1T\pi 1)$ is the image under
$\mathpzc{L}(T\eta_{T^2X},X)$ of the leftmost vertical factor
in the right hand side of the pentagon-like axiom \textbf{(LAA1)} for
$\pi$.
Namely, the image of \textbf{(LAA1)} under
$\mathpzc{L}(T\eta_{T^2X},X)$ is
\begin{align*}
  &((\pi 1 1)\ast 1_{11T^2m 1})\diamond (1_{m111}\ast (\Sigma^{-1}_{m,T^2m}1))\diamond ((\pi 11)\ast 1_{m111})
  \\ & = (1_{m111}\ast (1T\pi 1))\diamond ((\pi 11) \ast 1_{1Tm11}) \diamond
  (1_{m111}\ast (\pi 11))
\puncteq{,}
\end{align*}
where  $\Sigma^{-1}_{m,T^2m}1$
 is shorthand for
\begin{align*}
  M_{\mathpzc{L}}(M_{\mathpzc{L}}(\Sigma^{-1}_{m,T^2m}),1) & =M_{\mathpzc{L}}(M_{\mathpzc{L}}\otimes
  1)(\Sigma^{-1}_{m,T^2m},1) & \quad \mbox{(by definition of
    $(M_{\mathpzc{L}}\otimes 1)$)}\\
& = M_{\mathpzc{L}}(1\otimes
  M_{\mathpzc{L}})a (\Sigma^{-1}_{m,T^2m},1) & \quad \mbox{(by a
    $\mathpzc{Gray}$-category axiom)}\\
& = M_{\mathpzc{L}}(1\otimes
  M_{\mathpzc{L}}) (\Sigma^{-1}_{m,(T^{2}m,1)}) & \quad \mbox{(by
    eq. \eqref{aSigmafg1} from \secref{aSigmafg1})} \\
& =M_{\mathpzc{L}}(\Sigma^{-1}_{m,T^{2}m1}) \puncteq{,} & \quad \mbox{(by
    eq. \eqref{FGSigma} from \secref{FGSigma})} 
\end{align*}
for which the corresponding shorthand is just $\Sigma^{-1}_{m,T^2m 1}$.

In fact, we have that $T^2m 1= 1 Tm$ where on the left the identity is
$1_{T\eta_{T^2X}}$ and on the right it is $1_{T\eta_X}$ , thus this is simply
naturality of $T\eta$.
Hence,  $\Sigma^{-1}_{m,T^2m 1}=\Sigma^{-1}_{m,1Tm}$. In turn, this is shorthand for
\begin{align*}
  M_{\mathpzc{L}}(1\otimes M_{\mathpzc{L}})(\Sigma^{-1}_{m,(1,Tm)})
   & =M_{\mathpzc{L}}(M_{\mathpzc{L}}\otimes
  1)a^{-1}(\Sigma^{-1}_{(m,(1,Tm))}) & \quad \mbox{(by a
    $\mathpzc{Gray}$-category axiom)}\\
& =M_{\mathpzc{L}}(M_{\mathpzc{L}}\otimes
  1)(\Sigma^{-1}_{(m,1),Tm})& \quad \mbox{(by eq. \eqref{aSigmaf1h}
    from \secref{aSigmaf1h})}\\
 & =M_{\mathpzc{L}}(\Sigma^{-1}_{m1,Tm})  \puncteq{,} & \quad \mbox{(by
    eq. \eqref{FGSigma} from \secref{FGSigma})} 
\end{align*}
for which the corresponding shorthand is just $\Sigma^{-1}_{m1,Tm}$.

Implementing these transformations, the image of the pentagon-like axiom \textbf{(LAA1)} then has the form
\begin{align}\label{zwischen2}
  &((\pi 1 1)\ast 1_{111Tm})\diamond (1_{m111}\ast \Sigma^{-1}_{m1,Tm})\diamond ((\pi 11)\ast 1_{m111})
  \nonumber \\ & = (1_{m111}\ast (1T\pi 1))\diamond ((\pi 11) \ast 1_{1Tm11}) \diamond
  (1_{m111}\ast (\pi 11))
\puncteq{.}
\end{align}
Since they are invertible, composing  equation \eqref{zwischen1} with
the other two factors from the pentagon-like axiom for $\pi$ whiskered
with the (adjoint) equivalence $1Ti11$ on the right, gives an
equivalent equation.
Thus, our goal is now to prove the following equation:
\begin{align}  \label{lefthandrighthand}
  & (1_{m111}\ast  1_{1Tm11}\ast(1 T\lambda 1)) \diamond ((\pi 11) \ast
  1_{1Tm11} \ast
  1_{1Ti11}) \diamond
  ( 1_{m111}\ast (\pi 11)\ast  1_{1Ti11}) \nonumber \\ & =( 1_{m111}\ast(1 T\lambda 1)\ast
  1_{11Tm})\diamond( 1_{m111}\ast 1_{1Tm11}\ast (\Sigma^{-1}_{1Ti,Tm}))
\nonumber \\ &\quad
\diamond\Big(\Big(( 1_{m111}\ast (1T\pi 1)) \diamond((\pi 11) \ast  1_{1Tm11}) \diamond
( 1_{m111}\ast (\pi 11))\Big)\ast  1_{1Ti11}\Big)
\puncteq{.}
\end{align}

This is proved by transforming the right hand side by use of the
pentagon-like and triangle-like axiom until we finally obtain the left
hand side.

Namely, using the image of the pentagon-like axiom for $\pi$ in the
form \eqref{zwischen2} above, the right hand side of
\eqref{lefthandrighthand} is equal to
\begin{align*}
  & ( 1_{m111}\ast(1 T\lambda 1)\ast
  1_{11Tm})\diamond( 1_{m111}\ast 1_{1Tm11}\ast (\Sigma^{-1}_{1Ti,Tm}))
\\ &\quad
\diamond\Big(\Big( ((\pi 1 1)\ast  1_{111Tm})\diamond (
1_{m111}\ast \Sigma^{-1}_{m1,Tm})\diamond ((\pi 11)\ast
1_{m111})\Big)\ast  1_{1Ti11}\Big)\puncteq{.}
\end{align*}
The diagrammatic form of this is drawn below.
\begin{equation*}
\begin{tikzpicture}[remember picture, every node/.style={scale=0.77}]
\matrix (a) [matrix of math nodes, row sep=3em, column sep=1em, text
height=1.5ex, text depth=0.25ex]{ x1Tx T^2x &
xTx T\eta Tx T^2x & x\mu T \eta Tx T^2x & xTx\mu_T T^3xT\eta_{T^2} &
x\mu \mu_T T^3xT\eta_{T^2} & x\mu T\mu T^3x T\eta_{T^2}\\ 
x1Tx T\mu & xTx T\eta Tx T\mu  & x\mu T\eta TxT\mu&
xTxT^2x\mu_{T^2}T\eta_{T^2} & x\mu T^2\mu_{T^2}T\eta_{T^2} & x\mu
T^2xT\mu_T T\eta_{T^2} \\
 xTxT\mu T\eta_{T}T\mu & xTxT\mu T^2\mu T\eta_{T^2} & xTx\mu_T T^2\mu T\eta_{T^2} & xTxT\mu\mu_{T^2}T\eta_{T^2} & xTx\mu_T\mu_{T^2}T\eta_{T^2}& xTx\mu_TT\mu_T T\eta_{T^2}\\
& x\mu T\mu T^2\mu T\eta_{T^2} & x\mu\mu_TT^2\mu T\eta_{T^2}& x\mu T\mu\mu_{T^2}T\eta_{T^2} & x\mu
\mu_T\mu_{T^2}T\eta_{T^2} & x\mu\mu_T T\mu_T T\eta_{T^2}
 \\ };
\path[->] (a-1-1) edge 	node[above]{$\scriptstyle 1Ti11$} 
						  	(a-1-2);
\path[->] (a-1-2) edge 	node[above]{$\scriptstyle m111$} 
						  	(a-1-3);
\draw[-]
($ (a-1-3.east) + (.0,.04) $)
-- node [above] {}($ (a-1-4.west) + (.,.04) $);
\draw[-]
($ (a-1-3.east) + (.0,-.04) $)
-- node [above] {}($ (a-1-4.west) + (.,-.04) $);
\path[->] (a-1-4) edge 	node[above]{$\scriptstyle m111$} 
	       					  	(a-1-5);
\draw[-]
($ (a-1-5.east) + (.0,.04) $)
-- node [above] {}($ (a-1-6.west) + (.,.04) $);
\draw[-]
($ (a-1-5.east) + (.0,-.04) $)
-- node  [above] {}($ (a-1-6.west) + (.,-.04) $);
\path[->] (a-1-1) edge 	node[left]{$\scriptstyle 11Tm$} 
						  	(a-2-1);
\path[->] (a-1-2) edge 	node[left]{$\scriptstyle 111Tm$} 
										(a-2-2);
\path[->] (a-1-3) edge 	node[right]{$\scriptstyle 111Tm$} 
										(a-2-3);
\draw[-]
($ (a-1-4.south) + (.04,.0) $)
-- node [above] {}($ (a-2-4.north) + (.04,.0) $);
\draw[-]
($ (a-1-4.south) + (-.04,.0) $)
-- node [above] {}($ (a-2-4.north) + (-.04,.0) $);
\draw[-]
($ (a-1-5.south) + (.04,.0) $)
-- node [above] {}($ (a-2-5.north) + (.04,.0) $);
\draw[-]
($ (a-1-5.south) + (-.04,.0) $)
-- node [above] {}($ (a-2-5.north) + (-.04,.0) $);
\draw[-]
($ (a-1-6.south) + (.04,.0) $)
-- node [above] {}($ (a-2-6.north) + (.04,.0) $);
\draw[-]
($ (a-1-6.south) + (-.04,.0) $)
-- node [above] {}($ (a-2-6.north) + (-.04,.0) $);
\path[->] (a-2-1) edge 	node[above]{$\scriptstyle 1Ti11$} 
						  	(a-2-2);
\path[->] (a-2-2) edge 	node[below]{$\scriptstyle m111$} coordinate[midway] (doleftpi1) (a-2-3);
\path[->] (a-2-4) edge 	node[below]{$\scriptstyle m111$} (a-2-5);
\draw[-]
($ (a-2-1.south) + (.04,.0) $)
-- node [above] {}($ (a-3-1.north) + (.04,.0) $);
\draw[-]
($ (a-2-1.south) + (-.04,.0) $)
-- node [above] {}($ (a-3-1.north) + (-.04,.0) $);
\path[->] (a-2-2) edge 	node[left]{$\scriptstyle 1Tm11$} 
										(a-3-2);
\draw[-]
($ (a-2-3.south) + (.04,.0) $)
-- node [above] {}($ (a-3-3.north) + (.04,.0) $);
\draw[-]
($ (a-2-3.south) + (-.04,.0) $)
-- node [above] {}($ (a-3-3.north) + (-.04,.0) $);
\path[->] (a-2-4) edge 	node[left]{$\scriptstyle 1Tm11 $} 
										(a-3-4);
\draw[-]
($ (a-2-5.south) + (.04,.0) $)
-- node [above] {}($ (a-3-5.north) + (.04,.0) $);
\draw[-]
($ (a-2-5.south) + (-.04,.0) $)
-- node [above] {}($ (a-3-5.north) + (-.04,.0) $);
\draw[-]
($ (a-2-6.south) + (.04,.0) $)
-- node [above] {}($ (a-3-6.north) + (.04,.0) $);
\draw[-]
($ (a-2-6.south) + (-.04,.0) $)
-- node [above] {}($ (a-3-6.north) + (-.04,.0) $);
\draw[-]
($ (a-3-1.east) + (.0,.04) $)
-- node [above] {}($ (a-3-2.west) + (.,.04) $);
\draw[-]
($ (a-3-1.east) + (.0,-.04) $)
-- node [above] {}($ (a-3-2.west) + (.,-.04) $);
\draw[-]
($ (a-3-3.east) + (.0,.04) $)
-- node [above] {}($ (a-3-4.west) + (.,.04) $);
\draw[-]
($ (a-3-3.east) + (.0,-.04) $)
-- node [above] {}($ (a-3-4.west) + (.,-.04) $);
\draw[-]
($ (a-3-5.east) + (.0,.04) $)
-- node [above] {}($ (a-3-6.west) + (.,.04) $);
\draw[-]
($ (a-3-5.east) + (.0,-.04) $)
-- node [above] {}($ (a-3-6.west) + (.,-.04) $);
\path[->] (a-3-2) edge 	node[left]{$\scriptstyle m111 $} 
										(a-4-2);
\path[->] (a-3-3) edge 	node[left]{$\scriptstyle m111 $} 
	       								(a-4-3);
\path[->] (a-3-4) edge 	node[right]{$\scriptstyle m111 $} 
	       								(a-4-4);
\path[->] (a-3-5) edge 	node[left]{$\scriptstyle m111 $} 
	       								(a-4-5);
\path[->] (a-3-6) edge 	node[right]{$\scriptstyle m111 $} 
										(a-4-6);
\draw[-]
($ (a-4-2.east) + (.0,.04) $)
-- node [above] {}($ (a-4-3.west) + (.,.04) $);
\draw[-]
($ (a-4-2.east) + (.0,-.04) $)
-- node [above] {}($ (a-4-3.west) + (.,-.04) $);
\draw[-]
($ (a-4-3.east) + (.0,.04) $)
-- node [above] {}($ (a-4-4.west) + (.,.04) $);
\draw[-]
($ (a-4-3.east) + (.0,-.04) $)
-- node [above] {}($ (a-4-4.west) + (.,-.04) $);
\draw[-]
($ (a-4-4.east) + (.0,.04) $)
-- node [above] {}($ (a-4-5.west) + (.,.04) $);
\draw[-]
($ (a-4-4.east) + (.0,-.04) $)
-- node [above] {}($ (a-4-5.west) + (.,-.04) $);
\draw[-]
($ (a-4-5.east) + (.0,.04) $)
-- node [above] {}($ (a-4-6.west) + (.,.04) $);
\draw[-]
($ (a-4-5.east) + (.0,-.04) $)
-- node [above] {}($ (a-4-6.west) + (.,-.04) $);
\node at ($ 1/2*(a-3-2) + 1/2*(a-3-3) + (0,.3) $) {$ \Downarrow
  \pi 11$};
\node at ($ 1/2*(a-3-4) + 1/2*(a-3-5) + (0,.3) $) {$\Downarrow
  \pi 11$};
\node at ($ 1/2*(a-1-3) + 1/2*(a-2-2) $) {$ \Downarrow
  \Sigma^{-1}_{m1,Tm} $};
\node at ($ 1/2*(a-3-3) + 1/2*(a-1-4) + (0,.3) $) {$
  =$};
\node at ($ 1/2*(a-3-5) + 1/2*(a-1-6) + (0,.3)$) {$
  =$};
\node at ($ 1/2*(a-4-5) + 1/2*(a-3-6) $) {$
  =$};
\node at ($ 1/2*(a-2-4) + 1/2*(a-1-5) $) {$
  =$};
\node at ($ 1/2*(a-4-3) + 1/2*(a-3-4) $) {$
  =$};
\node at ($ 1/2*(a-2-1) + 1/2*(a-1-2) + (-.4,0) $) {$ \Downarrow
  \Sigma^{-1}_{1Ti,Tm} $};
\node at ($ 1/2*(a-3-1) + 1/2*(a-2-2) + (-.4,0) $) {$ \Downarrow
  1 T\lambda 1 $};
\puncttikz[a-2-2]{}
\end{tikzpicture}
\end{equation*}

The rectangle composed of the two interchange cells is shorthand for
\begin{align*}
  M_{\mathpzc{L}}((1_{m1},1)\ast\Sigma^{-1}_{1Ti,Tm}) \diamond
  M_{\mathpzc{L}}(\Sigma^{-1}_{m1,Tm} \ast
  (1_{1Ti},1))=M_{\mathpzc{L}}(\Sigma_{(m1)\ast(1Ti),Tm})
\puncteq{,}
\end{align*}
for which the corresponding shorthand is just $\Sigma_{(m1)\ast(1Ti),Tm}$.
Notice that we made use here of the image under $M_{\mathpzc{L}}$ of the Gray product relation 
\begin{align*}
(\Sigma_{f',g}\ast (1_f,1))\diamond ((1_{f'},1)\ast \Sigma_{f,g}) \sim
\Sigma_{f'\ast f,g}
\puncteq{.}
  \end{align*}

Next, the subdiagram formed by $1T\lambda 1$ and $\pi11$ may
be transformed by use of the image under $\mathpzc{L}(T\mu,X)$ of the
triangle-like axiom \textbf{(LAA3)}:
\begin{align*}
  ( 1_{m111}\ast (1T\lambda1))\diamond ((\pi 11)\ast  1_{1Ti1}) =
  1_{m111}\ast
  (\rho 11)
\puncteq{.}
\end{align*}
Implementing these transformations in the diagram, gives the one drawn below.
\begin{equation*}
\begin{tikzpicture}[remember picture, every node/.style={scale=0.77}]
\matrix (a) [matrix of math nodes, row sep=3em, column sep=1em, text
height=1.5ex, text depth=0.25ex]{ x1Tx T^2x &
xTx T\eta Tx T^2x & x\mu T \eta Tx T^2x & xTx\mu_T T^3xT\eta_{T^2} &
x\mu \mu_T T^3xT\eta_{T^2} & x\mu T\mu T^3x T\eta_{T^2} \\ 
x1Tx T\mu & xTx T\eta Tx T\mu  & x\mu T\eta TxT\mu&
xTxT^2x\mu_{T^2}T\eta_{T^2} & x\mu T^2\mu_{T^2}T\eta_{T^2} & x\mu
T^2xT\mu_T T\eta_{T^2} \\
 xTxT\mu T\eta_{T}T\mu & xTxT\mu T^2\mu T\eta_{T^2} & xTx\mu_T T^2\mu T\eta_{T^2} & xTxT\mu\mu_{T^2}T\eta_{T^2} & xTx\mu_T\mu_{T^2}T\eta_{T^2}& xTx\mu_TT\mu_T T\eta_{T^2}\\
& x\mu T\mu T^2\mu T\eta_{T^2} & x\mu\mu_TT^2\mu T\eta_{T^2}& x\mu T\mu\mu_{T^2}T\eta_{T^2} & x\mu
\mu_T\mu_{T^2}T\eta_{T^2} & x\mu\mu_T T\mu_T T\eta_{T^2}
 \\ };
\path[->] (a-1-1) edge 	node[above]{$\scriptstyle 1Ti11$} 
						  	(a-1-2);
\path[->] (a-1-2) edge 	node[above]{$\scriptstyle m111$} 
						  	(a-1-3);
\draw[-]
($ (a-1-3.east) + (.0,.04) $)
-- node [above] {}($ (a-1-4.west) + (.,.04) $);
\draw[-]
($ (a-1-3.east) + (.0,-.04) $)
-- node [above] {}($ (a-1-4.west) + (.,-.04) $);
\path[->] (a-1-4) edge 	node[above]{$\scriptstyle m111$} 
	       					  	(a-1-5);
\draw[-]
($ (a-1-5.east) + (.0,.04) $)
-- node [above] {}($ (a-1-6.west) + (.,.04) $);
\draw[-]
($ (a-1-5.east) + (.0,-.04) $)
-- node  [above] {}($ (a-1-6.west) + (.,-.04) $);
\path[->] (a-1-1) edge 	node[left]{$\scriptstyle 11Tm$} 
						  	(a-2-1);
\path[->] (a-1-3) edge 	node[right]{$\scriptstyle 111Tm$} 
										(a-2-3);
\draw[-]
($ (a-1-4.south) + (.04,.0) $)
-- node [above] {}($ (a-2-4.north) + (.04,.0) $);
\draw[-]
($ (a-1-4.south) + (-.04,.0) $)
-- node [above] {}($ (a-2-4.north) + (-.04,.0) $);
\draw[-]
($ (a-1-5.south) + (.04,.0) $)
-- node [above] {}($ (a-2-5.north) + (.04,.0) $);
\draw[-]
($ (a-1-5.south) + (-.04,.0) $)
-- node [above] {}($ (a-2-5.north) + (-.04,.0) $);
\draw[-]
($ (a-1-6.south) + (.04,.0) $)
-- node [above] {}($ (a-2-6.north) + (.04,.0) $);
\draw[-]
($ (a-1-6.south) + (-.04,.0) $)
-- node [above] {}($ (a-2-6.north) + (-.04,.0) $);
\path[->] (a-2-1) edge 	node[above]{$\scriptstyle 1Ti11$} 
						  	(a-2-2);
\path[->] (a-2-2) edge 	node[below]{$\scriptstyle m111$} coordinate[midway] (doleftpi1) (a-2-3);
\path[->] (a-2-4) edge 	node[below]{$\scriptstyle m111$} (a-2-5);
\draw[-]
($ (a-2-1.south) + (.04,.0) $)
-- node [above] {}($ (a-3-1.north) + (.04,.0) $);
\draw[-]
($ (a-2-1.south) + (-.04,.0) $)
-- node [above] {}($ (a-3-1.north) + (-.04,.0) $);
\draw[-]
($ (a-2-3.south) + (.04,.0) $)
-- node [above] {}($ (a-3-3.north) + (.04,.0) $);
\draw[-]
($ (a-2-3.south) + (-.04,.0) $)
-- node [above] {}($ (a-3-3.north) + (-.04,.0) $);
\path[->] (a-2-4) edge 	node[left]{$\scriptstyle 1Tm11 $} 
										(a-3-4);
\draw[-]
($ (a-2-5.south) + (.04,.0) $)
-- node [above] {}($ (a-3-5.north) + (.04,.0) $);
\draw[-]
($ (a-2-5.south) + (-.04,.0) $)
-- node [above] {}($ (a-3-5.north) + (-.04,.0) $);
\draw[-]
($ (a-2-6.south) + (.04,.0) $)
-- node [above] {}($ (a-3-6.north) + (.04,.0) $);
\draw[-]
($ (a-2-6.south) + (-.04,.0) $)
-- node [above] {}($ (a-3-6.north) + (-.04,.0) $);
\draw[-]
($ (a-3-1.east) + (.0,.04) $)
-- node [above] {}($ (a-3-2.west) + (.,.04) $);
\draw[-]
($ (a-3-1.east) + (.0,-.04) $)
-- node [above] {}($ (a-3-2.west) + (.,-.04) $);
\draw[-]
($ (a-3-2.east) + (.0,.04) $)
-- node [above] {}($ (a-3-3.west) + (.,.04) $);
\draw[-]
($ (a-3-2.east) + (.0,-.04) $)
-- node [above] {}($ (a-3-3.west) + (.,-.04) $);
\draw[-]
($ (a-3-3.east) + (.0,.04) $)
-- node [above] {}($ (a-3-4.west) + (.,.04) $);
\draw[-]
($ (a-3-3.east) + (.0,-.04) $)
-- node [above] {}($ (a-3-4.west) + (.,-.04) $);
\draw[-]
($ (a-3-5.east) + (.0,.04) $)
-- node [above] {}($ (a-3-6.west) + (.,.04) $);
\draw[-]
($ (a-3-5.east) + (.0,-.04) $)
-- node [above] {}($ (a-3-6.west) + (.,-.04) $);
\path[->] (a-3-2) edge 	node[left]{$\scriptstyle m111 $} 
										(a-4-2);
\path[->] (a-3-3) edge 	node[left]{$\scriptstyle m111 $} 
	       								(a-4-3);
\path[->] (a-3-4) edge 	node[right]{$\scriptstyle m111 $} 
	       								(a-4-4);
\path[->] (a-3-5) edge 	node[left]{$\scriptstyle m111 $} 
	       								(a-4-5);
\path[->] (a-3-6) edge 	node[right]{$\scriptstyle m111 $} 
										(a-4-6);
\draw[-]
($ (a-4-2.east) + (.0,.04) $)
-- node [above] {}($ (a-4-3.west) + (.,.04) $);
\draw[-]
($ (a-4-2.east) + (.0,-.04) $)
-- node [above] {}($ (a-4-3.west) + (.,-.04) $);
\draw[-]
($ (a-4-3.east) + (.0,.04) $)
-- node [above] {}($ (a-4-4.west) + (.,.04) $);
\draw[-]
($ (a-4-3.east) + (.0,-.04) $)
-- node [above] {}($ (a-4-4.west) + (.,-.04) $);
\draw[-]
($ (a-4-4.east) + (.0,.04) $)
-- node [above] {}($ (a-4-5.west) + (.,.04) $);
\draw[-]
($ (a-4-4.east) + (.0,-.04) $)
-- node [above] {}($ (a-4-5.west) + (.,-.04) $);
\draw[-]
($ (a-4-5.east) + (.0,.04) $)
-- node [above] {}($ (a-4-6.west) + (.,.04) $);
\draw[-]
($ (a-4-5.east) + (.0,-.04) $)
-- node [above] {}($ (a-4-6.west) + (.,-.04) $);
\node at ($ 1/2*(a-3-4) + 1/2*(a-3-5) + (0,.3)$) {$\Downarrow
  \pi 11$};
\node at ($ 1/2*(a-3-3) + 1/2*(a-1-4) + (0,.3) $) {$
  =$};
\node at ($ 1/2*(a-3-5) + 1/2*(a-1-6)+ (0,.3) $) {$
  =$};
\node at ($ 1/2*(a-4-5) + 1/2*(a-3-6) $) {$
  =$};
\node at ($ 1/2*(a-2-4) + 1/2*(a-1-5) $) {$
  =$};
\node at ($ 1/2*(a-4-3) + 1/2*(a-3-4) $) {$
  =$};
\node at ($ 1/2*(a-2-1) + 1/2*(a-1-3)  $) {$\Downarrow
  \Sigma^{-1}_{(m1)\ast (1Ti),Tm} $};
\node at ($ 1/2*(a-3-1) + 1/2*(a-2-3) $) {$ \Downarrow
  \rho 11 $};
\node at ($ 1/2*(a-4-2) + 1/2*(a-3-3) $) {$
  =$};
\puncttikz[a-2-2]{}
\end{tikzpicture}
\end{equation*}

The subdiagram formed by the interchange cell and $\rho11$ is
\begin{align*}
 ((\rho 11)\ast 1_{11Tm})\diamond  \Sigma^{-1}_{(m1)\ast (1Ti),Tm}
\puncteq{.}
\end{align*}
This is shorthand for
\begin{align*}
 M_{\mathpzc{L}}(((\rho,1)\ast (1,1_{Tm}))\diamond  \Sigma^{-1}_{(m1)\ast (1Ti),Tm})
  =M_{\mathpzc{L}}(\Sigma^{-1}_{1, Tm}\diamond ((1,1_{Tm})\ast
 (\rho,1))) 
  = M_{\mathpzc{L}}((1,1_{Tm})\ast (\rho,1)) 
\end{align*}
or $1_{111Tm}\ast (\rho 11)$,
where we have used the relation
  \begin{align}\label{interchangealphabetafg}
((1,\beta)\ast (\alpha,1))\diamond \Sigma_{f,g}\sim
\Sigma_{f',g'}\diamond ((\alpha,1)\ast (1,\beta))
  \end{align}
for interchange
cells in the Gray product 
 and the fact that
$\Sigma^{-1}_{1,Tm}$ is the identity $2$-cell cf. \secref{FGSigma}.
This means we now have the following diagram.
\begin{tikzpicture}[remember picture,overlay, 
every  node/.style={scale=0.77}]
 \tikzset{shift={(current page.center)},xshift=0cm, yshift=-7.5cm}
\matrix (a) [matrix of math nodes, row sep=3em, column sep=1em, text
height=1.5ex, text depth=0.25ex]{ x1Tx T^2x &
xTx T\eta Tx T^2x & x\mu T \eta Tx T^2x & xTx\mu_T T^3xT\eta_{T^2} &
x\mu \mu_T T^3xT\eta_{T^2} & x\mu T\mu T^3x T\eta_{T^2}\\ 
 x1Tx T^2x & 
& x\mu T \eta Tx T^2x &
xTxT^2x\mu_{T^2}T\eta_{T^2} & x\mu T^2\mu_{T^2}T\eta_{T^2} & x\mu
T^2xT\mu_T T\eta_{T^2} \\
 x1Tx T\mu &   & x\mu T\eta TxT\mu & xTxT\mu\mu_{T^2}T\eta_{T^2} &
 xTx\mu_T\mu_{T^2}T\eta_{T^2}& xTx\mu_TT\mu_T T\eta_{T^2}\\
 xTx T\mu
&  x\mu T\mu & x\mu\mu_TT^2\mu T\eta_{T^2}& x\mu T\mu\mu_{T^2}T\eta_{T^2} & x\mu
\mu_T\mu_{T^2}T\eta_{T^2} & x\mu\mu_T T\mu_T T\eta_{T^2}
 \\ };
\path[->] (a-1-1) edge 	node[above]{$\scriptstyle 1Ti11$} 
						  	(a-1-2);
\path[->] (a-1-2) edge 	node[above]{$\scriptstyle m111$} 
						  	(a-1-3);
\draw[-]
($ (a-1-3.east) + (.0,.04) $)
-- node [above] {}($ (a-1-4.west) + (.,.04) $);
\draw[-]
($ (a-1-3.east) + (.0,-.04) $)
-- node [above] {}($ (a-1-4.west) + (.,-.04) $);
\path[->] (a-1-4) edge 	node[above]{$\scriptstyle m111$} 
	       					  	(a-1-5);
\draw[-]
($ (a-1-5.east) + (.0,.04) $)
-- node [above] {}($ (a-1-6.west) + (.,.04) $);
\draw[-]
($ (a-1-5.east) + (.0,-.04) $)
-- node  [above] {}($ (a-1-6.west) + (.,-.04) $);
\draw[-]
($ (a-1-1.south) + (.04,.0) $)
-- node [above] {}($ (a-2-1.north) + (.04,.0) $);
\draw[-]
($ (a-1-1.south) + (-.04,.0) $)
-- node [above] {}($ (a-2-1.north) + (-.04,.0) $);
\draw[-]
($ (a-1-3.south) + (.04,.0) $)
-- node [above] {}($ (a-2-3.north) + (.04,.0) $);
\draw[-]
($ (a-1-3.south) + (-.04,.0) $)
-- node [above] {}($ (a-2-3.north) + (-.04,.0) $);
\draw[-]
($ (a-1-4.south) + (.04,.0) $)
-- node [above] {}($ (a-2-4.north) + (.04,.0) $);
\draw[-]
($ (a-1-4.south) + (-.04,.0) $)
-- node [above] {}($ (a-2-4.north) + (-.04,.0) $);
\draw[-]
($ (a-1-5.south) + (.04,.0) $)
-- node [above] {}($ (a-2-5.north) + (.04,.0) $);
\draw[-]
($ (a-1-5.south) + (-.04,.0) $)
-- node [above] {}($ (a-2-5.north) + (-.04,.0) $);
\draw[-]
($ (a-1-6.south) + (.04,.0) $)
-- node [above] {}($ (a-2-6.north) + (.04,.0) $);
\draw[-]
($ (a-1-6.south) + (-.04,.0) $)
-- node [above] {}($ (a-2-6.north) + (-.04,.0) $);
\draw[-]
($ (a-2-1.east) + (.0,.04) $)
-- node [above] {}($ (a-2-3.west) + (.,.04) $);
\draw[-]
($ (a-2-1.east) + (.0,-.04) $)
-- node [above] {}($ (a-2-3.west) + (.,-.04) $);
\path[->] (a-2-4) edge 	node[below]{$\scriptstyle m111$} (a-2-5);
\path[->] (a-2-1) edge 	node[left]{$\scriptstyle 11Tm$} 
						  	(a-3-1);
\path[->] (a-2-3) edge 	node[right]{$\scriptstyle 111Tm$} 
										(a-3-3);
\path[->] (a-2-4) edge 	node[left]{$\scriptstyle 1Tm11 $} 
										(a-3-4);
\draw[-]
($ (a-2-5.south) + (.04,.0) $)
-- node [above] {}($ (a-3-5.north) + (.04,.0) $);
\draw[-]
($ (a-2-5.south) + (-.04,.0) $)
-- node [above] {}($ (a-3-5.north) + (-.04,.0) $);
\draw[-]
($ (a-2-6.south) + (.04,.0) $)
-- node [above] {}($ (a-3-6.north) + (.04,.0) $);
\draw[-]
($ (a-2-6.south) + (-.04,.0) $)
-- node [above] {}($ (a-3-6.north) + (-.04,.0) $);
\draw[-]
($ (a-3-1.east) + (.0,.04) $)
-- node [above] {}($ (a-3-3.west) + (.,.04) $);
\draw[-]
($ (a-3-1.east) + (.0,-.04) $)
-- node [above] {}($ (a-3-3.west) + (.,-.04) $);
\draw[-]
($ (a-3-3.east) + (.0,.04) $)
-- node [above] {}($ (a-3-4.west) + (.,.04) $);
\draw[-]
($ (a-3-3.east) + (.0,-.04) $)
-- node [above] {}($ (a-3-4.west) + (.,-.04) $);
\draw[-]
($ (a-3-5.east) + (.0,.04) $)
-- node [above] {}($ (a-3-6.west) + (.,.04) $);
\draw[-]
($ (a-3-5.east) + (.0,-.04) $)
-- node [above] {}($ (a-3-6.west) + (.,-.04) $);
\draw[-]
($ (a-3-1.south) + (.04,.0) $)
-- node [above] {}($ (a-4-1.north) + (.04,.0) $);
\draw[-]
($ (a-3-1.south) + (-.04,.0) $)
-- node [above] {}($ (a-4-1.north) + (-.04,.0) $);
\path[->] (a-3-4) edge 	node[right]{$\scriptstyle m111 $} 
	       								(a-4-4);
\path[->] (a-3-5) edge 	node[left]{$\scriptstyle m111 $} 
	       								(a-4-5);
\path[->] (a-3-6) edge 	node[right]{$\scriptstyle m111 $} 
										(a-4-6);
\path[->] (a-4-1) edge 	node[above]{$\scriptstyle m1 $} 
										(a-4-2);
\draw[-]
($ (a-4-2.east) + (.0,.04) $)
-- node [above] {}($ (a-4-3.west) + (.,.04) $);
\draw[-]
($ (a-4-2.east) + (.0,-.04) $)
-- node [above] {}($ (a-4-3.west) + (.,-.04) $);
\draw[-]
($ (a-4-3.east) + (.0,.04) $)
-- node [above] {}($ (a-4-4.west) + (.,.04) $);
\draw[-]
($ (a-4-3.east) + (.0,-.04) $)
-- node [above] {}($ (a-4-4.west) + (.,-.04) $);
\draw[-]
($ (a-4-4.east) + (.0,.04) $)
-- node [above] {}($ (a-4-5.west) + (.,.04) $);
\draw[-]
($ (a-4-4.east) + (.0,-.04) $)
-- node [above] {}($ (a-4-5.west) + (.,-.04) $);
\draw[-]
($ (a-4-5.east) + (.0,.04) $)
-- node [above] {}($ (a-4-6.west) + (.,.04) $);
\draw[-]
($ (a-4-5.east) + (.0,-.04) $)
-- node [above] {}($ (a-4-6.west) + (.,-.04) $);
\node at ($ 1/2*(a-3-4) + 1/2*(a-3-5)+ (0,.3) $) {$\Downarrow
  \pi 11$};
\node at ($ 1/2*(a-3-3) + 1/2*(a-1-4) + (0,.3) $) {$
  =$};
\node at ($ 1/2*(a-3-5) + 1/2*(a-1-6) + (0,.3) $) {$
  =$};
\node at ($ 1/2*(a-4-5) + 1/2*(a-3-6) $) {$
  =$};
\node at ($ 1/2*(a-2-4) + 1/2*(a-1-5) $) {$
  =$};
\node at ($ 1/2*(a-4-1) + 1/2*(a-3-4) $) {$
  =$};
\node at ($ 1/2*(a-3-1) + 1/2*(a-2-3)  $) {$ = 
  \ (\Sigma^{-1}_{1,Tm}= 1)$};
\node at ($ 1/2*(a-2-1) + 1/2*(a-1-3) $) {$ \Downarrow
  \rho 11 $};
\puncttikz[a-2-2]{}
\end{tikzpicture}

\vfill
\newpage
\noindent Slightly rewritten, this is the same as the diagram below.
  \begin{tikzpicture}[overlay, remember picture, 
every  node/.style={scale=0.77}]
\tikzset{shift={(current page.center)},xshift=0cm, yshift=6.75cm}
\matrix (a) [matrix of math nodes, row sep=3em, column sep=1em, text
height=1.5ex, text depth=0.25ex]{ x1Tx T^2x &
xTx T\eta Tx T^2x & x\mu T \eta Tx T^2x & xTx\mu_T T^3xT\eta_{T^2} &
x\mu \mu_T T^3xT\eta_{T^2} & x\mu  T^2x \\ 
 x1Tx T^2x & 
& x\mu T \eta Tx T^2x &
xTxT^2x & xTxT^2x T\mu_T T\eta_{T^2}  & x\mu
T^2xT\mu_T T\eta_{T^2} \\
  &   & &  &
x Tx T\mu T\mu_T T\eta_{T^2} & xTx\mu_TT\mu_T T\eta_{T^2}\\
&  & &  & x\mu
T\mu T\mu_T T\eta_{T^2} & x\mu\mu_T T\mu_T T\eta_{T^2}
 \\ };
\path[->] (a-1-1) edge 	node[above]{$\scriptstyle 1Ti11$} 
						  	(a-1-2);
\path[->] (a-1-2) edge 	node[above]{$\scriptstyle m111$} 
						  	(a-1-3);
\draw[-]
($ (a-1-3.east) + (.0,.04) $)
-- node [above] {}($ (a-1-4.west) + (.,.04) $);
\draw[-]
($ (a-1-3.east) + (.0,-.04) $)
-- node [above] {}($ (a-1-4.west) + (.,-.04) $);
\path[->] (a-1-4) edge 	node[above]{$\scriptstyle m111$} 
	       					  	(a-1-5);
\draw[-]
($ (a-1-5.east) + (.0,.04) $)
-- node [above] {}($ (a-1-6.west) + (.,.04) $);
\draw[-]
($ (a-1-5.east) + (.0,-.04) $)
-- node  [above] {}($ (a-1-6.west) + (.,-.04) $);
\draw[-]
($ (a-1-1.south) + (.04,.0) $)
-- node [above] {}($ (a-2-1.north) + (.04,.0) $);
\draw[-]
($ (a-1-1.south) + (-.04,.0) $)
-- node [above] {}($ (a-2-1.north) + (-.04,.0) $);
\draw[-]
($ (a-1-3.south) + (.04,.0) $)
-- node [above] {}($ (a-2-3.north) + (.04,.0) $);
\draw[-]
($ (a-1-3.south) + (-.04,.0) $)
-- node [above] {}($ (a-2-3.north) + (-.04,.0) $);
\draw[-]
($ (a-1-4.south) + (.04,.0) $)
-- node [above] {}($ (a-2-4.north) + (.04,.0) $);
\draw[-]
($ (a-1-4.south) + (-.04,.0) $)
-- node [above] {}($ (a-2-4.north) + (-.04,.0) $);
\draw[-]
($ (a-1-6.south) + (.04,.0) $)
-- node [above] {}($ (a-2-6.north) + (.04,.0) $);
\draw[-]
($ (a-1-6.south) + (-.04,.0) $)
-- node [above] {}($ (a-2-6.north) + (-.04,.0) $);
\draw[-]
($ (a-2-1.east) + (.0,.04) $)
-- node [above] {}($ (a-2-3.west) + (.,.04) $);
\draw[-]
($ (a-2-1.east) + (.0,-.04) $)
-- node [above] {}($ (a-2-3.west) + (.,-.04) $);
\draw[-]
($ (a-2-3.east) + (.0,.04) $)
-- node [above] {}($ (a-2-4.west) + (.,.04) $);
\draw[-]
($ (a-2-3.east) + (.0,-.04) $)
-- node [above] {}($ (a-2-4.west) + (.,-.04) $);
\draw[-]
($ (a-2-4.east) + (.0,.04) $)
-- node [above] {}($ (a-2-5.west) + (.,.04) $);
\draw[-]
($ (a-2-4.east) + (.0,-.04) $)
-- node [above] {}($ (a-2-5.west) + (.,-.04) $);
\path[->] (a-2-5) edge 	node[below]{$\scriptstyle m111$} (a-2-6);
\path[->] (a-2-5) edge 	node[left]{$\scriptstyle 1Tm11 $} 
										(a-3-5);
\draw[-]
($ (a-2-6.south) + (.04,.0) $)
-- node [above] {}($ (a-3-6.north) + (.04,.0) $);
\draw[-]
($ (a-2-6.south) + (-.04,.0) $)
-- node [above] {}($ (a-3-6.north) + (-.04,.0) $);
\path[->] (a-3-5) edge 	node[left]{$\scriptstyle m111 $} 
	       								(a-4-5);
\path[->] (a-3-6) edge 	node[right]{$\scriptstyle m111 $} 
										(a-4-6);
\draw[-]
($ (a-4-5.east) + (.0,.04) $)
-- node [above] {}($ (a-4-6.west) + (.,.04) $);
\draw[-]
($ (a-4-5.east) + (.0,-.04) $)
-- node [above] {}($ (a-4-6.west) + (.,-.04) $);
\node at ($ 1/2*(a-2-4) + 1/2*(a-1-6) $) {$
  =$};
\node at ($ 1/2*(a-2-3) + 1/2*(a-1-4) $) {$
  =$};
\node at ($ 1/2*(a-4-5) + 1/2*(a-2-6) + (0,.3)$) {$
  \Downarrow \pi 1 1$};
\node at ($ 1/2*(a-2-1) + 1/2*(a-1-3) $) {$ \Downarrow
  \rho 11 $};
\puncttikz[a-2-2]{}
\end{tikzpicture}

\vspace{7.25cm}
%

\noindent For the upper right entry we have used another identity to make
commutativity obvious. Finally, by another instance of the triangle identity---in fact its image under
$\mathpzc{L}(T^2x,X)$---we end up with the diagram below. It is
easily seen to be the diagrammatic form of the left hand side of
\eqref{lefthandrighthand}, so this ends the proof.

\begin{tikzpicture}[
remember picture,overlay, 
every  node/.style={scale=0.77}]]
\tikzset{shift={(current page.center)},xshift=0cm,yshift=-3.25cm}
\matrix (a) [matrix of math nodes, row sep=3em, column sep=1.5em, text
height=1.5ex, text depth=0.25ex]{ x1Tx T^2x &
xTx T\eta Tx T^2x & xTx T^2x T\eta_T T^2x & x\mu T^2x T\eta_T T^2x &
x Tx \mu_T T\eta_T T^2x \\ 
 x1Tx T^2x & 
&  xTx T\mu T\eta_T  T^2x 
   &  x\mu T\mu T\eta_T  T^2x &   x\mu \mu_T T\eta_T  T^2x  \\
 &   & xTxT^2x  & 
xTxT^2x T\mu_T T\eta_{T^2}  & x\mu
T^2xT\mu_T T\eta_{T^2}\\
&   &   & 
x Tx T\mu T\mu_T T\eta_{T^2} & xTx\mu_TT\mu_T T\eta_{T^2}
 \\
& & & x\mu T\mu T\mu_T T\eta_{T^2} & x\mu\mu_T T\mu_T T\eta_{T^2}
 \\ };
\path[->] (a-1-1) edge 	node[above]{$\scriptstyle 1Ti11$} 
						  	(a-1-2);
\draw[-]
($ (a-1-2.east) + (.0,.04) $)
-- node [above] {}($ (a-1-3.west) + (.,.04) $);
\draw[-]
($ (a-1-2.east) + (.0,-.04) $)
-- node [above] {}($ (a-1-3.west) + (.,-.04) $);
\path[->] (a-1-3) edge 	node[above]{$\scriptstyle m111$} 
						  	(a-1-4);
\draw[-]
($ (a-1-4.east) + (.0,.04) $)
-- node [above] {}($ (a-1-5.west) + (.,.04) $);
\draw[-]
($ (a-1-4.east) + (.0,-.04) $)
-- node [above] {}($ (a-1-5.west) + (.,-.04) $);
\draw[-]
($ (a-1-1.south) + (.04,.0) $)
-- node [above] {}($ (a-2-1.north) + (.04,.0) $);
\draw[-]
($ (a-1-1.south) + (-.04,.0) $)
-- node [above] {}($ (a-2-1.north) + (-.04,.0) $);				  
\path[->] (a-1-3) edge 	node[below left]{$\scriptstyle 1Tm11$} 
	(a-2-3);
\path[->] (a-1-5) edge 	node[right]{$\scriptstyle m111$} 
	(a-2-5);
\draw[-]
($ (a-2-1.east) + (.0,.04) $)
-- node [above] {}($ (a-2-3.west) + (.,.04) $);
\draw[-]
($ (a-2-1.east) + (.0,-.04) $)
-- node [above] {}($ (a-2-3.west) + (.,-.04) $);
\path[->] (a-2-3) edge 	node[below]{$\scriptstyle m111$} (a-2-4);
\draw[-]
($ (a-2-4.east) + (.0,.04) $)
-- node [above] {}($ (a-2-5.west) + (.,.04) $);
\draw[-]
($ (a-2-4.east) + (.0,-.04) $)
-- node [above] {}($ (a-2-5.west) + (.,-.04) $);
\draw[-]
($ (a-2-3.south) + (.04,.0) $)
-- node [above] {}($ (a-3-3.north) + (.04,.0) $);
\draw[-]
($ (a-2-3.south) + (-.04,.0) $)
-- node [above] {}($ (a-3-3.north) + (-.04,.0) $);
\draw[-]
($ (a-2-5.south) + (.04,.0) $)
-- node [above] {}($ (a-3-5.north) + (.04,.0) $);
\draw[-]
($ (a-2-5.south) + (-.04,.0) $)
-- node [above] {}($ (a-3-5.north) + (-.04,.0) $);
\draw[-]
($ (a-3-3.east) + (.0,.04) $)
-- node [above] {}($ (a-3-4.west) + (.,.04) $);
\draw[-]
($ (a-3-3.east) + (.0,-.04) $)
-- node [above] {}($ (a-3-4.west) + (.,-.04) $);
\path[->] (a-3-4) edge 	node[below]{$\scriptstyle m111$} (a-3-5);
\path[->] (a-3-4) edge 	node[left]{$\scriptstyle 1Tm11 $} 
										(a-4-4);
\draw[-]
($ (a-3-5.south) + (.04,.0) $)
-- node [above] {}($ (a-4-5.north) + (.04,.0) $);
\draw[-]
($ (a-3-5.south) + (-.04,.0) $)
-- node [above] {}($ (a-4-5.north) + (-.04,.0) $);
\path[->] (a-4-4) edge 	node[right]{$\scriptstyle m111 $} 
	       								(a-5-4);
\path[->] (a-4-5) edge 	node[left]{$\scriptstyle m111 $} 
	       								(a-5-5);
\draw[-]
($ (a-5-4.east) + (.0,.04) $)
-- node [above] {}($ (a-5-5.west) + (.,.04) $);
\draw[-]
($ (a-5-4.east) + (.0,-.04) $)
-- node [above] {}($ (a-5-5.west) + (.,-.04) $);
\node at ($ 1/2*(a-2-3) + 1/2*(a-1-5) $) {$
  \Downarrow \pi 1 1$};
\node at ($ 1/2*(a-5-4) + 1/2*(a-3-5) + (0,.3)$) {$
  \Downarrow \pi 1 1$};
\node at ($ 1/2*(a-3-3) + 1/2*(a-2-5) $) {$
  =$};
\node at ($ 1/2*(a-2-1) + 1/2*(a-1-3) $) {$ \Downarrow
  1T\lambda1 $};
\puncttikz[a-2-2]{}
\end{tikzpicture}

\vspace{8.5cm}

   \end{proof}

\begin{definition}\cite[Def. 13.6 and Def. 13.9]{gurskicoherencein}
  A pseudo $T$-functor
  \begin{align*}
    (f,F,\mathpzc{h},\mathpzc{m})\co(X,x,m^X,i^X,\pi^X,\lambda^X,\rho^X)\rightarrow
    (Y,y,m^Y,i^Y,\pi^Y,\lambda^Y,\rho^Y)
  \end{align*}
consists of
\begin{itemize}
\item a $1$-cell $f\co X\rightarrow Y$ in
  $\mathpzc{K}$ i.e. an object of $\mathpzc{K}(X,Y)$;
\item a $2$-cell adjoint equivalence $(F,F^\bullet)\co fx\rightarrow yTf$ i.e. a
  $1$-cell adjoint equivalence internal to $\mathpzc{K}(TX,Y)$ ;
\item  and two  invertible $3$-cells $\mathpzc{h},\mathpzc{m}$ as in \textbf{(PSF1)}-\textbf{(PSF2)} subject
  to the three axioms \textbf{(LFA1)}-\textbf{(LFA3)} of a lax $T$-functor:
\end{itemize}

\noindent\textbf{(PSF1)}\; An invertible $3$-cell $\mathpzc{h}$ given by an
invertible 
  $2$-cell in $\mathpzc{K}(X,Y)$: 
  \begin{align*}
    \mathpzc{h}\co (F1_{\eta_X})\ast (1_f i^X)\Rightarrow (i^Y 1_f)
  \end{align*}
where the codomains match by $\mathpzc{Gray}$-naturality of $\eta$.

\noindent\textbf{(PSF2)}\; An invertible $3$-cell $\mathpzc{m}$ given by an
invertible $2$-cell in  $\mathpzc{K}(T^2X,Y)$:
\begin{align*}
  \mathpzc{m}\co (m^Y 1_{T^2f})\ast (1_yTF)\ast (F1_x)\Rightarrow
  (F1_\mu)\ast (1_fm)
\end{align*}
where the codomains match by $\mathpzc{Gray}$-naturality of $\mu$.

The three lax $T$-functor axioms are:

\noindent\textbf{(LFA1)}\; The following equation in $\mathpzc{K}(T^3X,Y)$ of vertical
composites of whiskered $3$-cells is required:
\begin{align*}
  & (1_{F11}\ast(1\pi^X))\diamond ((\mathpzc{m}1)\ast 1_{1m^X1})\diamond
  (1_{m^Y11}\ast 1_{11T^2F}\ast (\mathpzc{m}1))\diamond (1_{m^Y11}
  \ast \Sigma_{m^Y,T^2F}\ast 1_{1TF1}\ast 1_{F11}) \\
    & = ((\mathpzc{m}1)\ast 1_{11Tm^X})\diamond
  (1_{m^Y11}\ast1_{1TF1}\ast\Sigma^{-1}_{F,Tm^X})\diamond
  (1T\mathpzc{m})\diamond((\pi^Y 1)\ast 1_{11T^2F}\ast 1_{TF1}\ast
1_{F11})
\puncteq{.}
\end{align*}

\noindent\textbf{(LFA2)}\; The following equation in $\mathpzc{K}(TX,Y)$ of vertical
composites of whiskered $3$-cells is required:
\begin{align*}
  & ((\lambda^Y 1)\ast 1_{1F})\diamond(1_{m^Y11}\ast
  \Sigma^{-1}_{i^Y,F})\diamond (1_{m^Y11}\ast 1_{11F}\ast
  (\mathpzc{h}1))\\
   & = (1_{F}\ast (1\lambda^Y))\diamond ((\mathpzc{m}1)\ast
  1_{1i^X1})
\puncteq{.}
\end{align*}

\noindent\textbf{(LFA3)}\; The following equation in $\mathpzc{K}(TX,Y)$ of vertical
composites of whiskered $3$-cells is required:
\begin{align*}
  &  ((\rho^Y1)\ast 1_{F1})\diamond  (1_{m^Y11}\ast
  (1T\mathpzc{h})\ast 1_{F1}) \\
  &= (1_{F}\ast (1\rho^X))\diamond ((\mathpzc{m}1)\ast
  1_{1Ti^X})\diamond (1_{m^Y1}\ast 1_{1TF1}\ast
\Sigma^{-1}_{F,Ti^X})
\puncteq{.}
\end{align*}
A careful inspection shows that the horizontal and vertical factors do
indeed compose in all of these axioms. Diagrams may be found in
Gurski's definition.
\end{definition}
\begin{definition}\label{Ttransformation}\cite[Def. 13.6 and Def. 13.10]{gurskicoherencein}
  A $T$-transformation
  \begin{align*}
    (\alpha,A)\co (f,F,\mathpzc{h}^f,\mathpzc{m}^f)\Rightarrow
    (g,G,\mathpzc{h}^g,\mathpzc{m}^g)\co
    (X,x,m^X,i^X,\pi^X,\lambda^X,\rho^X)\rightarrow
    (Y,y,m^Y,i^Y,\pi^Y,\lambda^Y,\rho^Y)
  \end{align*}
consists of
\begin{itemize}
\item a $2$-cell $\alpha\co f\Rightarrow g$ i.e. an object of
  $\mathpzc{K}(X,Y)$;
\item an invertible $3$-cell $A$ as in \textbf{(T1)} subject to the two axioms \textbf{(LTA1)}-\textbf{(LTA2)}
  of a lax \mbox{$T$-algebra}:
\end{itemize}

\noindent\textbf{(T1)}\; An
 invertible $3$-cell $A$ given by an invertible $2$-cell in $\mathpzc{K}(TX,Y)$:
 \begin{align*}
   A\co (1_y T\alpha)\ast F\Rightarrow G\ast (\alpha 1_x)
\puncteq{.}
 \end{align*}

The two lax $T$-transformation axioms are:


\noindent\textbf{(LTA1)}\; The following equation in $\mathpzc{K}(X,Y)$ of vertical
composites of whiskered $3$-cells is required:
\begin{align*}
  & (\mathpzc{h}^g\ast 1_{\alpha 1})\diamond (1_{G1}\ast
  \Sigma_{\alpha,i^X})\diamond ((A1)\ast 1_{1i^X}) 
  = \Sigma^{-1}_{i^Y,\alpha}\diamond (1_{1T\alpha 1}\ast \mathpzc{h}^f)
  \puncteq{.}
\end{align*}

\noindent\textbf{(LTA2)}\; The following equation in $\mathpzc{K}(T^2X,X)$ of vertical
composites of whiskered $3$-cells is required:
\begin{align*}
  & (\mathpzc{m}^g\ast 1_{\alpha 1})\diamond (1_{m^Y1}\ast
  1_{1TG}\ast (A1))\diamond (1_{m^Y1}\ast (1TA)\ast
  1_{F1})\diamond (\Sigma^{-1}_{m^Y,T^2\alpha}\ast 1_{1TF}\ast
  1_{F1}) \\
  & = (1_{G1}\ast \Sigma_{\alpha,m^X})\diamond ((A1)\ast
  1_{1m^X})\diamond (1_{1T^2\alpha}\ast \mathpzc{m}^f)
\puncteq{.}
\end{align*}
A careful inspection shows that the horizontal and vertical factors do
indeed compose in the two axioms. Diagrams may be found in
Gurski's definition.
\end{definition}
\begin{definition}
  A $T$-modification $\Gamma\co (\alpha,A)\Rrightarrow (\beta,B)$ of
  $T$-transformations 
  \begin{align*}
(f,F,\mathpzc{h}^f,\mathpzc{m}^f)\Rightarrow
    (g,G,\mathpzc{h}^g,\mathpzc{m}^g)\co
    (X,x,m^X,i^X,\pi^X,\lambda^X,\rho^X)\rightarrow
    (Y,y,m^Y,i^Y,\pi^Y,\lambda^Y,\rho^Y)
  \end{align*}
consists of a 
\begin{itemize}
\item $3$-cell $\Gamma\co \alpha\Rrightarrow \beta$ i.e. a
$2$-cell in $\mathpzc{K}(X,Y)$; 
\item subject to one axiom \textbf{(MA1)}:
\end{itemize}

\noindent\textbf{(MA1)}\; The following equation in $\mathpzc{K}(TX,Y)$ of vertical
composites of whiskered $3$-cells is required:
\begin{align*}
  B\diamond ((1 T\Gamma)\ast 1_{F})= (1_{G}\ast (\Gamma
  1))\diamond A
\puncteq{.}
\end{align*}
\end{definition}
Finally, we provide the $\mathpzc{Gray}$-category structure of
$\mathrm{Ps}\text{-}T\text{-}\mathrm{Alg}$. We begin with its hom
\mbox{$2$-categories}.
\begin{definition}\label{localPSTAlg}
Given $T$-algebras $ (X,x,m^X,i^X,\pi^X,\lambda^X,\rho^X)$ and
$(Y,y,m^Y,i^Y,\pi^Y,\lambda^Y,\rho^Y)$, the prescriptions below give
the $2$-globular set  
$\mathrm{Ps}\text{-}T\text{-}\mathrm{Alg}(X,Y)$ whose objects are
pseudo $T$-functors from $X$ to $Y$, whose $1$-cells are
$T$-transformations between pseudo $T$-functors, and whose $2$-cells
are $T$-modifications between those, the structure of a $2$-category \cite[Prop. 13.11]{gurskicoherencein}.

Given $T$-modifications $\Gamma\co (\alpha,A)\Rrightarrow (\beta,B)$
and $\Delta\co (\beta,B)\Rrightarrow (\epsilon,E)$ of
$T$-transformations
\begin{align*}
  (f,F,\mathpzc{h}^f,\mathpzc{m}^f)\Rightarrow
    (g,G,\mathpzc{h}^g,\mathpzc{m}^g)\co
    (X,x,m^X,i^X,\pi^X,\lambda^X,\rho^X)\rightarrow
    (Y,y,m^Y,i^Y,\pi^Y,\lambda^Y,\rho^Y)
\puncteq{,}
\end{align*}
their vertical composite $\Delta\diamond \Gamma$ is defined by the
vertical composite $\Delta\diamond \Gamma$ of $2$-cells in
$\mathpzc{K}(X,Y)$.
The identity $T$-modification of $(\alpha,A)$ as above is defined by
the $2$-cell $1_\alpha$ in $\mathpzc{K}(X,Y)$.

Given $T$-transformations $(\alpha,A)\co (f,F,\mathpzc{h}^f,\mathpzc{m}^f)\Rightarrow
    (g,G,\mathpzc{h}^g,\mathpzc{m}^g)$ and $(\beta,B)\co 
    (g,G,\mathpzc{h}^g,\mathpzc{m}^g)\Rightarrow
    (h,H,\mathpzc{h}^h,\mathpzc{m}^h)$ of pseudo $T$-functors 
    \begin{align*}
       (X,x,m^X,i^X,\pi^X,\lambda^X,\rho^X)\rightarrow
    (Y,y,m^Y,i^Y,\pi^Y,\lambda^Y,\rho^Y)
\puncteq{,}
    \end{align*}
their horizontal composite $(\beta,B)\ast (\alpha,A)$ is defined by
\begin{align*}
  (\beta\ast \alpha, (B\ast 1_{\alpha 1})\diamond (1_{1T\beta} A))
\puncteq{.}
\end{align*}
The identity $T$-transformation of $(f,F,\mathpzc{h}^f,\mathpzc{m}^f)$ is defined by $(1_f,1_F)$.

Given $T$-modifications $\Gamma\co (\alpha,A)\Rrightarrow
(\alpha',A')\co (f,F,\mathpzc{h}^f,\mathpzc{m}^f)\Rightarrow
    (g,G,\mathpzc{h}^g,\mathpzc{m}^g)$ and $\Delta\co
    (\beta,B)\Rrightarrow (\beta',B')\co 
    (g,G,\mathpzc{h}^g,\mathpzc{m}^g)\Rightarrow
    (h,H,\mathpzc{h}^h,\mathpzc{m}^h$ of pseudo $T$-functors 
 \begin{align*}
       (X,x,m^X,i^X,\pi^X,\lambda^X,\rho^X)\rightarrow
    (Y,y,m^Y,i^Y,\pi^Y,\lambda^Y,\rho^Y)
\puncteq{,}
    \end{align*}
their horizontal composite is defined by the horizontal composite
$\Delta\ast\Gamma$ of $2$-cells in $\mathpzc{K}(X,Y)$.

We omit the proof that this is indeed a $2$-category.
\end{definition}

\begin{definition}\label{globalPSTAlg}
The prescriptions below give the set of pseudo $T$-algebras with the
hom $2$-categories from the proposition above, the structure of  a
$\mathpzc{Gray}$-category denoted
$\mathrm{Ps}\text{-}T\text{-}\mathrm{Alg}$, see \cite[Prop. 13.12 and Th. 13.13]{gurskicoherencein}.

Given pseudo $T$-algebras $ (X,x,m^X,i^X,\pi^X,\lambda^X,\rho^X)$,
$(Y,y,m^Y,i^Y,\pi^Y,\lambda^Y,\rho^Y)$ and
$(Z,z,m^Z,i^Z,\pi^Z,\lambda^Z,\rho^Z)$, the composition law is defined
by the strict
functor 
\begin{align*}
  \boxtimes\co \mathrm{Ps}\text{-}T\text{-}\mathrm{Alg}(Y,Z)\otimes
  \mathrm{Ps}\text{-}T\text{-}\mathrm{Alg}(X,Y)\rightarrow
  \mathrm{Ps}\text{-}T\text{-}\mathrm{Alg}(X,Z)
\end{align*}
specified as follows.

On an object $(g,f)$ in $\mathrm{Ps}\text{-}T\text{-}\mathrm{Alg}(Y,Z)\otimes
\mathrm{Ps}\text{-}T\text{-}\mathrm{Alg}(X,Y)$ i.e. on functors
$(g,G,\mathpzc{h}^g,\mathpzc{m}^g)$ and $
(f,F,\mathpzc{h}^f,\mathpzc{m}^f)$, $\boxtimes$ is defined by
\begin{align*}
  &  \bigg(gf,(G1_{Tf})\ast(1_gF), (\mathpzc{h}^g 1)\diamond
  (1_{G11}\ast (1\mathpzc{h}^f)), \\ & \quad
  (1_{G11}\ast
  (1\mathpzc{m}^f))\diamond ((\mathpzc{m}^g1)\ast 1_{11TF}\ast
  1_{G11}\ast 1_{1F1}) \diamond
  (1_{m^Y11}\ast 1_{1TG1}\ast\Sigma^{-1}_{G,TF}\ast1_{1F1} )\bigg)
\puncteq{,}
\end{align*}
which we denote by $g\boxtimes f$.

On a generating $1$-cell of the form $((\alpha,A),1)\co
(g,f)\rightarrow (g',f)$ in the Gray product $\mathrm{Ps}\text{-}T\text{-}\mathrm{Alg}(Y,Z)\otimes
\mathrm{Ps}\text{-}T\text{-}\mathrm{Alg}(X,Y)$,
where $(\alpha,A)$ is a $T$-transformation 
\begin{align*}
  (g,G,\mathpzc{h}^g,\mathpzc{m}^g)\Rightarrow (g',G',\mathpzc{h}^{g'} ,\mathpzc{m}^{g'}),
\end{align*}
and $f$ is as above,
$\boxtimes$ is defined by
\begin{align*}
  & ( \alpha1_f,(1_{G'1}\ast\Sigma_{\alpha,F})\diamond ((A1)\ast
  1_{1F})) 
\end{align*}
and denoted $\alpha\boxtimes 1_f$.

Similarly, on a generating $1$-cell of the form $(1,(\beta,B))\co
(g,f)\rightarrow (g,f')$ in the Gray product $\mathrm{Ps}\text{-}T\text{-}\mathrm{Alg}(Y,Z)\otimes
\mathrm{Ps}\text{-}T\text{-}\mathrm{Alg}(X,Y)$,
where $(\beta,B)$ is a $T$-transformation 
\begin{align*}
  (f,F,\mathpzc{h}^f,\mathpzc{m}^f)\Rightarrow (f',F',\mathpzc{h}^{f'} ,\mathpzc{m}^{f'}),
\end{align*}
and $g$ is as above,
$\boxtimes$ is defined by
\begin{align*}
  & ( 1_g\beta,(1_{G1}\ast(1B))\diamond
  (\Sigma^{-1}_{G,\beta}\ast 1_{1F})) 
\end{align*}
and denoted $\alpha\boxtimes 1_f$.

On a generating $2$-cell of the form $(\Gamma,1)\co
((\alpha,A),1)\Rrightarrow ((\alpha',A'),1)$, $\boxtimes$ is defined
by the underlying $2$-cell $\Gamma 1_f$ in $\mathpzc{K}(X,Y)$ and
denoted by $\Gamma\boxtimes 1$, and
similarly for $2$-cells of the form $(1,\Delta)\co
(1,(\beta,B))\Rrightarrow (1,(\beta',B'))$.

Finally, on an interchange cell $\Sigma_{(\alpha,A),(\beta,B)}$ in $\mathrm{Ps}\text{-}T\text{-}\mathrm{Alg}(Y,Z)\otimes
\mathrm{Ps}\text{-}T\text{-}\mathrm{Alg}(X,Y)$, $\boxtimes$ is defined
by
the $2$-cell $M_{\mathpzc{K}}(\Sigma_{\alpha,\beta})$, the shorthand of
which is $\Sigma_{\alpha,\beta}$.

The unit at an object
$(X,x,m,i,\pi,\lambda,\rho)$, that is,
the functor $j_X\co I\rightarrow
\mathrm{Ps}\text{-}T\text{-}\mathrm{Alg}(X,X)$ is determined by strictness and the requirement that
it sends the unique object $\ast$ of $I$ to the $T$-functor $(1_X,1_x,1_{i},1_m)$.

We omit the proof that this is well-defined and that
$\mathrm{Ps}\text{-}T\text{-}\mathrm{Alg}$ is indeed a
$\mathpzc{Gray}$-category.
\end{definition}

\subsection{Coherence via codescent}
Recall from \cite[3.1]{kelly} that given a complete and cocomplete locally small symmetric
monoidal closed category $\mathpzc{V}$, $\mathpzc{V}$-categories
$\mathpzc{K}$ and $\mathpzc{B}$, and $\mathpzc{V}$-functors $F\co
\mathpzc{K}\op\rightarrow \mathpzc{V}$ and $\mathpzc{G}\co
\mathpzc{K}\rightarrow \mathpzc{B}$, the colimit of $G$ indexed by $F$
is a representation $(F\ast G,\nu)$ of the $\mathpzc{V}$-functor $\lbrack
  \mathpzc{K}\op,\mathpzc{V}\rbrack(F-,\mathpzc{B}( G-,  ?))\co
  \mathpzc{B}\rightarrow \mathpzc{V}$ (where
  this is assumed to exist) with representing object $F\ast G$ in
  $\mathpzc{B}$ and unit $\nu\co F\rightarrow \mathpzc{B}(G-,F\ast
  G)$. 
For the concept of representable functors see \cite[1.10]{kelly}. In
particular, there is a $\mathpzc{V}$-natural (in $B$) isomorphism
\begin{align}\label{repcolimit}
   \mathpzc{B}(F\ast G, B ) \xrightarrow{\quad\cong\quad}
  \lbrack \mathpzc{K}\op,\mathpzc{V}\rbrack(F,\mathpzc{B}( G-,  B)) 
\puncteq{,}
\end{align}
and the unit is obtained by Yoneda when this is composed with the unit
$j_{F\ast G}$ of the \mbox{$\mathpzc{V}$-category} $\mathpzc{B}$ at the
object $F\ast G$.

\begin{definition}
  A $\mathpzc{V}$-functor $T\co \mathpzc{B}\rightarrow \mathpzc{C}$
  preserves the  colimit of $G\co
  \mathpzc{K}\rightarrow \mathpzc{B}$  indexed by  $F\co
  \mathpzc{K}\op\rightarrow \mathpzc{V}$ if when $(F\ast G,\nu)$ exists,
  the composite
  \begin{align*}
    T_{G-,B}\nu\co F\rightarrow \mathpzc{B}(G-,F\ast G)\rightarrow
    \mathpzc{L}(TG-,T\{F,G\}) 
  \end{align*}
exhibits $T(F\ast G)$ as
  the colimit of $TG$ indexed by $F$ i.e. the composite corresponds
  under Yoneda to an isomorphism as in \eqref{repcolimit} with $TG$
  instead of $G$.
\end{definition}
Now let again $\mathpzc{K}$ be a $\mathpzc{Gray}$-category and $T$ be
a $\mathpzc{Gray}$-monad on it. Below we will make use of the
following corollary of the central coherence theorem from three-dimensional monad
theory \cite[Corr. 15.14]{gurskicoherencein}.
\begin{theorem}[Gurski's coherence theorem]\label{gurskicoherence}
  Assume that $\mathpzc{K}$ has codescent objects of codescent
  diagrams, and that $T$ preserves them. Then the inclusion $i\co
  \mathpzc{K}^T\hookrightarrow
  \mathrm{Ps}\text{-}T\text{-}\mathrm{Alg}$ has a left adjoint $L\co
  \mathrm{Ps}\text{-}T\text{-}\mathrm{Alg}\rightarrow \mathpzc{K}^T$ and
  each component $\eta_X\co X\rightarrow iLX$ of the unit of this
  adjunction is a biequivalence in $\mathrm{Ps}\text{-}T\text{-}\mathrm{Alg}$.
\end{theorem}

\begin{remark}
  Codescent objects are certain indexed colimits, see
  \cite[12.3]{gurskicoherencein}. In fact, they are
built from co-$2$-inserters, co-$3$-inserters and coequifiers. These
are classes of indexed colimits  where each 
of these classes is determined separately by considering all indexed
colimits $F\ast G$ with a particular fixed  $\mathpzc{Gray}$-functor
$F\co \mathpzc{K}\op\rightarrow \mathpzc{V}$.
 Thus there
 is no other restriction
on $G$ apart from the fact that it must have the same domain as
$F$. In particular, if $T$ is a $\mathpzc{Gray}$-monad and $T$
preserves co-$2$-inserters, co-$3$-inserters and coequifiers, then
also $TT$ preserves co-$2$-inserters, co-$3$-inserters and coequifiers
because $G$ and $TG$ have the same domain. This is used in the proof of the theorem to show that the
Eilenberg-Moore object $\mathpzc{K}^T$ has codescent objects and that they are
preserved by the forgetful $\mathpzc{Gray}$-functor
$\mathpzc{K}^T\rightarrow \mathpzc{K}$. Namely, in enriched monad
theory one can
show that the forgetful functor
$\mathpzc{K}^T\rightarrow \mathpzc{K}$ creates any colimit that is
preserved by $T$ and $TT$, just as in ordinary monad theory.

For any of the classes of indexed colimits above, the domain of $F$ is 
small, so codescent objects are small indexed colimits.
 Hence, if $\mathpzc{K}$ is
  cocomplete, it has codescent objects of codescent
  diagrams in particular. This observation gives the following
  corollary.
\end{remark}
\begin{corollary}\label{corollarycocomplete}
  Let $\mathpzc{K}$ be cocomplete and let  $T$ be a monad on $\mathpzc{K}$ that
  preserves small indexed colimits. Then the inclusion $i\co
  \mathpzc{K}^T\hookrightarrow
  \mathrm{Ps}\text{-}T\text{-}\mathrm{Alg}$ has a left adjoint $L\co
  \mathrm{Ps}\text{-}T\text{-}\mathrm{Alg}\rightarrow \mathpzc{K}^T$ and
  each component $\eta_X\co X\rightarrow iLX$ of the unit of the
  adjunction is an internal biequivalence in
  $\mathrm{Ps}\text{-}T\text{-}\mathrm{Alg}$. \qed
\end{corollary}

\section{The monad of the Kan adjunction}\label{sectionmonad}
\subsection{A $\mathpzc{V}$-monad on  $[\mathrm{ob}\mathpzc{P},\mathpzc{L}]$}\label{subsectionmonad}

Let $\mathpzc{V}$ be a complete and cocomplete symmetric monoidal closed
category such that the underlying category $\mathpzc{V}_0$ is locally
small.
By cocompleteness
we have an initial object which we denote by $\emptyset$.
Let $\mathpzc{P}$ be a small 
$\mathpzc{V}$-category and let $\mathpzc{L}$ be a cocomplete
$\mathpzc{V}$-category. In this general situation, we  now describe in
more detail a $\mathpzc{V}$-monad corresponding to the Kan adjunction with left adjoint left Kan extension $\mathrm{Lan}_H$
along a particular $\mathpzc{V}$-functor $H$ and right adjoint the functor $[H,1]$ from enriched category theory. 
For $\mathpzc{V}=\mathpzc{Gray}$, this is the $\mathpzc{Gray}$-monad
mentioned in the introduction, for which the pseudo algebras shall be compared
to locally strict trihomomorphisms.

First, observe that the set
$\mathrm{ob}\mathpzc{P}$ of the objects of $\mathpzc{P}$ may be considered as a
discrete $\mathpzc{V}$-category. More precisely, there is a 
$\mathpzc{V}$-category structure on  $\mathrm{ob}\mathpzc{P}$ such
that for objects
$P,Q\in\mathpzc{P}$ the hom object $(\mathrm{ob}\mathpzc{P})(P,Q)$ is
given by $I$  if $P=Q$ and by
$\emptyset$ 
otherwise and such that the nontrivial hom morphisms are given by
$l_I=r_I$.
The  $\mathpzc{V}$-functor $H$ is defined to be the unique
$\mathpzc{V}$-functor $\mathrm{ob}\mathpzc{P}\rightarrow \mathpzc{P}$
such that the underlying map on objects is the identity.

Since $\mathpzc{P}$ is small and since $\mathpzc{V}_0$ is complete, 
the functor category $[\mathpzc{P},\mathpzc{L}]$ exists. 
For two \mbox{$\mathpzc{V}$-functors} $A,B\co \mathpzc{P}\rightarrow \mathpzc{L}$ the
hom object $[\mathpzc{P},\mathpzc{L}](A,B)$ is the end
\begin{align*}
  \int_{P\in\mathrm{ob}\mathpzc{P}}\mathpzc{L}(AP,BP)
\puncteq{,}
\end{align*}
which is given by an equalizer
\begin{equation}\label{functorcatequalizer}
\begin{tikzpicture}
\matrix (a) [matrix of math nodes, row sep=3em, column sep=1em, text
height=1.5ex, text depth=0.25ex]{
  \underset{P\in\mathrm{ob}\mathpzc{P}}{\int} \mathpzc{L}(AP,BP)
    & \underset{P\in\mathrm{ob}\mathpzc{P}}{\prod}
   \mathpzc{L}(AP,BP)
    & \underset{P,Q\in\mathrm{ob}\mathpzc{P}}{\prod}
   [\mathpzc{P}(P,Q),\mathpzc{L}(AP,BQ)]\\ };

\draw[->]
($ (a-1-1.east) + (0,.0) $)
-- node [above] {}($ (a-1-2.west) + (0,.0) $);

\draw[->]
($ (a-1-2.east) + (0,.05) $)
-- node [above] {$\scriptstyle \rho$}($ (a-1-3.west) + (0,.05) $);
\draw[->]
($ (a-1-2.east) + (0,-.05) $)
-- node [below] {$\scriptstyle \sigma$}($ (a-1-3.west) + (0,-.05) $);		
\puncttikz[a-1-3]{}		
\end{tikzpicture}
\end{equation}
in $\mathpzc{V}_0$,
see \cite[(2.2), p. 27]{kelly}, where---if we denote by $\pi$ the cartesian projections---
$\rho$ and $\sigma$ are determined by requiring $\pi_{P,Q}\rho $ and
$\pi_{P,Q}\sigma $ to be $\pi_{P}$  composed with the transform of
$\mathpzc{L}(AP,B-)_{PQ}$ and $\pi_Q$ composed with the transform of
$\mathpzc{L}(A-,BQ)_{QP}$ respectively.

Now let $\mathpzc{M}$ be another small $\mathpzc{V}$-category and
$K\co \mathpzc{M}\rightarrow \mathpzc{P}$ be a $\mathpzc{V}$-functor
, e.g. $K=H$.
The $\mathpzc{V}$-functor
$K$ induces a $\mathpzc{V}$-functor
\begin{align*}
  [K,1]\co [\mathpzc{P},\mathpzc{L}]\rightarrow
  [\mathpzc{M},\mathpzc{L}]
\puncteq{,}
\end{align*}
which sends a $\mathpzc{V}$-functor $A\co \mathpzc{P}\rightarrow
\mathpzc{L}$ to the composite $\mathpzc{V}$-functor $AK$,
cf. \cite[(2.26)]{kelly}, and its hom morphisms are determined by the
universal property of the end and commutativity of the following diagram
\begin{equation}\label{homK1}
\begin{tikzpicture}[baseline=(current  bounding  box.center)]
      \matrix (a) [matrix of math nodes, row sep=4em, column sep=4em,
      text height=1.5ex, text depth=0.25ex]{
          \lbrack\mathpzc{P},\mathpzc{L}\rbrack(A,B) \pgfmatrixnextcell
        \lbrack\mathpzc{M},\mathpzc{L}\rbrack(AK,BK) \\
        \mathpzc{L}(AKM,BKM) \pgfmatrixnextcell \mathpzc{L}(AKM,BKM) \\};
\path[->] (a-1-2) edge node[right]{$\scriptstyle
        E_{M}$} (a-2-2); 
\path[->] (a-1-1) edge
      node[right]{$\scriptstyle E_{KM} $} (a-2-1); 
\path[->] (a-1-1)
      edge node[above]{$\scriptstyle [K,1]_{A,B}$}
      (a-1-2); 
 \draw[-]
 ($ (a-2-1.east) + (0,.05) $)
-- node [above] {}($ (a-2-2.west) + (0,.05) $);
\draw[-]
($ (a-2-1.east) + (0,-.05) $)
-- node [above] {}($ (a-2-2.west) + (0,-.05) $);
\puncttikz[a-2-2]{.}
    \end{tikzpicture}
  \end{equation}

Left Kan extension $\mathrm{Lan}_K\co
[\mathpzc{M},\mathpzc{L}]\rightarrow [\mathpzc{P},
\mathpzc{L}]$ along $K$ provides a left adjoint to $[K,1]$: this is
the usual Theorem of Kan adjoints as given in\cite[Th. 4.50,
p. 67]{kelly}, and it applies since $\mathpzc{M}$ and $\mathpzc{P}$
are small 
and since
$\mathpzc{L}$ is cocomplete.
In particular, we have a hom $\mathpzc{V}$-adjunction
\begin{align}
  \label{Kanadjunction}
  [\mathpzc{P},\mathpzc{L}](\mathrm{Lan}_K A,S)\cong
  [\mathpzc{M},\mathpzc{L}](A,[K,1](S))
\puncteq{,}
\end{align}
cf. \cite[(4.39)]{kelly}, which is
$\mathpzc{V}$-natural in $A\in[\mathpzc{M},
\mathpzc{L}]$ and $S\in[\mathpzc{P},
\mathpzc{L}]$.
Thus we have a monad
\begin{align}\label{monadKan}
  T=[H,1]\mathrm{Lan}_K\co [\mathpzc{M},\mathpzc{L}]\rightarrow
  [\mathpzc{M},\mathpzc{L}]
\end{align}
on $[\mathpzc{M},\mathpzc{L}]$, which we call the monad of the  Kan adjunction.
The unit $\eta\co  1\Rightarrow T$ of $T$ is given by the unit $\eta$ of the
adjunction \eqref{Kanadjunction},
while the multiplication $\mu\co TT\Rightarrow T$, is given by
\begin{align*}
 [H,1]\epsilon \mathrm{Lan}_H\co [H,1]\mathrm{Lan}_H
 [H,1]\mathrm{Lan_H}\Rightarrow [H,1]\mathrm{Lan}_H
\puncteq{}
\end{align*}
where $\epsilon$ is the counit of the adjunction \eqref{Kanadjunction}.

We now come back to the special case that
$\mathpzc{M}=\mathrm{ob}\mathpzc{P}$ and $K=H$.
Since $\mathpzc{P}$ is small, we may identify a functor
$\mathrm{ob}\mathpzc{P}\rightarrow \mathpzc{L}$ with its family of
values in $\mathpzc{L}$ i.e. the set of functors is identified with
the (small) limit in $\mathpzc{Set}$ given by the product
$\prod_{\mathrm{ob}\mathpzc{P}}\mathrm{ob}\mathpzc{L}$. 

 In fact, the equalizer \eqref{functorcatequalizer} is trivial for
$[\mathrm{ob}\mathpzc{P},\mathpzc{L}]$, so for two functors
$A,B\co \mathrm{ob}\mathpzc{P}\rightarrow \mathpzc{L}$, the hom
object $[\mathrm{ob}\mathpzc{P},\mathpzc{L}](A,B)$ is given by
the (small) limit in $\mathpzc{V}_0$ given by the product
$\prod_{P\in\mathrm{ob}\mathpzc{P}}\mathpzc{L}(AP,BP)$.

Namely,
$\rho$ and $\sigma$ are equal in \eqref{functorcatequalizer}: Denoting
by $\pi$ the projections of the cartesian products, $\pi_{P,Q}\rho =
\rho_{P,Q}\pi_P$ equals $\pi_{P,Q}\sigma=\sigma_{P,Q}\pi_Q$ because
for $P\neq Q$, the two morphisms
\begin{align*}
   \underset{P\in\mathrm{ob}\mathpzc{P}}{\prod}
   \mathpzc{L}(AP,BP)\rightarrow [\emptyset,\mathpzc{L}(AP,BQ)]
\end{align*}
must both be the transform of the unique morphism
\begin{align*}
   \emptyset  \rightarrow [\underset{P\in\mathrm{ob}\mathpzc{P}}{\prod}
   \mathpzc{L}(AP,BP),\mathpzc{L}(AP,BQ)]
\puncteq{;}
\end{align*}
and for $P=Q$, we have $\rho_{P,P}=\sigma_{P,P}$ because these are the
transforms of $\mathpzc{L}(AP,B-)_{PP}$ and
$\mathpzc{L}(A-,BP)_{PP}$,
which are both equal to
\begin{align*}
  j_{\mathpzc{L}(AP,BP)}\co
  I=(\mathrm{ob}\mathpzc{P})(P,P)\rightarrow
  [\mathpzc{L}(AP,BP),\mathpzc{L}(AP,BP)]
\end{align*}
by the unit axioms for the $\mathpzc{V}$-functors $A$, $B$,
$\mathpzc{L}(AP,-)$, and $\mathpzc{L}(-,BP)$. To see this, note that
$j_P\co I\rightarrow I$ is the identity functor, so $\mathpzc{L}(AP,B-)_{PP}=\mathpzc{L}(AP,B-)_{PP}j_P$ and
$\mathpzc{L}(A-,BP)_{PP}=\mathpzc{L}(A-,BP)_{PP}j_p$.

From diagram \eqref{homK1}, we see that
\begin{align*}
 [H,1]_{A,B}\co [\mathpzc{P},\mathpzc{L}](A,B)\rightarrow [\mathrm{ob}\mathpzc{P},\mathpzc{L}](AH,BH)
\end{align*}
is given by the strict functor of the equalizer \eqref{functorcatequalizer},
\begin{align*}
  \underset{P\in\mathrm{ob}\mathpzc{P}}{\int} \mathpzc{L}(AP,BP)
    \rightarrow \underset{P\in\mathrm{ob}\mathpzc{P}}{\prod}
   \mathpzc{L}(AP,BP)
\puncteq{,}
\end{align*}
that is, the strict functor into the product induced by the family of evaluation
functors $E_P$ where $P$ runs through the objects of $\mathpzc{P}$.

\begin{lemma}\label{pointwise}
  Let $\lbrace F,G\rbrace$ be a pointwise limit, then any representation is pointwise.
\end{lemma}
\begin{proof}
  Let $(B,\mu)$ be a pointwise representation and let $(B',\mu')$ be
  any other representation. By Yoneda,
  $\mu'$ has the form
  $[\mathpzc{P},\mathpzc{L}](\alpha,G-)\mu$
  for a unique isomorphism $\alpha\co B'\Rightarrow B$. It follows
  that
 \begin{align*}
   E_P\mu'=  
    E_P[\mathpzc{P},\mathpzc{L}](\alpha,G-)\mu  =
   \mathpzc{L}(\alpha_P,E_PG-) E_P \mu
  \end{align*}
and by extraordinary naturality this induces 
\begin{align*}
  \mathpzc{L}(L,B'P)\xrightarrow{\mathpzc{L}(L,\alpha_P)}\mathpzc{L}(L,BP)\xrightarrow{\beta}[\mathpzc{K},\mathpzc{V}](F,\mathpzc{L}(L,(G-)P))
\end{align*}
where $\beta$ is the isomorphism induced by $E_P\mu$, and this is an
isomorphism that is $\mathpzc{V}$-natural in $L$ and $P$ because
$(B,\mu)$ is pointwise and
$\alpha_P$ is an isomorphism that is $\mathpzc{V}$-natural in $P$.
This proves that $(B',\mu')$ is a pointwise limit.
\end{proof}

The following is the usual non-invariant notion of limit creation as
in MacLane's book
\cite[p. 108]{maclane71} adapted to the enriched context:
\begin{definition}\label{defcreate}
  A $\mathpzc{V}$-functor $T\co \mathpzc{B}\rightarrow \mathpzc{C}$
  creates $F\ast G$
  or creates colimits of $G\co
  \mathpzc{K}\rightarrow \mathpzc{B}$  indexed by  $F\co
  \mathpzc{K}\op\rightarrow \mathpzc{V}$  if
  \begin{enumerate*}
[label=(\roman*),ref= Def. \ref{defcreate}~(\roman*)]
  \item for every $(C,\nu)$ where $\nu\co
    F\rightarrow \mathpzc{C}(TG-,C)$ exhibits the object $C\in
    \mathpzc{C}$ as the colimit
    $F\ast (TG)$, there is a
    unique $(B,\xi)$ consisting of an object $B\in \mathpzc{B}$
    with $TB=C$ and a $\mathpzc{V}$-natural transformation $\xi\co F\rightarrow \mathpzc{B}(G-,B)$ with
    $T_{G-,B}\xi= \nu$, and if, moreover;
  \item $\xi$ exhibits $B$ as the colimit $F\ast G$.

There is a dual notion for creation of limits.
  \end{enumerate*}
\end{definition}
In particular, a colimit $F\ast G$ created by the
$\mathpzc{V}$-functor $T$ is also preserved by $T$.

\begin{lemma}\label{H1creates}
  The functor $[H,1]$ creates arbitrary pointwise (co)limits.
\end{lemma}
\begin{proof}
We only prove the colimit case, the proof for limits is analogous.
  If the colimit $(C,\nu)=(F\ast [H,1]G,\nu)$ exists pointwise, we
  have
  \begin{align*}
    CP=(F\ast [H,1]G)P=F\ast ([H,1](G-)P)=F\ast (G-)HP=F\ast (G-)P
\puncteq{,}
  \end{align*}
  which means that the value of $C$ at $P$ is a colimit $(CP,\nu^P)$ of $(G-)P$
  indexed by $F$. In fact, this determines the  $\mathpzc{V}$-functor $C$ uniquely since 
   the domain $\mathrm{ob}\mathpzc{P}$ is discrete.
  Further, it implies that the colimit $(F\ast
  G,\xi)$ exists pointwise because $F\ast (G-)P$ exists as
  $(CP,\nu^P)$, and this means that the functoriality of $F\ast G$ is
  induced from the pointwise representation and that $E_P\xi=\nu^P$. Now since
  \begin{align*}
    ([H,1](F\ast G))P=(F\ast G)HP=(F\ast G)P=F\ast (G-)P=CP
\puncteq{,}
  \end{align*}
the two functors $[H,1](F\ast G)$ and $F\ast ([H,1]G)$ coincide pointwise, and this means that they must also coincide as functors
$\mathrm{ob}\mathpzc{P}\rightarrow \mathpzc{L}$, i.e.
 $[H,1](F\ast G)= F\ast ([H,1]G)$.
Moreover, since the units coincide pointwise, $E_P\xi=\nu^P$, we must
have $[H,1]_{G-,A}\xi=\nu=\Pi_P\nu^P$.

This proves the existence of a $(B,\xi)$ as in 
  Definition \ref{defcreate}. Suppose there would be another $(B',\xi')$ with $[H,1]B'=C$ and
$[H,1]_{G-,B'}\xi'=\nu$. Then $B$ and $B'$ would coincide pointwise 
i.e. $BP=BP'$ for any object $P\in \mathpzc{P}$, and via $\xi$ and
$\xi'$  would both give rise to the same representation isomorphism---by the fact that $[H,1]_{G-,A}\xi'=\nu=[H,1]_{G-,A}\xi$ and thus $E_P
\xi'=E_p[H,1]_{G-,A}\xi'=E_P\nu=E_P[H,1]_{G-,A}\xi=E_P\xi$--- and this
representation isomomorphism is
$\mathpzc{V}$-natural in $P$ as well as in $L$:
\begin{align}\label{pointwiserep}
  \mathpzc{L}(BP,L)\cong [\mathpzc{K}\op,\mathpzc{V}](F,\mathpzc{L}((G-)P,L))
\puncteq{.}
\end{align}
But for such a representation isomorphism there is a unique way of
making $B$ a $\mathpzc{V}$-functor \mbox{$\mathpzc{P}\rightarrow \mathpzc{L}$}
such that the representation isomorphism is $\mathpzc{V}$-natural in
$P$ as well as in $L$, see for example \cite[1.10]{kelly}, so $B$ and
$B'$ have to coincide as $\mathpzc{V}$-functors. Clearly, by Yoneda, also
$\xi=\xi'$ then as the representations of $(B,\xi)$ and
$(B,\xi')$ coincide because the pointwise representations \eqref{pointwiserep} do, cf. \cite[3.3]{kelly}. Note here that $(B',\xi')$ must be a
pointwise colimit too because by assumption, it is preserved by
$[H,1]$ and $(C,\nu)$ is preserved by any $E_P$, so $(B',\xi')$ is
preserved by any $E_P$ and thus it is a pointwise colimit. On the
other hand, this is just the general fact that if a colimit exists
pointwise, then any representation must in fact be pointwise,  see
Lemma \ref{pointwise} above.

\end{proof}
\begin{corollary}
 The functor $[H,1]$ preserves any limit and any pointwise colimit
 that exists.
\end{corollary}
\begin{proof}
  This follows from the lemma above and the fact that $[H,1]$ is a
  right adjoint.
\end{proof}

\begin{remark}
  In case
  that $[H,1]$ is also a left adjoint, it in fact preserves any colimit that
  exists. This is for example the case when the target $\mathpzc{L}$
  is complete,  where the
  right adjoint is given by right Kan extension $\mathrm{Ran}_H$ along
  $H$,
  which exists because $\mathpzc{L}$ and $\mathrm{ob}\mathpzc{P}$ were
  assumed to be complete and small respectively. In particular, this applies in
  the situation that $\mathpzc{L}=\mathpzc{V}$. 
\end{remark}

\begin{corollary}\label{corollarycoherence}
Let $\mathpzc{P}$ be a small and $\mathpzc{L}$ be a cocomplete
$\mathpzc{Gray}$-category, and let  $T$ be the monad
$[H,1]\mathrm{Lan}_H$ on $[\mathrm{ob}\mathpzc{P},\mathpzc{L}]$ given
by the Kan adjunction.
  Then the inclusion $i\co
 [\mathpzc{P},\mathpzc{L}]\hookrightarrow
  \mathrm{Ps}\text{-}T\text{-}\mathrm{Alg}$ has a left adjoint $L\co
  \mathrm{Ps}\text{-}T\text{-}\mathrm{Alg}\rightarrow [\mathpzc{P},\mathpzc{L}]$ and
  each component $\eta_A\co A\rightarrow iLA$ of the unit of the
  adjunction is an internal biequivalence in
  $\mathrm{Ps}\text{-}T\text{-}\mathrm{Alg}$.
\qed
\end{corollary}
\begin{proof}
  We aim at applying Corollary \ref{corollarycocomplete} of
  Gurski's coherence theorem. Thus, we have to show that
  $T=[H,1]\mathrm{Lan}_H$ preserves small colimits. Since
  $\mathrm{Lan}_H$ is a left adjoint, it preserves any colimit that
  exists. Since this limit is again a small limit and since,
  $\mathpzc{L}$ being complete, small 
  limits are pointwise limits (cf. Lemma \ref{pointwise}), Lemma
  \ref{H1creates} implies that it is preserved by $[H,1]$. This
  proves that any small limit is preserved by $T$.
\end{proof}

\subsection{Explicit description of the
  monad}\label{subsectionexplicit}
In this paragraph, we will give an explicit description of the monad
from \ref{subsectionmonad} in terms of a coend over tensor
products. As a matter of fact, the explicit identification of the monad
structure is involved, and an alternative economical strategy adequate
for the purpose of this paper, would be to take the description in
terms of coends and tensor products as a definition. By functoriality
of the colimit it is then readily shown that this gives a monad on
$[\mathrm{ob}\mathpzc{P},\mathpzc{L}]$ as
required, but one has to show that it preserves pointwise (and thus
small) colimits in
order to apply Corollary \ref{corollarycocomplete} from \secref{corollarycocomplete}. This follows from
an appropriate form of the interchange of colimits theorem. For this
reason, we will be short on proofs below.

First, we recall the notions of tensor products and coends to present
the well-known Kan extension formula \eqref{explicitLan} below. Then we
determine the monad structure \mbox{$\mu\co TT\Rightarrow T$} and \mbox{$\eta\co 1\Rightarrow T$} for the monad from \ref{subsectionmonad}.
Given an object $X\in\mathpzc{V}$ and an object
$L\in\mathpzc{L}$, recall that the tensor product $X\otimes L$ is defined as the
colimit $X\ast L$ where $X$ and $L$ are considered as objects i.e. as
$\mathpzc{V}$-functors in the underlying categories
$\mathpzc{V}_0=\mathpzc{V}\text{-}\mathpzc{CAT}(\mathcal{I}\op,\mathpzc{\mathpzc{V}})$
and
$\mathpzc{L}_0=\mathpzc{V}\text{-}\mathpzc{CAT}(\mathcal{I},\mathpzc{\mathpzc{L}})$
where $\mathcal{I}$ is the unit $\mathpzc{V}$-category. With the
identification $[\mathcal{I},\mathpzc{L}]\cong \mathpzc{L}$, the
corresponding contravariant
representation \eqref{repcolimit} from \secref{repcolimit} has the form
\begin{align}\label{homVadjtensor}
 n\co \mathpzc{L}(X\otimes L,M)\cong [X,\mathpzc{L}(L,M)]
\puncteq{,}
\end{align}
and this is $\mathpzc{V}$-natural in all variables by 
functoriality of the colimit cf. \cite[(3.11)]{kelly}. This means
that tensor products are in fact $\mathpzc{V}$-adjunctions, and we
will dwell on this in the next paragraph \ref{someproperties}.
 Because
$\mathpzc{L}$ is assumed to be cocomplete, tensor products indeed exist.

Next, recall that for a $\mathpzc{V}$-functor $G\co
\mathpzc{A}\op\otimes \mathpzc{A}\rightarrow \mathpzc{L}$, the coend
\begin{align}\label{coendGAA}
  \int^{A}G(A,A)
\end{align}
is defined as the colimit $\mathrm{Hom}\op_{\mathpzc{A}}\ast G$. The
corresponding representation \eqref{repcolimit} from \secref{repcolimit} transforms under the
extra-variable enriched Yoneda lemma cf. \cite[(2.38)]{kelly} into the
following characteristic isomorphism of the coend:
\begin{align}\label{repcoend}
\beta\co \mathpzc{L}(\int^A G(A,A),L)\cong \int_{A}\mathpzc{L}(G(A,A),L)
\puncteq{,}
\end{align}
which is
$\mathpzc{V}$-natural in $L$ and where on the right we have an end in
the ordinary sense cf. \cite[2.1]{kelly}.
The unit of $\mathrm{Hom}\op_{\mathpzc{A}}\ast G$ corresponds to a
$\mathpzc{V}$-natural family
\begin{align}\label{unitcoend}
  \kappa_{A}=\lambda_A\beta j_{\int^A G(A,A)}\co G(A,A)\rightarrow \int^A G(A,A)
\puncteq{,}
\end{align}
where $\lambda_A$ is the counit of the end, 
and \eqref{repcoend} induces the following universal property of
$\kappa_P$:
\begin{align}\label{universalcoend}
 \mathpzc{L}_0(\int^A G(A,A),L)\cong
 \mathpzc{V}\text{-}\mathrm{nat}(G(A,A),L)
\puncteq{.}
\end{align}
This is a bijection of sets and it is given by precomposition with
$\kappa_A$, which proves that $\kappa_A$ is the universal
$\mathpzc{V}$-natural family with domain $G(A,A)$.
Since $\mathpzc{L}$ was assumed to be cocomplete, small coends in
$\mathpzc{L}$ do in fact exist.

We are now ready to present the explicit description of left Kan
extension and thus of the monad from \ref{subsectionmonad}. Since
$\mathpzc{L}$ admits tensor products and since $\mathpzc{P}$ was
assumed to be small, left Kan extension along the functor $H\co
\mathrm{ob}\mathpzc{P}\rightarrow \mathpzc{P}$ from
\ref{subsectionmonad} is given by the following small coend:
\begin{align}\label{explicitLan}
 \mathrm{Lan}_HA\cong
 \int^{P}\mathpzc{P}(P,-)\otimes AP
\puncteq{}
\end{align}
cf. \cite[(4.25)]{kelly}.
\begin{example*}
  In case that $\mathpzc{L}=\mathpzc{V}$, the coend in
  \eqref{coendGAA} is given by a
  coproduct in $\mathpzc{V}_0$: Indeed $\mathrm{ob}\mathpzc{P}$ is the
  free $\mathpzc{V}$-category
  $((\mathrm{ob}\mathpzc{P})_0)_{\mathpzc{V}}$ where the set of objects
  of $\mathpzc{P}$ is considered as the ordinary discrete
  category $(\mathrm{ob}\mathpzc{P})_0$, and
  $\mathrm{Hom}\op_{\mathrm{ob}\mathpzc{P}}=(\mathrm{Hom}\op_{(\mathrm{ob}\mathpzc{P})_0})_{\mathpzc{V}}$
  is the $\mathpzc{V}$-functor corresponding to the ordinary hom
  functor $\mathrm{Hom}\op_{(\mathrm{ob}\mathpzc{P})_0}$
  under the identification
  \begin{align*}
    ((\mathrm{ob}\mathpzc{P})_0\times
    (\mathrm{ob}\mathpzc{P})_0)_{\mathpzc{V}}\cong
    \mathrm{ob}\mathpzc{P}\otimes \mathrm{ob}\mathpzc{P}
\puncteq{,}
  \end{align*}
where we have
  dropped the superfluous superscript $\op$. Thus,
  $\mathrm{Hom}\op_{\mathrm{ob}\mathpzc{P}}\ast G(A,A)$ reduces to a
  conical colimit in $\mathpzc{V}$, which, $\mathpzc{V}$ being
  cotensored, coincides with the ordinary colimit, hence the
  coproduct.

Next one observes that tensor products in $\mathpzc{V}$ are given by the monoidal
structure as is easily seen from \eqref{homVadjtensor}
cf. \eqref{repp} in \secref{repp}.
 Therefore, \eqref{explicitLan} reduces to the coproduct
  \begin{align}
    \mathrm{Lan}_HA\cong
    \sum_{P\in\mathrm{ob}\mathpzc{P}}\mathpzc{P}(P,-)\otimes AP
    \puncteq{.}
  \end{align}
\end{example*}

Now let $\mathpzc{M}$ be another small $\mathpzc{V}$-category and
$K\co \mathpzc{M}\rightarrow \mathpzc{P}$ be a $\mathpzc{V}$-functor.
Then left Kan extension along $K$ exists in the form of
\begin{align}\label{explicitLanK}
 \mathrm{Lan}_KA\cong
 \int^{M}\mathpzc{P}(KM,-)\otimes AM
\puncteq{,}
\end{align}
 the relevant functor categories exist,
and we again have the Kan
adjunction $\mathrm{Lan}_K\dashv [K,1]$.

\begin{lemma}\label{counitKanadj}
  The component at $A\in[\mathpzc{P},\mathpzc{L}]$
 of the counit $\epsilon\co \mathrm{Lan}_K[K,1]\Rightarrow
 1_{[\mathpzc{P},\mathpzc{L}]}$ of the adjunction
 $\mathrm{Lan}_K\dashv[K,1]$
 has
 component
 \begin{align*}
  \epsilon_{A,Q}\co \int^{M}\mathpzc{P}(KM,Q)\otimes
   AKM\rightarrow AQ
 \end{align*}
at $Q\in\mathpzc{P}$
induced from the $\mathpzc{V}$-natural transform
\begin{align*}
  \mathpzc{P}(KM,Q)\otimes AKM\rightarrow AQ
\end{align*}
of the hom morphism
\begin{align*}
  A_{KM,Q}\co \mathpzc{P}(KM,Q)\rightarrow \mathpzc{L}(AKM,AQ)
\end{align*}
under the adjunction \eqref{homVadjtensor} of the tensor product.
\end{lemma}
\begin{proof}
  The component at $A$ is obtained by composing the unit
  \begin{align*}
   j_{AK}\co I\rightarrow
    [\mathrm{ob}\mathpzc{P},\mathpzc{L}](AK,AK)
  \end{align*}
with the
  inverse of the $\mathpzc{V}$-natural isomorphism of the Kan
  adjunction \eqref{Kanadjunction}. The lemma then follows from inspection of the proof
  of the theorem of Kan adjoints \cite[Th. 4.38]{kelly}. In
  particular, the
   transform of the $\mathpzc{V}$-natural
   $\mathpzc{L}(AKM,-)_{AKM,AQ}A_{KM,Q}$, which gives rise to the
  extra-variable Yoneda isomorphism \cite[(2.33)]{kelly}, enters
  in the inverse of \eqref{Kanadjunction}, and this is the point where the hom morphism
  $A_{KM,Q}$ shows up.

\end{proof}
We will show in the next paragraph \ref{someproperties} that there are
obvious left unitors $\lambda$ and associators $\alpha$ for the tensor
products. These already show up in the following two lemmata, but since
we mostly omit the proofs, it seems more stringent to state the lemmata
here in order to have the explicit description of $T$ at one place.
\begin{lemma}\label{unitmonad}
The component at $A\in[\mathpzc{M},\mathpzc{L}]$ of the unit
  of the adjunction $\mathrm{Lan}_K\dashv [K,1]$ and the corresponding
  monad $T= \mathrm{Lan}_K\dashv [K,1]$ on $[\mathpzc{M},\mathpzc{L}]$, i.e. the $\mathpzc{V}$-natural transformation $\eta\co
  1_{[\mathpzc{M},\mathpzc{L}]}\Rightarrow [K,1]\mathrm{Lan}_K$  has component
  \begin{align*}
    \eta_{A,M}\co AM\xrightarrow{\lambda_{AM}^{-1}} I\otimes
    AM\xrightarrow{j_{KM}\otimes 1} \mathpzc{P}(KM,KM)\otimes AM
    \xrightarrow{\kappa_{M,KM}} \int^{O}\mathpzc{P}(KO,KM)\otimes AO
  \end{align*}
at $M$ in $\mathpzc{M}$ where $\lambda_{AM}^{-1}$ is the
unitor of the tensor product cf. \ref{someproperties}.
\end{lemma}

\begin{proof}
Note that we have stressed in the statement that the unit of the monad
$T$ is exactly given by the unit of the
  adjunction $\mathrm{Lan}_K\dashv [K,1]$. Hence, its component at
  $A\in[\mathpzc{M},\mathpzc{L}]$ is given by composing the
  $\mathpzc{V}$-natural isomorphism of the Kan
  adjunction \eqref{Kanadjunction} with
  the unit
  \begin{align*}
   j_{\mathrm{Lan}_KA} \co I\rightarrow [\mathpzc{P},\mathpzc{L}](\mathrm{Lan}_{K}A,
    \mathrm{Lan}_KA)
\puncteq{.}
  \end{align*}
This
  gives an element $I\rightarrow
  [\mathrm{ob}\mathpzc{P},\mathpzc{L}](A,[K,1](\mathrm{Lan}_KA))$, that is, a
  $\mathpzc{V}$-natural transformation \mbox{$A\Rightarrow
  [K,1](\mathrm{Lan}_KA)=TA$} cf. \eqref{Kanadjunction}.
Since  the inverse of the extra-variable enriched Yoneda
isomorphism \cite[(2.33)]{kelly} takes part in \eqref{Kanadjunction}, this is 
converse to the situation in Lemma \ref{unitmonad}. Correspondingly, one has to consider the
transform of $\mathpzc{L}(AO,(\mathrm{Lan}_KA)-)_{O,KM}$, although one
does not have to determine $((\mathrm{Lan}_KA)-)_{O,KM}$ in the
argument as one only uses the unit axiom for a $\mathpzc{V}$-functor.
\end{proof}

\begin{lemma}\label{homLan}
  The hom morphism of $\mathrm{Lan}_KA=\int^{M}\mathpzc{P}(KM,-)\otimes AM$,
  \begin{align*}
(\mathrm{Lan}_KA)_{Q,R}\co\mathpzc{P}(Q,R)\rightarrow
    \mathpzc{L}(\int^M\mathpzc{P}(KM,Q)\otimes
    AM,\int^M\mathpzc{P}(KM,R)\otimes AM)
 \end{align*}
  corresponds to the $\mathpzc{V}$-natural family 
(in $M\in\mathpzc{M}$ and
  also in $Q\in\mathpzc{P}$ but $Q$ is held constant here)
    \begin{align}
  \label{homKan}
\kappa_{M,R}(M_{\mathpzc{P}}\otimes 1)\alpha^{-1}
\end{align}
under \eqref{homVadjtensor}, exchange of the colimits
$\mathpzc{P}(Q,R)\otimes -$ and $\int^M$, and \eqref{universalcoend}.
\end{lemma}
\begin{proof}
A neat way of proving this is by showing that the prescription in the
statement of the lemma gives rise to the correct unit of the
representation for left
Kan adjunction along $K$ via
\begin{align}\label{prescriptionhomKan}
  \mathpzc{P}(KM,R)\xrightarrow{(\mathrm{Lan}_K)_{KM,R}}\mathpzc{L}((\mathrm{Lan}_KA)KM,(\mathrm{Lan}_KA)R)\xrightarrow{\mathpzc{L}(\eta_{A,M},1)} \mathpzc{L}(AM,(\mathrm{Lan}_KA)R)
\end{align}
cf. \cite[dual of Th. 4.6 (ii)]{kelly}, where $\eta_{A,M}$ was determined
in Lemma \ref{unitmonad}, 
and the unit of
the representation of the left
Kan extension as a colimit in the form of \eqref{explicitLanK} 
is quickly determined to be
$\mathpzc{L}(1,\kappa_{M,R})\eta^{AM}_{\mathpzc{P}(KM,R)}$. Namely, the
unit of \eqref{repcoend}  is $\kappa_{M,R}$, then $n$ is applied to this, which
by \eqref{nintermsof} in \secref{nintermsof} below gives
\begin{align*}
  n(\kappa_{M,R})=[\eta_{\mathpzc{P}(KM,R)}^{AM},1]
  \mathpzc{L}(AM,-)_{\mathpzc{P}(KM,R)\otimes AM,\int^M
    \mathpzc{P}(KM,R)\otimes AM} (\kappa_{M,R})
\end{align*}
or
$[\eta^{AM}_{\mathpzc{P}(KM,R)},1]\mathpzc{L}(AM,\kappa_{M,R})$,  where $\eta$ is the counit of the
adjunction of the tensor product cf. \eqref{homVadjtensor}. Thus this is the
counit in question and it can be identified with $\mathpzc{L}(AM,\kappa_Q)\eta^{AM}_{\mathpzc{P}(KM,R)}$.
One then proves that \eqref{prescriptionhomKan} in fact has exactly
this form:

Denoting by $x$ exchange of the tensor product and the coend $\int^M$,
the relevant calculation is displayed below.
\begin{align*}
 &  \mathpzc{L}(\eta_{A,M},1)
  \mathpzc{L}(1,(V\beta)^{-1}(\kappa_{M,R}(M^{\mathpzc{P}}_{KM,KM,R}\otimes
  1_{AM})\alpha^{-1})x)\eta_{\mathpzc{P}(KM,R)}^{\int^M\mathpzc{P}(KM,KM)\otimes
    AM} \\
  =& \mathpzc{L}(\kappa_{M,KM}(j_{KM}\otimes
  1)\lambda^{-1}_{AM},1)
  \mathpzc{L}(1,(V\beta)^{-1}(\kappa_{M,R}(M^{\mathpzc{P}}_{KM,KM,R}\otimes
  1_{AM})\alpha^{-1})x)\eta_{\mathpzc{P}(KM,R)}^{\int^M\mathpzc{P}(KM,KM)\otimes
    AM} \\
& \mbox{(by Lemma \ref{unitmonad})} \\
   =& \mathpzc{L}((j_{KM}\otimes
  1)\lambda^{-1}_{AM},1)
  \mathpzc{L}(1,(V\beta)^{-1}(\kappa_{M,R}(M^{\mathpzc{P}}_{KM,KM,R}\otimes
  1_{AM})\alpha^{-1})x)\mathpzc{L}(\kappa_{M,KM},1)\eta_{\mathpzc{P}(KM,R)}^{\int^M\mathpzc{P}(KM,KM)\otimes
    AM} \\
& \mbox{(functoriality of $\mathrm{hom}_{\mathpzc{L}}$)} \\
  =& \mathpzc{L}((j_{KM}\otimes
  1)\lambda^{-1}_{AM},1)
  \mathpzc{L}(1,(V\beta)^{-1}(\kappa_{M,R}(M^{\mathpzc{P}}_{KM,KM,R}\otimes
  1_{AM})\alpha^{-1})x)\mathpzc{L}(1,1\otimes \kappa_{M})\eta_{\mathpzc{P}(KM,R)}^{\mathpzc{P}(KM,KM)\otimes
    AM} \\
& \mbox{(naturality of $\eta$)} \\
 =& \mathpzc{L}((j_{KM}\otimes
  1)\lambda^{-1}_{AM},1)
  \mathpzc{L}(1,(V\beta)^{-1}(\kappa_{M,R}(M^{\mathpzc{P}}_{KM,KM,R}\otimes
  1_{AM})\alpha^{-1})\kappa_{M})\eta_{\mathpzc{P}(KM,R)}^{\mathpzc{P}(KM,KM)\otimes
    AM} \\
& \mbox{(exchange of colimits is induced by an isomorphism of
  represented functors: $x(1\otimes \kappa_M)=\kappa_M$)} \\
 =& \specialcell{\mathpzc{L}((j_{KM}\otimes
  1)\lambda^{-1}_{AM},1)
  \mathpzc{L}(1,\kappa_{M,R}(M^{\mathpzc{P}}_{KM,KM,R}\otimes
  1_{AM})\alpha^{-1})\eta_{\mathpzc{P}(KM,R)}^{\mathpzc{P}(KM,KM)\otimes
    AM} \hfill\mbox{(since $(V\beta)=V[\kappa_M,1]$)}}\\
 =& \specialcell{\mathpzc{L}(\lambda^{-1}_{AM},1)
  \mathpzc{L}(1,\kappa_{M,R}(M^{\mathpzc{P}}_{KM,KM,R}\otimes
  1_{AM})a^{-1}(1\otimes (j_{KM}\otimes
  1)))\eta_{\mathpzc{P}(KM,R)}^{I\otimes
    AM} \hfill\mbox{(functoriality of $\mathrm{hom}_{\mathpzc{L}}$)}} \\
 =& \specialcell{\mathpzc{L}(\lambda^{-1}_{AM},1)
  \mathpzc{L}(1,\kappa_{M,R}(M^{\mathpzc{P}}_{KM,KM,R}(1\otimes j_{KM})\otimes
  1_{AM})\alpha^{-1})\eta_{\mathpzc{P}(KM,R)}^{I\otimes
    AM}\hfill \mbox{(by naturality of $\alpha$)} }\\
 =& \specialcell{\mathpzc{L}(\lambda^{-1}_{AM},1)
  \mathpzc{L}(1,\kappa_{M,R}(r_{\mathpzc{P}(KM,R)}\otimes
  1_{AM})\alpha^{-1})\eta_{\mathpzc{P}(KM,R)}^{I\otimes
    AM} \hfill\mbox{(by a $\mathpzc{V}$-category axiom)}} \\
 =&\specialcell{ \mathpzc{L}(\lambda^{-1}_{AM},1)
  \mathpzc{L}(1,\kappa_{M,R}(1\otimes \lambda_{AM}))\eta_{\mathpzc{P}(KM,R)}^{I\otimes
    AM} \hfill \mbox{(by the triangle identity \eqref{triangleidentity})}}\\
 =&\specialcell{\mathpzc{L}(\lambda_{AM}\lambda^{-1}_{AM},1)
  \mathpzc{L}(1,\kappa_{M,R})\eta_{\mathpzc{P}(KM,R)}^{AM} \hfill
  \mbox{(by naturality of $\eta$)}}\\
 =&
  \mathpzc{L}(1,\kappa_{M,R})\eta_{\mathpzc{P}(KM,R)}^{AM}=[\eta^{AM}_{\mathpzc{P}(KM,R)},1] \mathpzc{L}(AM,\kappa_{M,R}) 
\end{align*}

Thus the prescription \eqref{homKan} leads to the right counit, but this means
that the hom morphism $(\mathrm{Lan}_K)_{Q,R}$  must have precisely the
claimed form  since $\mathrm{Lan}_KA$ is uniquely
functorial such that the representation, which is induced from this
counit, is appropriately natural cf. \cite[1.10]{kelly} (and indeed
this is how the functoriality of \eqref{explicitLanK} is defined).
\end{proof}

\begin{corollary}
  Let $T=[K,1]\mathrm{Lan}_K$ be the monad of the Kan adjunction from
\ref{subsectionmonad}. 
  The component at $A\in[\mathpzc{M},\mathpzc{L}]$ of the $\mathpzc{V}$-natural transformation $\mu\co
 TT\Rightarrow T $ has component 
   corresponding  to the $\mathpzc{V}$-natural (in
   $M,N\in\mathpzc{M}$) family
\begin{align}
  \label{multimonad}
\kappa_{M,R}(M_{\mathpzc{P}}\otimes 1)\alpha^{-1}
\end{align}
  under exchange of colimits, 
 Fubini, and \eqref{universalcoend},
where $\alpha$ is the associator of the tensor product
cf. \ref{someproperties}.
\end{corollary}
\begin{proof}
The $\mathpzc{V}$-natural transformation $\mu$ of the monad is
determined by the counit $\epsilon$ of
the adjunction $\mathrm{Lan}_H\dashv [H,1]$. Namely, it is given by
the $\mathpzc{V}$-natural transformation denoted
\begin{align*}
 [H,1]\epsilon \mathrm{Lan}_H\co [H,1]\mathrm{Lan}_H
 [H,1]\mathrm{Lan_H}\Rightarrow [H,1]\mathrm{Lan}_H
\puncteq{}
\end{align*}
with component
\begin{align*}
[H,1]_{\mathrm{Lan}_H((\mathrm{Lan}_HA)H),\mathrm{Lan}_HA}\epsilon_{\mathrm{Lan}_H
A}\co I \rightarrow
[\mathrm{ob}\mathpzc{P}, \mathpzc{L}](
(\mathrm{Lan}_H((\mathrm{Lan}_HA)H))H, (\mathrm{Lan}_HA)H)
\end{align*}
at $A\in[\mathrm{ob}\mathpzc{P},\mathpzc{L}]$.
 Since $E_P$ factorizes through  $[H,1]$ and $\pi_P$, the component at
$Q\in\mathpzc{P}$ of $\mu_A$ is  simply given by the component of 
$\epsilon_{\mathrm{Lan}_H A}$ at $Q$. According to Lemma \ref{counitKanadj},
the component of $\epsilon_{\mathrm{Lan}_H A}$ at $Q\in\mathpzc{P}$ is
induced from the transform of $(\mathrm{Lan}_HA)_{P,Q}$, and by Lemma
\ref{homLan}, this transform is precisely given by \eqref{multimonad}.
\end{proof}

\subsection{Some properties of tensor products}\label{someproperties}

It is clear from the defining representation isomorphism
\eqref{homVadjtensor} from \secref{homVadjtensor}
of the tensor
product and its naturality in $X$,$L$, and $M$ 
that tensor
products, for any object $L$ in a tensored $\mathpzc{V}$-category $\mathpzc{L}$,
give an adjunction of $\mathpzc{V}$-categories as below
\begin{align}\label{adjunctionofthetensorproduct}
  (-\otimes L\co \mathpzc{V}\rightarrow \mathpzc{L})\quad \dashv \quad
  (\mathpzc{L}(L,-)\co \mathpzc{L}\rightarrow \mathpzc{V})
\puncteq{.}
\end{align}
Because the representation isomorphism is also $\mathpzc{V}$-natural
in $L$, it is a consequence of the extra-variable Yoneda lemma \cite[1.9]{kelly}
that the unit and counit of this adjunction are also extraordinarily
$\mathpzc{V}$-natural in $L$:

\begin{lemma}\label{natcounitten}
  The unit $\eta_X^L\co X\rightarrow \mathpzc{L}(L,X\otimes L)$ and
  counit $\epsilon^L_M\co \mathpzc{L}(L,M)\otimes L\rightarrow M$ of
  these adjunctions are extraordinarily $\mathpzc{V}$-natural in $L$
  (and ordinarily $\mathpzc{V}$-natural in $X$ and $M$). \qed
\end{lemma}

Recall that there is a natural (in $X,Y,Z\in\mathpzc{V}$) isomorphism
\begin{align}\label{repp}
  p\co[X\otimes Y,Z]\cong [X,[Y,Z]]
\puncteq{,}
\end{align}
which is
induced from the closed structure of $\mathpzc{V}$ via the ordinary
Yoneda lemma cf. \cite[1.5]{kelly}.
From $p$ and the hom $\mathpzc{V}$-adjunction \eqref{homVadjtensor} of the
tensor product, we construct a $\mathpzc{V}$-natural isomorphism
\begin{align}\label{repalpha}
  n^{-1}_{X,Y\otimes L, M}[X,n^{-1}_{Y,L,M}]p n_{X\otimes Y,L,M}\co
  \mathpzc{L}((X\otimes Y)\otimes L,M)\rightarrow \mathpzc{L}(X\otimes
  (Y\otimes L),M)
\puncteq{,}
\end{align}
which, by Yoneda, must be of the form $\mathpzc{L}(\alpha_{X,Y,L}^{-1},1)$ for a unique 
$\mathpzc{V}$-natural (in $X,Y,L$) isomorphism
\begin{align}\label{associatortensor}
  \alpha_{X,Y,L}^{-1} \co X\otimes (Y\otimes L)\cong (X\otimes Y)\otimes L
\puncteq{.}
\end{align}
The natural isomorphism \eqref{associatortensor} is called the associator for the tensor product.
Since tensor products reduce to the monoidal structure if
$\mathpzc{L}=\mathpzc{V}$,  the natural
isomorphism \eqref{associatortensor} is in this special case, by uniqueness, given by the
associator $a^{-1}$ for the monoidal structure of $\mathpzc{V}$.

In fact, there is a pentagon identity in terms of associators $\alpha$
and associators $a$:
\begin{lemma}
  Given objects $W,X,$ and $Y$ in $\mathpzc{V}$, and an object $L$ in a
  tensored $\mathpzc{V}$-category
  $\mathpzc{L}$, the associators $\alpha$ and $a$ satisfy the pentagon
  identity
\begin{align}
    \label{pentagonidentity}
    \alpha_{W,X,Y\otimes L} \alpha_{W\otimes X,Y,L}= (1_W\otimes
    \alpha_{X,Y,L}) \alpha_{W,X\otimes Y,Z} (a_{W,X,Z}\otimes 1_Z)
    \puncteq{,}
  \end{align}
 which is an identity of isomorphisms
  \begin{align*}
    ((W\otimes X)\otimes Y)\otimes L\rightarrow W\otimes (X\otimes
    (Y\otimes L))
\puncteq{.}
  \end{align*}
\end{lemma}

\begin{proof}
 The corresponding identity for the inverses is proved by showing that
 the corresponding  $\mathpzc{V}$-natural isomorphisms
  \begin{align*}
    \mathpzc{L}(((W\otimes X)\otimes Y)\otimes L,M)\cong
    \mathpzc{L}(W\otimes (X\otimes (Y\otimes L)),M) \puncteq{}
  \end{align*}
coincide.
The $\mathpzc{V}$-natural isomorphism corresponding to the inverse of
the left hand side of \eqref{pentagonidentity} is readily seen to be
given by
  \begin{align*}
    n^{-1}_{W,X\otimes (Y\otimes L),M} [W,n^{-1}_{X,Y\otimes
      L,M}[X,n^{-1}_{Y,L,M}]] p p n_{(W\otimes X)\otimes Y,L,M}
  \end{align*}
cf. \eqref{repalpha}, where we have used naturality of $p$ and cancelled
out two factors.
Similarly, the \mbox{$\mathpzc{V}$-natural} isomorphism corresponding
to the inverse of the right hand side of \eqref{pentagonidentity} is
given by
  \begin{align*}
    n^{-1}_{W,X\otimes (Y\otimes L),M}[W,n^{-1}_{X,Y\otimes
      L,M}[X,[n^{-1}_{Y,L,M}]]][W,p]p
    [\alpha^{-1},\mathpzc{L}(L,M)]n_{(W\otimes X)\otimes Y,L,M}
    \puncteq{.}
  \end{align*}
Thus, the identity \eqref{pentagonidentity} is proved as soon as we show that
\begin{align*}
  pp=[W,p]p [a^{-1},\mathpzc{L}(L,M)]
\puncteq{.}
\end{align*}
In fact, this last equation reduces to
  the pentagon identity for $a$ since $p$ is defined via Yoneda by
  \begin{align*}
    \mathpzc{V}_{0}(W, p)=\pi \pi \mathpzc{V}_{0}(a,1) \pi^{-1} 
\puncteq{,}
  \end{align*}
where
  $\pi$ is the hom $\mathpzc{Set}$-adjunction of the closed structure.
Namely, one observes that on the one hand,
\begin{align*}
  \mathpzc{V}_{0}(V,pp)=\mathpzc{V}_{0}(V,p)\mathpzc{V}_{0}(V,p)= \pi
  \pi \mathpzc{V}_{0}(a,1) \pi^{-1}\pi \pi \mathpzc{V}_{0}(a,1)
  \pi^{-1} =\pi \pi \mathpzc{V}_{0}(a,1) \pi
  \mathpzc{V}_{0}(a,1) \pi^{-1} \\
  = \pi \pi \pi \mathpzc{V}_{0}(a\otimes 1,1)
  \mathpzc{V}_{0}(a,1) \pi^{-1}  \\
  = \pi \pi \pi \mathpzc{V}_{0}(a(a\otimes 1),1) \pi^{-1}
\puncteq{,}
\end{align*}
and on the other hand,
\begin{align*}
  \mathpzc{V}_{0}(V, [W,p]p [a^{-1},\mathpzc{L}(L,M)])=
  \mathpzc{V}_{0}(V,[W,p]) \pi \pi \mathpzc{V}_{0}(a^{-1},1) \pi^{-1}
  \pi
  \mathpzc{V}_{0}(1\otimes a^{-1},1) \pi^{-1} \\
  = \mathpzc{V}_{0}(V,[W,p]) \pi \pi \mathpzc{V}_{0}(a,1)
  \mathpzc{V}_{0}(1\otimes a^{-1},1)
  \pi^{-1} \\
  = \pi \mathpzc{V}_0(V\otimes W,p) \pi^{-1}\pi \pi
  \mathpzc{V}_{0}(a,1) \mathpzc{V}_{0}(1\otimes a^{-1},1)
  \pi^{-1} \\
  = \pi \pi \pi \mathpzc{V}_{0}(a,1) \mathpzc{V}_{0}(a,1)
  \mathpzc{V}_{0}(1\otimes a^{-1},1)
  \pi^{-1} \\
  = \pi \pi \pi \mathpzc{V}_{0}((1\otimes a^{-1}) aa,1) \pi^{-1}
  \puncteq{.}
\end{align*}
\end{proof}

Similarly, recall that there is a natural (in $Z\in\mathpzc{V}$) isomorphism
\begin{align}\label{isomorphismi}
  i\co Z\cong [I,Z]
\end{align}
which is defined via Yoneda by
\begin{align}\label{iYoneda}
  [X,i^{-1}]=[r_X^{-1},1]p^{-1}\co [X,[I,Z]]\cong [X\otimes I,Z]\cong
[X,Z]
\puncteq{.}
\end{align}
 Thus from
$i$ and the hom $\mathpzc{V}$-adjunction $n$ of the tensor product, we
construct a $\mathpzc{V}$-natural (in $L,M$) isomorphism
\begin{align}\label{replambda}
  n^{-1}_{I,L,M}i\co \mathpzc{L}(L,M)
\cong [I,\mathpzc{L}(L,M)]\cong
 \mathpzc{L}(I\otimes L,M)
\puncteq{.}
\end{align}
By Yoneda, \eqref{replambda} must be of the
form $\mathpzc{L}(\lambda_L,1)$ for a unique isomorphism
\begin{align}
  \lambda_L\co I\otimes L\rightarrow L 
\puncteq{,}\label{unitortensor}
\end{align}
which must moreover be $\mathpzc{V}$-natural in
$L$. 

For $M=I\otimes L$, composing the inverse of the isomorphism \eqref{replambda} above
with $j_{I\otimes L}$,
\begin{align*}
 I\rightarrow  \mathpzc{L}(I\otimes L,I\otimes L)\cong
 [I,\mathpzc{L}(L,I\otimes L)]\cong \mathpzc{L}(L,I\otimes L)
\puncteq{,}
\end{align*}
must give $\lambda_L^{-1}$---since this is how we get hold of
the counit of such a natural transformation in general---and it is of course also the
map corresponding to $\eta_I^L$ under the isomorphism
$[I,\mathpzc{L}(L,I\otimes L)]\cong \mathpzc{L}(L,I\otimes L)$
---because this is exactly the inverse of the first isomorphism in
\eqref{replambda}.
Conversely, composing \eqref{replambda} for $M=L$ with $j_M$ must
give $\lambda_L$.

For $Y=I$, consider the composition of the representation isomorphism \eqref{repalpha} i.e. 
 \begin{align}
 \mathpzc{L}(\alpha^{-1},1)= n^{-1}_{X,I\otimes L, M}[X,n^{-1}_{I,L,M}]p n_{X\otimes I,L,M}
\puncteq{,}
\end{align}
with $\mathpzc{L}(r_X\otimes 1,1)$.
By naturality of $n$, functoriality of $[-,-]$, and the definition
of $i$ via Yoneda cf. \eqref{iYoneda}, we have the following chain of equations:
 \begin{align*}
\mathpzc{L}((r_X\otimes 1)\alpha^{-1},1) & =
  n^{-1}_{X,I\otimes L, M}[X,n^{-1}_{I,L,M}]p n_{X\otimes
    I,L,M}\mathpzc{L}(r_X\otimes 1,1) \\
& = n^{-1}_{X,I\otimes L,
    M}[X,n^{-1}_{I,L,M}]p [r_X,1]n_{X,L,M} \\
& = n^{-1}_{X,I\otimes L,
    M}[X,n^{-1}_{I,L,M}][X,i] n_{X,L,M} \\
& = n^{-1}_{X,I\otimes L,
    M} [X,\mathpzc{L}(\lambda_L,1)] n_{X,L,M} \\
& = \mathpzc{L}(1\otimes \lambda_L,1)
\puncteq{.}
\end{align*}
Hence, by Yoneda, we have proved the following lemma.
\begin{lemma}
 Given an objects $X$ in $\mathpzc{V}$ and an object $L$ in a
  tensored $\mathpzc{V}$-category
  $\mathpzc{L}$, there is a triangle identity for $r,\lambda$, and $\alpha$,
  \begin{align}
    \label{triangleidentity}
    (r_X \otimes 1_L)\alpha^{-1}_{X,I,L}=1_X\otimes \lambda_L
    \puncteq{,} 
  \end{align} 
which is an identity of isomorphisms
\begin{align*}
  X\otimes (I\otimes L)\rightarrow X\otimes L \puncteq{.}
\end{align*}
\qed
\end{lemma}

\begin{lemma}\label{identificationofML}
  Let $L,M,N$ be objects in a  tensored $\mathpzc{V}$-category
  $\mathpzc{L}$. Then
  \begin{align*}
    M_{\mathpzc{L}}\co \mathpzc{L}(M,N)\otimes
    \mathpzc{L}(L,M)\rightarrow \mathpzc{L}(L,N)
  \end{align*}
can be identified in terms of the associator $\alpha$ for tensor
products and units $\eta$ and counits $\epsilon$ of the tensor product
adjunctions:
\begin{align*}
  M_{\mathpzc{L}}= \mathpzc{L}(L, \epsilon^M_N(1\otimes
  \epsilon^L_M)\alpha)\eta^L_{\mathpzc{L}(M,N)\otimes
    \mathpzc{L}(L,M)}
\puncteq{.}
\end{align*}
\end{lemma}
\begin{proof}

First, recall that as for any $\mathpzc{V}$-adjunction, Yoneda implies
the identity $\mathpzc{L}(L,-)_{MN}=n
\mathpzc{L}(\epsilon^L_M,1)$ cf.
\cite[(1.53), p. 24]{kelly}, and thus
$M_{\mathpzc{L}}=e^{\mathpzc{L}(L,M)}_{\mathpzc{L}(L,N)}(n\mathpzc{L}(\epsilon^L_M,1)\otimes
1_{\mathpzc{L}(L,M)})$ cf. \eqref{covapartialhom} in \secref{covapartialhom}.

The lemma is now proved by the following chain of equations 
\begin{align*}
  M_{\mathpzc{L}}& =e^{\mathpzc{L}(L,M)}_{\mathpzc{L}(L,N)}(n\mathpzc{L}(\epsilon^L_M,1)\otimes
  1_{\mathpzc{L}(L,M)}) \\ 
& =\mathpzc{L}(1_L,\epsilon^L_N)\eta^L_{\mathpzc{L}(L,N)}e^{\mathpzc{L}(L,M)}_{\mathpzc{L}(L,N)}(n \mathpzc{L}(\epsilon^L_M,1)\otimes
  1_{\mathpzc{L}(L,M)})  \\
  & \quad\mbox{(by a triangle identity for unit and counit of the
    adjunction \eqref{adjunctionofthetensorproduct})} \\
& = \mathpzc{L}(1_L,\epsilon^L_N)\mathpzc{L}(1_L,(e^{\mathpzc{L}(L,M)}_{\mathpzc{L}(L,N)}(n \mathpzc{L}(\epsilon^L_M,1)\otimes
  1_{\mathpzc{L}(L,M)}))\otimes 1_L)\eta^L_{\mathpzc{L}(M,N)\otimes
    \mathpzc{L}(L,M)} \\
  & \quad\mbox{(by ordinary naturality of $\eta^L_X$ in $X$)} \\
& = \mathpzc{L}(1_L, \epsilon^{\mathpzc{L}(L,M)\otimes L}_{N}\alpha((\mathpzc{L}(\epsilon^L_M,1)\otimes
  1_{\mathpzc{L}(L,M)})\otimes 1_L))\eta^L_{\mathpzc{L}(M,N)\otimes
    \mathpzc{L}(L,M)} \\
  & \quad\mbox{(see below \astpar, by functoriality and the identity $\epsilon=\epsilon( (e(n\otimes
    1))\otimes 1)\alpha^{-1}$)} \\
& = \mathpzc{L}(1_L, \epsilon^{\mathpzc{L}(L,M)\otimes L}_{N}(\mathpzc{L}(\epsilon^L_M,1)\otimes
 ( 1_{\mathpzc{L}(L,M)}\otimes 1_L))\alpha)\eta^L_{\mathpzc{L}(M,N)\otimes
    \mathpzc{L}(L,M)} \\
  & \quad\mbox{(by ordinary naturality of $\alpha$)} \\
& = \mathpzc{L}(1_L, \epsilon^{M}_{N}(1\otimes \epsilon^L_M)\alpha)\eta^L_{\mathpzc{L}(M,N)\otimes
    \mathpzc{L}(L,M)} \\
  & \quad\mbox{(by extraordinary naturality of $\epsilon^L_{M}$ in $L$)} 
\end{align*}

To prove the identity used in \astpar\, consider the morphism 
\begin{align*}
 [\mathpzc{L}(\mathpzc{L}(L,M)\otimes
   L,N),\mathpzc{L}(\mathpzc{L}(L,M)\otimes L,N)]  \rightarrow \mathpzc{L}(\mathpzc{L}(\mathpzc{L}(L,M)\otimes
   L,N)\otimes (\mathpzc{L}(L,M)\otimes
   L),N)
\end{align*}
given by the composite
\begin{align}\label{compositefornLML}
  \mathpzc{L}(\alpha, N)  \mathpzc{L}((n^{\vphantom{-1}}_{\mathpzc{L}(L,M),L,N}\otimes 1)\otimes 1,N)
n_{[\mathpzc{L}(L,M),\mathpzc{L}(L,N)]\otimes \mathpzc{L}(L,M),L,N}^{-1} p^{-1}[n_{\mathpzc{L}(M,N),L,N}^{-1},n^{\vphantom{-1}}_{\mathpzc{L}(M,N),L,N}]
\end{align}
 in $\mathpzc{V}_0$,
where  $p\co[X\otimes Y,Z]\cong [X,[Y,Z]]$ is again the natural
isomorphism \eqref{repp} induced from the closed structure of $\mathpzc{V}$, and where we have added subscripts
such that the hom $\mathpzc{V}$-adjunction \eqref{homVadjtensor} from \secref{homVadjtensor}  is
now denoted by $n_{X,L,M}$.
If $\mathpzc{L}(\alpha,N)$ is spelled out in terms of
hom $\mathpzc{V}$-adjunctions $n$ and $p$
according to \eqref{repalpha}, then it is seen by naturality that \eqref{compositefornLML}
is in fact the same as 
\begin{align}\label{nLML}
  n^{-1}_{\mathpzc{L}(\mathpzc{L}(L,M)\otimes L,N),\mathpzc{L}(L,M)\otimes
    L,N}
\puncteq{.}
\end{align}
In particular, it is an isomorphism (although this is also clear
because each factor is an isomorphism) and appropriately
$\mathpzc{V}$-natural (and this follows from the composition calculus
respectively).
We now want to show that the unit of this natural isomorphism is given by
\begin{align*}
  \epsilon^L_N((e^{\mathpzc{L}(L,M)}_{\mathpzc{L}(L,N)}(n_{\mathpzc{L}(M,N),L,N}\otimes
  1))\otimes 1)\alpha^{-1}
\puncteq{}
\end{align*}
because then, by Yoneda, this must be the same as
$\epsilon_N^{\mathpzc{L}(L,M)\otimes L}$ since this is by definition the
unit of \eqref{nLML}.

The unit is obtained by applying \eqref{compositefornLML} (or rather
$V$ of it) to
$1_{\mathpzc{L}(\mathpzc{L}(L,M)\otimes L,N)}$, and we do this
factor-by-factor.
First, note that
\begin{align*}
  [n_{\mathpzc{L}(M,N),L,N}^{-1},n^{\vphantom{-1}}_{\mathpzc{L}(M,N),L,N}](1_{\mathpzc{L}(\mathpzc{L}(L,M)\otimes
    L,N)})= 1_{[\mathpzc{L}(L,M),\mathpzc{L}(L,N)]}
\puncteq{,}
\end{align*}
and
\begin{align*}
  p(1_{[\mathpzc{L}(L,M),\mathpzc{L}(L,N)]})=e^{\mathpzc{L}(L,M)}_{\mathpzc{L}(L,N)}
\end{align*}
because $p$ is the hom $\mathpzc{V}$-adjunction with underlying
adjunction $-\otimes Y\dashv [Y,-]$ given by the closed
structure i.e. $Vp=\pi$.

Now note that by ordinary naturality of $n$, we have
\begin{align*}
  [e^{\mathpzc{L}(L,M)}_{\mathpzc{L}(L,N)},1]=
  n_{[\mathpzc{L}(L,M),\mathpzc{L}(L,N)]\otimes \mathpzc{L}(L,M),L,N}
  \mathpzc{L}(e^{\mathpzc{L}(L,M)}_{\mathpzc{L}(L,N)}\otimes
  1_L,1)n^{-1}_{\mathpzc{L}(L,N),L,N}
\puncteq{.}
\end{align*}
Thus because $[e^{\mathpzc{L}(L,M)}_{\mathpzc{L}(L,N)},1](1_{\mathpzc{L}(L,N)})=e^{\mathpzc{L}(L,M)}_{\mathpzc{L}(L,N)}$, we may compute $n_{[\mathpzc{L}(L,M),\mathpzc{L}(L,N)]\otimes \mathpzc{L}(L,M),L,N}(e^{\mathpzc{L}(L,M)}_{\mathpzc{L}(L,N)})$ by applying
$\mathpzc{L}(e^{\mathpzc{L}(L,M)}_{\mathpzc{L}(L,N)}\otimes
1_L,1)n^{-1}_{\mathpzc{L}(L,N),L,N}$ to $1_{\mathpzc{L}(L,N)}$, the
result of which is
$\epsilon_{N}^{L}(e^{\mathpzc{L}(L,M)}_{\mathpzc{L}(L,N)} \otimes
1)$. Finally, applying the remaining factors $
\mathpzc{L}((n^{\vphantom{-1}}_{\mathpzc{L}(L,M),L,N}\otimes 1)\otimes
1,N)$ and $\mathpzc{L}(\alpha, N)$ indeed gives \eqref{compositefornLML}.
\end{proof}
\begin{remark}
  A different strategy for the proof of the lemma, is to first observe
  that the right hand side just as $M_{\mathpzc{L}}$ is ordinarily
  $\mathpzc{V}$-natural in $L$ and $N$ and extraordinarily 
  $\mathpzc{V}$-natural in $M$ by Lemma \ref{natcounitten}, naturality of
  $\alpha$, and the composition calculus. Then the identity in the
  lemma can be proved variable-by-variable by use of the Yoneda lemma
  where one considers the transforms in the case of the variable $M$.
\end{remark}


For an object $X\in\mathpzc{V}$ and objects $L,M\in\mathpzc{L}$, recall that the hom $\mathpzc{V}$-adjunction
\eqref{homVadjtensor} from \secref{homVadjtensor}  of the tensor product has the following
description in terms of the unit and the strict hom functor of the right
adjoint $\mathpzc{L}(L,-)$,
\begin{align}\label{nintermsof}
  n= [\eta_{X}^L,1]\mathpzc{L}(L,-)_{X\otimes L,M} \co
  \mathpzc{L}(X\otimes L,M)\rightarrow [X,\mathpzc{L}(L,M)]
\puncteq{.}
\end{align}

With this description of $n$ we are able to derive two
important identities for $n$ stated in the two lemmata below. These
are in fact the main technical tools that we employ to achieve the
promised identification of
$\mathrm{Ps}\text{-}T\text{-}\mathrm{Alg}$.
Recall that there is a $\mathpzc{V}$-functor $\mathrm{Ten}\co \mathpzc{V}\otimes
\mathpzc{V}\rightarrow \mathpzc{V}$ which is given on objects by
sending $(X,Y)\in\mathrm{ob}\mathpzc{V}\times \mathrm{ob}\mathpzc{V}$ to their product $X\otimes
Y\in\mathrm{ob}\mathpzc{V}$ and whose hom morphism
$\mathrm{Ten}_{(X,X'),(Y,Y')}\co [X,X']\otimes [Y,Y']\rightarrow [X\otimes
Y,X'\otimes Y']$  is such that 
\begin{align}\label{Tenhom}
e^{X\otimes Y}_{X'\otimes Y'}(\mathrm{Ten}_{(X,X'),(Y,Y')}\otimes 1_{X\otimes Y}) = (e^X_{X'}\otimes e^{Y}_{Y'})m
\end{align}
where $e$ denotes evaluation i.e. the counits of the adjunctions
comprising the closed structure of $\mathpzc{V}$ and where $m$
denotes interchange in $\mathpzc{V}$.

The two lemmata specify how $n$ behaves with respect to
 $\mathrm{Ten}\co \mathpzc{V}\otimes \mathpzc{V}\rightarrow
\mathpzc{V}$ and $M_{\mathpzc{L}}$ in two specific situations that we
will constantly face below.
\begin{lemma}[First Transformation Lemma]\label{asabovebynatofML}
Given objects $X,Y$ in $\mathpzc{V}$, and objects $L,M,N$ in
$\mathpzc{L}$, the following equality of $\mathpzc{V}$-morphisms
$\mathpzc{L}(X\otimes M, N) \otimes \mathpzc{L}(Y\otimes
L,M) \rightarrow [X\otimes Y, \mathpzc{L}(L,N)]$ holds.
  \begin{align*}
    &[\eta^{L}_{X\otimes Y},1]\mathpzc{L}(L,-)_{(X\otimes Y)\otimes L,
      N} \mathpzc{L}(\alpha,1) M_{\mathpzc{L}}(1_{\mathpzc{L}(X\otimes M, N)}\otimes (X\otimes -)_{Y\otimes L,M}) \\
    &
    =[1,M_{\mathpzc{L}}]\mathrm{Ten}_{(X,Y),(\mathpzc{L}(M,N),\mathpzc{L}(Y\otimes
      L,M))}(([\eta^{M}_{X},1]\mathpzc{L}(M,-)_{X\otimes
      M,N})\otimes
  (
  [\eta^{L}_{Y},1]\mathpzc{L}(L,-)_{Y\otimes L,M}))
  \end{align*}
In terms of the hom $\mathpzc{V}$-adjunction \eqref{homVadjtensor}
from \secref{homVadjtensor}, this means that
\begin{align*}
  n \mathpzc{L}(\alpha,1) M_{\mathpzc{L}}(\mathrm{Ten}(\mathpzc{L}(X\otimes
  M,N),-)_{X\otimes (X\otimes L),X\otimes M} (X\otimes -)_{Y\otimes L,M})=
  [1,M_{\mathpzc{L}}]\mathrm{Ten}_{(X,Y),(\mathpzc{L}(M,N),\mathpzc{L}(L,M))}(n\otimes
  n)
\puncteq{.}
\end{align*}
\end{lemma}
\begin{proof}
This is proved by the following chain of equations.
\begin{align*}
     &[\eta^{L}_{X\otimes Y},1]\mathpzc{L}(L,-)_{(X\otimes Y)\otimes L,
      N} \mathpzc{L}(\alpha,1)M_{\mathpzc{L}}(1_{\mathpzc{L}(X\otimes M, N)}\otimes (X\otimes -)_{Y\otimes
      L,M}) \\
& = [\mathpzc{L}(L,\alpha)\eta^{L}_{X\otimes Y},1]\mathpzc{L}(L,-)_{X\otimes (Y\otimes L),
      N} M_{\mathpzc{L}}(1_{\mathpzc{L}(X\otimes M, N)}\otimes (X\otimes -)_{Y\otimes
      L,M})\\
    & \mbox{\quad(by the functor axiom for $\mathpzc{L}(L,-)$ or ordinary $\mathpzc{V}$-naturality of $\mathpzc{L}(L,-)_{M,N}$ in $M$)} \\
& = [\mathpzc{L}(L,\alpha)\eta^{L}_{X\otimes Y},1] M_{\mathpzc{V}}(
\mathpzc{L}(L,-)_{X\otimes M, N}  \otimes (\mathpzc{L}(L,-)_{X\otimes (Y\otimes L),
      X\otimes M}(X\otimes -)_{Y\otimes
      L,M}))\\
    & \mbox{\quad(by the functor axiom for $\mathpzc{L}(L,-)$ )} \\
& =  M_{\mathpzc{V}}(
\mathpzc{L}(L,-)_{X\otimes M, N}  \otimes ([\mathpzc{L}(L,\alpha)\eta^{L}_{X\otimes Y},1]\mathpzc{L}(L,-)_{X\otimes (Y\otimes L),
      X\otimes M}(X\otimes -)_{Y\otimes
      L,M}))\\
    & \mbox{\quad(by ordinary $\mathpzc{V}$-naturality of $M_{\mathpzc{V}}$ or
      a $\mathpzc{V}$-category axiom if spelled out)} \\
& =  M_{\mathpzc{V}}(
\mathpzc{L}(L,-)_{X\otimes M, N}  \otimes ([ \mathpzc{L}(L,
(1_X\otimes \epsilon^L_{Y\otimes L})\alpha) \eta^L_{X\otimes
  \mathpzc{L}(L,Y\otimes L)}(X\otimes \eta^L_{Y}),1]\mathpzc{L}(L,-)_{X\otimes (Y\otimes L),
      X\otimes M}(X\otimes -)_{Y\otimes
      L,M}))\\
    & \mbox{\quad(see below (A), by a triangle identity and naturality)} \\
& =  M_{\mathpzc{V}}(
\mathpzc{L}(L,-)_{X\otimes M, N}  \otimes  ([ (X\otimes \eta^L_{Y}), \mathpzc{L}(L,
(1_X\otimes \epsilon^L_{M})\alpha) \eta^L_{X\otimes
  \mathpzc{L}(L,M)}]
\mathrm{Ten}(X,-)_{\mathpzc{L}(L,Y\otimes L),\mathpzc{L}(L,M)}\mathpzc{L}(L,-)_{Y\otimes L,M}
))\\
& \mbox{\quad(see below (C), by  ordinary $\mathpzc{V}$-naturality of $\mathpzc{L}(L,
  (1_X\otimes \epsilon^L_{K})\alpha) \eta^L_{X\otimes
    \mathpzc{L}(L,K)}$ in $K$)} \\
& =  M_{\mathpzc{V}}(
\mathpzc{L}(L,-)_{X\otimes M, N}  \otimes([ 1, \mathpzc{L}(L,
(1_X\otimes \epsilon^L_{M})\alpha) \eta^L_{X\otimes
  \mathpzc{L}(L,M)}]
(\mathrm{Ten}(X,-)_{Y,\mathpzc{L}(L,M)}[\eta^L_{Y},1]\mathpzc{L}(L,-)_{Y\otimes L,M})
))\\
& \mbox{\quad(ordinary $\mathpzc{V}$-naturality of $\mathrm{Ten}(X,-)_{V,Z}$ in
  $V$)} \\
& =  M_{\mathpzc{V}}(
( [ \mathpzc{L}(L,
(1_X\otimes \epsilon^L_{M})\alpha) \eta^L_{X\otimes
  \mathpzc{L}(L,M)},1]\mathpzc{L}(L,-)_{X\otimes M, N})  \otimes 
(\mathrm{Ten}(X,-)_{Y,\mathpzc{L}(L,M)}[\eta^L_{Y},1]\mathpzc{L}(L,-)_{Y\otimes L,M})
)\\
& \mbox{\quad(extraordinary $\mathpzc{V}$-naturality of $M_{\mathpzc{V}}$)}
\\
& =  M_{\mathpzc{V}}(
( [ M_{\mathpzc{L}}(\eta^M_{X}\otimes 1_{\mathpzc{L}(L,M)}),1]\mathpzc{L}(L,-)_{X\otimes M, N})  \otimes 
(\mathrm{Ten}(X,-)_{Y,\mathpzc{L}(L,M)}[\eta^L_{Y},1]\mathpzc{L}(L,-)_{Y\otimes L,M})
)\\
& \mbox{\quad(see below (B), by a triangle identity, naturality, and Lemma \ref{identificationofML})} \\
& =  M_{\mathpzc{V}}(
( [ (\eta^M_{X}\otimes
1_{\mathpzc{L}(L,M)}),M_{\mathpzc{L}}]\mathrm{Ten}(-,\mathpzc{L}(L,M))_{\mathpzc{L}(M,X\otimes
  M),\mathpzc{L}(M,N)}\mathpzc{L}(M,-)_{X\otimes
  M,N})  \otimes \\
& \quad\quad
(\mathrm{Ten}(X,-)_{Y,\mathpzc{L}(L,M)}[\eta^L_{Y},1]\mathpzc{L}(L,-)_{Y\otimes L,M})
)\\
& \mbox{\quad(see below (D), by ordinary $\mathpzc{V}$-naturality of $M_{\mathpzc{L}}$)}
\\
& =  M_{\mathpzc{V}}( 
([ 1,M_{\mathpzc{L}}] \mathrm{Ten}(-,\mathpzc{L}(L,M))_{X,\mathpzc{L}(M,N)} [ \eta^M_{X},1 ]\mathpzc{L}(M,-)_{X\otimes
M,N})  \otimes   
(\mathrm{Ten}(X,-)_{Y,\mathpzc{L}(L,M)}[\eta^L_{Y},1]\mathpzc{L}(L,-)_{Y\otimes L,M})
)\\
& \mbox{\quad(ordinary $\mathpzc{V}$-naturality of $\mathrm{Ten}(-,Z)_{U,W}$ in
  $U$)} \\
& =  [ 1,M_{\mathpzc{L}}]M_{\mathpzc{V}}( 
(\mathrm{Ten}(-,\mathpzc{L}(L,M))_{X,\mathpzc{L}(M,N)} [ \eta^M_{X},1 ]\mathpzc{L}(M,-)_{X\otimes
  M,N})  \otimes  
(\mathrm{Ten}(X,-)_{Y,\mathpzc{L}(L,M)}[\eta^L_{Y},1]\mathpzc{L}(L,-)_{Y\otimes L,M})
)\\
& \mbox{\quad(ordinary $\mathpzc{V}$-naturality of $M_{\mathpzc{V}}$)} \\
& =  [ 1,M_{\mathpzc{L}}]
\mathrm{Ten}_{(X,Y),(\mathpzc{L}(M,N),\mathpzc{L}(L,M))} ( M_{\mathpzc{V}\otimes \mathpzc{V}} \\
& \quad\quad
( ((([ \eta^M_{X},1 ]\mathpzc{L}(M,-)_{X\otimes
  M,N} )\otimes  j^{\mathpzc{V}}_{\mathpzc{L}(L,M)}) r^{-1}_{\mathpzc{L}(X\otimes M,N)})
 \otimes  
(( j^{\mathpzc{V}}_X\otimes
([\eta^L_{Y},1]\mathpzc{L}(L,-)_{Y\otimes
  L,M}))l^{-1}_{\mathpzc{L}(Y\otimes L,M)}))
)\\
& \mbox{\quad(functor axiom for $\mathrm{Ten}$ (partial functors spelled out))} \\
& =  [ 1,M_{\mathpzc{L}}]
\mathrm{Ten}_{(X,Y),(\mathpzc{L}(M,N),\mathpzc{L}(L,M))} \\
& \quad\quad(( M_{\mathpzc{V}}(([ \eta^M_{X},1 ]\mathpzc{L}(M,-)_{X\otimes
  M,N} )\otimes j^{\mathpzc{V}}_X) r^{-1}_{\mathpzc{L}(X\otimes M,N)})
 \otimes  
(M_{\mathpzc{V}}( j^{\mathpzc{V}}_{\mathpzc{L}(L,M)}\otimes
([\eta^L_{Y},1]\mathpzc{L}(L,-)_{Y\otimes
  L,M}))l^{-1}_{\mathpzc{L}(Y\otimes L,M)})
)\\
& \mbox{\quad(by $M_{\mathpzc{V}\otimes \mathpzc{V}}=(M_{\mathpzc{V}}\otimes
  M_{\mathpzc{V}})m$ where $m$ is interchange, naturality of $m$, and
  $m(r^{-1}\otimes l^{-1})=r^{-1}\otimes l^{-1}$)} \\
    &
    =[1,M_{\mathpzc{L}}]\mathrm{Ten}_{(X,Y),(\mathpzc{L}(M,N),\mathpzc{L}(L,M))}(([\eta^{M}_{X},1]\mathpzc{L}(M,-)_{X\otimes
      M,N})\otimes
  (
  [\eta^{L}_{Y},1]\mathpzc{L}(L,-)_{Y\otimes L,M}))\\
  & \mbox{\quad($\mathpzc{V}$-category axioms for $\mathpzc{V}$)} 
  \end{align*}

In (A) and (B) we have used the following identities of
$\mathpzc{V}$-morphisms.
For (A), observe that
\begin{align*}
  \mathpzc{L}(L,\alpha)\eta^L_{X\otimes Y} & =
  \mathpzc{L}(L,(1_X\otimes (\epsilon^L_{Y\otimes L}(\eta^L_{Y}\otimes
  1_L)))\alpha)\eta^L_{X\otimes Y}\\
  & \mbox{\quad(by a triangle identity)} \\
& =
  \mathpzc{L}(L,(1_X\otimes \epsilon^L_{Y\otimes L}) (1_X\otimes (\eta^L_{Y}\otimes
  1_L))\alpha)\eta^L_{X\otimes Y}\\
  & \quad\mbox{(by functoriality of $X\otimes -$)} \\
& =
  \mathpzc{L}(L,(1_X\otimes \epsilon^L_{Y\otimes L}) \alpha ((1_X\otimes \eta^L_{Y})\otimes
  1_L))\eta^L_{X\otimes Y}\\
  & \mbox{\quad(by ordinary $\mathpzc{V}$-naturality of $\alpha$)} \\
& =
  \mathpzc{L}(L,(1_X\otimes \epsilon^L_{Y\otimes L}) \alpha
  )\eta^L_{X\otimes \mathpzc{L}(L,Y\otimes L)} (1_X\otimes
  \eta^L_{Y}) \\
  & \mbox{\quad(by ordinary $\mathpzc{V}$-naturality of $\eta$)} \\
& = \mathpzc{L}(L,(1_X\otimes \epsilon^L_{Y\otimes L}) \alpha
  )\eta^L_{X\otimes \mathpzc{L}(L,Y\otimes L)} (X\otimes
  \eta^L_{Y} ) \\
& = \mathpzc{L}(L,(1_X\otimes \epsilon^L_{Y\otimes L}) \alpha
  )\eta^L_{X\otimes \mathpzc{L}(L,Y\otimes L)}
  \mathrm{Ten}(X,\eta^L_{Y}) 
\puncteq{.}
\end{align*}
Similarly, for (B) observe that
\begin{align*}
   \mathpzc{L}(L,(1_X\otimes \epsilon^L_{M}) \alpha
  )\eta^L_{X\otimes \mathpzc{L}(L,M)} & =
  \mathpzc{L}(L,\epsilon^M_{X\otimes M}(\eta^M_{X}\otimes
  1_M)(1_X\otimes \epsilon^L_M)\alpha) \eta^L_{X\otimes
    \mathpzc{L}(L,M)}\\
  & \mbox{\quad(by a triangle identity)} \\
& = \mathpzc{L}(L,\epsilon^M_{X\otimes M}(1_{\mathpzc{L}(M,X\otimes M)}\otimes \epsilon^L_M)(\eta^M_{X}\otimes
  1_{\mathpzc{L}(L,M)\otimes L}) \alpha)\eta^L_{X\otimes
    \mathpzc{L}(L,M)}\\
  & \mbox{\quad(by underlying functoriality of the functor
    $\otimes$ )} \\
& = \mathpzc{L}(L,\epsilon^M_{X\otimes M}(1_{\mathpzc{L}(M,X\otimes M)}\otimes \epsilon^L_M)\alpha ((\eta^M_{X}\otimes
  1_{\mathpzc{L}(L,M)})\otimes 1_{L}))\eta^L_{X\otimes
    \mathpzc{L}(L,M)}\\
  & \mbox{\quad(by underlying functoriality of $\otimes$,
    $1_{\mathpzc{L}(L,M)\otimes L}=1_{\mathpzc{L}(L,M)}\otimes 1_L$,} \\
  & \quad \quad\mbox{
    and by ordinary $\mathpzc{V}$-naturality of $\alpha$  )} \\
& = \mathpzc{L}(L,\epsilon^M_{X\otimes M}(1_{\mathpzc{L}(M,X\otimes
  M)}\otimes \epsilon^L_M)\alpha )\eta^L_{\mathpzc{L}(M,X\otimes M)\otimes
    \mathpzc{L}(L,M)}  (\eta^M_{X}\otimes
  1_{\mathpzc{L}(L,M)})\\
  & \mbox{\quad(by ordinary $\mathpzc{V}$-naturality of $\eta$)} \\
& =M_{\mathpzc{L}}(\eta^M_{X}\otimes
  1_{\mathpzc{L}(L,M)})\\
  & \mbox{\quad(by the identification of $M_{\mathpzc{L}}$ in Lemma \ref{identificationofML} above)}
\\
& =M_{\mathpzc{L}}(\eta^M_{X}\otimes
  \mathpzc{L}(L,M))= M_{\mathpzc{L}} \mathrm{Ten}(\eta^M_{X},
  \mathpzc{L}(L,M))
\puncteq{.}
\end{align*}
Finally, we comment on the $\mathpzc{V}$-naturality used in (C) and
(D):

For (C) recall that $\eta^L_{J}$ is $\mathpzc{V}$-natural in
$J$. Then so is $\eta^L_{X\otimes \mathpzc{L}(L,K)}$ because this is
$\eta^L_{PK}$ for $P=\mathpzc{L}(L,-)(X\otimes-)$. Next, recall that
$\alpha$ is ordinarily $\mathpzc{V}$-natural in all of its variables,
and that $\epsilon^L_K$ is ordinarily $\mathpzc{V}$-natural in
$K$. Then so is $1_X\otimes \epsilon^L_K$ because this is $Q_0(\epsilon^L_K)$ for
$Q=(X\otimes-)$, and thus the composite $(1_X\otimes
\epsilon^L_K)\alpha$ is ordinarily $\mathpzc{V}$-natural in $K$. From this it
follows that $\mathpzc{L}(L,(1_X\otimes
\epsilon^L_K)\alpha)$ is ordinarily $\mathpzc{V}$-natural in $K$
because this is $Q_0((1_X\otimes
\epsilon^L_K)\alpha)$ for $Q=\mathpzc{L}(L,-)$. Hence, we conclude
that the composite
 family $\mathpzc{L}(L,
(1_X\otimes \epsilon^L_{K})\alpha) \eta^L_{X\otimes
  \mathpzc{L}(L,K)}$ is ordinarily $\mathpzc{V}$-natural in $K$.
That is, $\mathpzc{L}(L,
(1_X\otimes \epsilon^L_{K})\alpha) \eta^L_{X\otimes
  \mathpzc{L}(L,K)}$ is the component at $K\in\mathpzc{L}$  of a
$\mathpzc{V}$-natural transformation
\begin{align*}
  \mathrm{Ten}(X,-)\mathpzc{L}(L,-)\Rightarrow
  \mathpzc{L}(L,-)(X\otimes -)
\end{align*}
where $X$ and $L$ are held constant, and where $X\otimes -\co \mathpzc{L}\rightarrow \mathpzc{L}$ is the
partial functor  of the tensor
product in contrast to the partial functor $\mathrm{Ten}(X,-)\co
\mathpzc{V}\rightarrow \mathpzc{V}$ induced by the Gray product.

For (D) recall that $M_{\mathpzc{L}}\co \mathpzc{L}(M,K)\otimes
\mathpzc{L}(L,M)\rightarrow \mathpzc{L}(L,K)$ is
$\mathpzc{V}$-natural. Since $\mathpzc{V}$-naturality may be verified
variable-by-variable, $M_{\mathpzc{L}}$ is in particular ordinarily
$\mathpzc{V}$-natural in $K$. That is, it is the component at
$K\in\mathpzc{L}$ of a
$\mathpzc{V}$-natural transformation
\begin{align*}
  \mathrm{Ten}(-,\mathpzc{L}(L,M))\mathpzc{L}(M,-)\Rightarrow \mathpzc{L}(L,-)
\end{align*}
where $L$ and $M$ are held constant.
\end{proof}
\begin{lemma}[Second Transformation Lemma]\label{noticethesymmetry}
Given an object $X$ in $\mathpzc{V}$, and objects $L,M,N$ in
$\mathpzc{L}$, the following equality of $\mathpzc{V}$-morphisms
$\mathpzc{L}(X\otimes M, N) \otimes \mathpzc{L}(
L,M) \rightarrow [X, \mathpzc{L}(L,N)]$ holds, where $c$ is the
symmetry of $\mathpzc{V}$.
\begin{align*}
    &[\eta^{L}_{X},1]\mathpzc{L}(L,-)_{X\otimes L,
      N} M_{\mathpzc{L}}(1_{\mathpzc{L}(X\otimes M,N)}\otimes (X\otimes -)_{L,M}) \\
    &
    =M_{\mathpzc{V}}((\mathpzc{L}(-,N)_{M,N})\otimes
  (
  [\eta^{M}_{X},1]\mathpzc{L}(M,-)_{(X\otimes M,N}))c
  \end{align*}
In terms of the hom $\mathpzc{V}$-adjunction \eqref{homVadjtensor}
from \secref{homVadjtensor},
this means that
\begin{align*}
  n M_{\mathpzc{L}}\mathrm{Ten}(\mathpzc{L}(X\otimes
  M,N),-)_{X\otimes L,X\otimes M} (X\otimes -)_{L,M} =
  M_{\mathpzc{V}}(\mathpzc{L}(-,N)_{L,M}\otimes n)c
\end{align*}

\end{lemma}
\begin{proof}
  This is proved by the following chain of equations.
  \begin{align*}
      &[\eta^{L}_{X},1]\mathpzc{L}(L,-)_{X\otimes L,
      N} M_{\mathpzc{L}}(1_{\mathpzc{L}(X\otimes M,N)}\otimes
    (X\otimes -)_{L,M}) \\
&  =[\eta^{L}_{X},1] M_{\mathpzc{V}}(\mathpzc{L}(L,-)_{X\otimes M,
      N}\otimes
    (\mathpzc{L}(L,-)_{X\otimes L,
      X\otimes M}(X\otimes -)_{L,M}))\\
    & \quad\mbox{(functor axiom for $\mathpzc{L}(L,-)$)} \\
&  = M_{\mathpzc{V}}(\mathpzc{L}(L,-)_{X\otimes M,
      N}\otimes
    ([\eta^{L}_{X},1]\mathpzc{L}(L,-)_{X\otimes L,
      X\otimes M}(X\otimes -)_{L,M}))\\
    & \quad\mbox{(ordinary $\mathpzc{V}$-naturality of $M_{\mathpzc{V}}$)} \\
        &  = M_{\mathpzc{V}}(\mathpzc{L}(L,-)_{X\otimes M,
      N}\otimes
    ([\eta^{L}_{X},1]\mathpzc{L}(-,X\otimes M)_{L, M}))\\
    & \quad\mbox{(extraordinary $\mathpzc{V}$-naturality of $\eta^L_X$ in $L$)}
\\
        &  = [\eta^{L}_{X},1]M_{\mathpzc{V}}(\mathpzc{L}(L,-)_{X\otimes M,
      N}\otimes
   \mathpzc{L}(-,X\otimes M)_{L, M})\\
   & \quad\mbox{(by ordinary $\mathpzc{V}$-naturality of $M_{\mathpzc{V}}$)} \\
    &
    = [\eta^{M}_{X},1]M_{\mathpzc{V}}((\mathpzc{L}(-,N)_{L,M})\otimes
  \mathpzc{L}(M,-)_{X\otimes M,N})c \\
  & \quad\mbox{(see below, by extraordinary $\mathpzc{V}$-naturality of
    $\mathpzc{L}(L,-)_{M,N}$ in $L$ )} \\
 &   =M_{\mathpzc{V}}((\mathpzc{L}(-,N)_{L,M})\otimes
  (
  [\eta^{M}_{X},1]\mathpzc{L}(M,-)_{X\otimes M,N}))c \\
  & \quad\mbox{(by ordinary $\mathpzc{V}$-naturality of $M_{\mathpzc{V}}$)} 
  \end{align*}

For the last equation, recall that $\mathpzc{L}(L,-)_{M,N}\co
\mathpzc{L}(M,N)\rightarrow [\mathpzc{L}(L,M),\mathpzc{L}(L,N)]$ is
extraordinarily $\mathpzc{V}$-natural in $L$ when $M$ and $N$ are held
constant. That is, $\mathpzc{L}(L,-)_{M,N}$ has the form
$I\rightarrow [\mathpzc{L}(M,N), T(A,A)]$ for the
$\mathpzc{V}$-functor $T=
\mathrm{Hom}_{\mathpzc{V}}(\mathpzc{L}(-,M)\op\otimes
\mathpzc{L}(-,N))\mathpzc{c} \co \mathpzc{L}\op\otimes \mathpzc{L}\rightarrow \mathpzc{V}$,
where $\mathpzc{c}\co \mathpzc{L}\op\otimes \mathpzc{L}\cong
\mathpzc{L}\otimes \mathpzc{L}\op$ is the $\mathpzc{V}$-functor
mediating the symmetry of the $2$-category
$\mathpzc{V}\text{-}\mathpzc{CAT}$ of $\mathpzc{V}$-categories,
which is locally given by $c$, and
the corresponding naturality equation is

\begin{align*}
 &  [\mathpzc{L}(L,-)_{X\otimes M,N},1]
 \mathrm{Hom}_{\mathpzc{V}}(-,\mathpzc{L}(L,N))_{\mathpzc{L}(M,X\otimes
   M),\mathpzc{L}(L,X\otimes
     M)}\mathpzc{L}(-,X\otimes M)_{L,M} \\
   & = [\mathpzc{L}(M,-)_{X\otimes M,N},1]
   \mathrm{Hom}_{\mathpzc{V}}(\mathpzc{L}(M,X\otimes
   M),-)_{\mathpzc{L}(M,N),\mathpzc{L}(L,N)}\mathpzc{L}(-,N)_{L,M}
\puncteq{.}
\end{align*}
This is an equation of $\mathpzc{V}$-morphisms $\mathpzc{L}(L,M)\rightarrow
[\mathpzc{L}(X\otimes M,N),[\mathpzc{L}(M,X\otimes
M),\mathpzc{L}(L,N)]]$, and it corresponds to an equation of $\mathpzc{V}$-morphisms
$\mathpzc{L}(L,M)\otimes \mathpzc{L}(X\otimes M,N)\rightarrow[\mathpzc{L}(M,X\otimes
M),\mathpzc{L}(L,N)] $ under the
adjunction of the closed structure, which in turn corresponds to an equation of \mbox{$\mathpzc{V}$-morphisms}
$\mathpzc{L}(X\otimes M,N)\otimes \mathpzc{L}(L,M)\rightarrow [\mathpzc{L}(M,X\otimes
M),\mathpzc{L}(L,N)]$ by composition with $c$.

Recall that $\mathrm{Hom}_{\mathpzc{V}}(\mathpzc{L}(M,X\otimes
M,-),-)_{\mathpzc{L}(M,N),\mathpzc{L}(L,N)}$ corresponds to
$M_{\mathpzc{V}}$ under the adjunction, while $\mathrm{Hom}_{\mathpzc{V}}(-,\mathpzc{L}(L,N))_{\mathpzc{L}(M,X\otimes
   M),\mathpzc{L}(L,X\otimes
     M)}$ corresponds to $M_{\mathpzc{V}}c$. 
The correspondence is given by application of $-\otimes \mathpzc{L}(X\otimes
M,N)$ and composition with $e^{\mathpzc{L}(X\otimes M,N)}_{[\mathpzc{L}(M,X\otimes
M),\mathpzc{L}(L,N)]}$ (and we also compose with $c$).

Then the  transform of the left hand side of the naturality condition is:
\begin{align*}
& e^{\mathpzc{L}(X\otimes M,N)}_{[\mathpzc{L}(M,X\otimes
M),\mathpzc{L}(L,N)]} \\
& \quad (( [\mathpzc{L}(L,-)_{X\otimes M,N},1]
 \mathrm{Hom}_{\mathpzc{V}}(-,\mathpzc{L}(L,N))_{\mathpzc{L}(M,X\otimes
   M),\mathpzc{L}(L,X\otimes
     M)}\mathpzc{L}(-,X\otimes M)_{L,M} )\otimes 1_{\mathpzc{L}(X\otimes
M,N)}) c \\ 
& = e^{[\mathpzc{L}(L,X\otimes M),\mathpzc{L}(L,N)]}_{[\mathpzc{L}(M,X\otimes
M),\mathpzc{L}(L,N)]}  (( 
 \mathrm{Hom}_{\mathpzc{V}}(-,\mathpzc{L}(L,N))_{\mathpzc{L}(M,X\otimes
   M),\mathpzc{L}(L,X\otimes
     M)}\mathpzc{L}(-,X\otimes M)_{L,M} )\otimes
   \mathpzc{L}(L,-)_{X\otimes M,N}) c \\
   & \quad\mbox{(by extraordinary $\mathpzc{V}$-naturality of $e$)} \\
 & =M_{\mathpzc{V}}c(\mathpzc{L}(-,X\otimes M)_{L,M} \otimes
\mathpzc{L}(L,-)_{X\otimes M,N})c \\
& = M_{\mathpzc{V}}( 
\mathpzc{L}(L,-)_{X\otimes M,N}\otimes\mathpzc{L}(-,X\otimes M)_{L,M})
\\
& \quad\mbox{(by naturality of $c$ and $c^2=1$)}
\puncteq{.}
\end{align*}

Similarly, the transform of the right hand side of the naturality
condition is:
\begin{align*}
     & e^{\mathpzc{L}(X\otimes M,N)}_{[\mathpzc{L}(M,X\otimes
M),\mathpzc{L}(L,N)]} \\ 
& \quad(([\mathpzc{L}(M,-)_{X\otimes M,N},1]
   \mathrm{Hom}_{\mathpzc{V}}(\mathpzc{L}(M,X\otimes
   M),-)_{\mathpzc{L}(M,N),\mathpzc{L}(L,N)}\mathpzc{L}(-,N)_{L,M})\otimes
   1_{\mathpzc{L}(X\otimes
M,N)})c \\
& = e^{[\mathpzc{L}(M,X\otimes M),\mathpzc{L}(M,N)]}_{[\mathpzc{L}(M,X\otimes
M),\mathpzc{L}(L,N)]}  (( 
  \mathrm{Hom}_{\mathpzc{V}}(\mathpzc{L}(M,X\otimes
   M),-)_{\mathpzc{L}(M,N),\mathpzc{L}(L,N)}\mathpzc{L}(-,N)_{L,M}
   )\otimes \mathpzc{L}(M,-)_{X\otimes M,N}) c \\
   & \quad\mbox{(by extraordinary $\mathpzc{V}$-naturality of $e$)} \\
 & =M_{\mathpzc{V}}(\mathpzc{L}(-,N)_{L,M} \otimes
\mathpzc{L}(M,-)_{X\otimes M,N})c 
\puncteq{.}
\end{align*}
 This proves the missing equation used in the chain of equations
 above, and thus ends the proof.
\end{proof}

\section{Identification of the $\mathpzc{V}$-category of
  $T$-algebras}\label{sectionstrict}
Let $\mathpzc{V}$ be complete and cocomplete.
Let $\mathpzc{P}$ be a small and $\mathpzc{L}$ be a cocomplete
$\mathpzc{V}$-category, and denote again by $T$ the $\mathpzc{V}$-monad
$\mathrm{Lan}_H[H,1]\co
[\mathrm{ob}\mathpzc{P},\mathpzc{L}]\rightarrow
[\mathrm{ob}\mathpzc{P},\mathpzc{L}]$ from \ref{subsectionmonad} given by the Kan adjunction
$\mathrm{Lan}_{H}\dashv [H,1]$. Recall that we denote by
$[\mathrm{ob}\mathpzc{P},\mathpzc{L}]^T$ the
Eilenberg-Moore object i.e. the $\mathpzc{V}$-category of $T$-algebras, which we described in Proposition \ref{propstrictalgebras}
from \secref{propstrictalgebras} 
explicitly in the special case that $\mathpzc{V}=\mathpzc{Gray}$.

We are now going to show that
$T\text{-}\mathrm{Alg}$ is isomorphic as a $\mathpzc{V}$-category to the functor
$\mathpzc{V}$-category $[\mathpzc{P},\mathpzc{L}]$ i.e. that
$[H,1]$ is strictly monadic, which is the content of Theorem
\ref{strictmonadicitytheorem} below.

\begin{lemma}
  Let $(A,a)$ be a $T$-algebra cf. \secref{algebraaxioms}. Then $A\co
  \mathrm{ob}\mathpzc{P}\rightarrow \mathpzc{L}$ and the transforms
  $A_{PQ}\coloneqq n(a_{PQ})\co \mathpzc{P}(P,Q)\rightarrow
  \mathpzc{L}$ under the adjunction \eqref{homVadjtensor} from \secref{homVadjtensor} of the components $a_{PQ}$ of $a$ at objects
  $P,Q\in\mathpzc{P}$ have the structure of a
  $\mathpzc{V}$-functor. 
Conversely, if $A\co \mathpzc{P}\rightarrow \mathpzc{L}$ is a
$\mathpzc{V}$-functor, then the function on objects considered as a
functor $A\co \mathrm{ob}\mathpzc{P}\rightarrow \mathpzc{L}$ and the
transformation $a$ induced by the transforms  $n^{-1}(A_{PQ})$ of
the strict hom functors are the underlying data of a $T$-algebra.
\end{lemma}
\begin{proof}
  The $1$-cell $a\co TA\rightarrow A$ has component 
  \begin{align*}
    a_{Q}\co (TA)Q=
 \int^{P\in\mathrm{ob}\mathpzc{P}}\mathpzc{P}(P,Q)\otimes AP
 \rightarrow AQ
  \end{align*}
at the object $Q\in \mathpzc{P}$, and this is in turn induced from the
$\mathpzc{V}$-natural family of components
\begin{align*}
  a_{PQ}\co \mathpzc{P}(P,Q)\otimes AP\rightarrow AQ
\puncteq{.}
\end{align*}
These are elements of  $\mathpzc{L}(\mathpzc{P}(P,Q)\otimes AP,AQ)$.
Under the hom $\mathpzc{V}$-adjunction \eqref{homVadjtensor} from \secref{homVadjtensor} of the
tensor product, these correspond to elements of the internal hom
$[\mathpzc{P}(P,Q),\mathpzc{L}(AP,AQ)]$, i.e. 
\begin{align*}
  A_{PQ}\co \mathpzc{P}(P,Q)\rightarrow
  \mathpzc{L}(AP,AQ)
\puncteq{.}
\end{align*}
We now have to examine how the algebra axiom cf. \secref{algebraaxioms}
\begin{align*}
M_{[\mathrm{ob}\mathpzc{P},\mathpzc{L}]}(a,Ta)=M_{[\mathrm{ob}\mathpzc{P},\mathpzc{L}]}(a,\mu_A)
\puncteq{}
\end{align*}
transforms under the adjunction.

This is an equation of morphisms in
$[\mathrm{ob}\mathpzc{P},\mathpzc{L}]_0$, which is equivalent to
the equations
\begin{align*}
  M_{\mathpzc{L}}(a_{QR},\mathpzc{P}(Q,R)\otimes a_{PQ})=
  M_{\mathpzc{L}}(a_{PR}, (M_{\mathpzc{P}}\otimes 1_{AP})\alpha^{-1})
\end{align*}
of elements in $\mathpzc{L}(\mathpzc{P}(Q,R)\otimes
(\mathpzc{P}(P,Q)\otimes AP),AR)$ where $P,Q,R$ run through the
objects in $\mathpzc{P}$.
To apply the hom $\mathpzc{V}$-adjunction \eqref{homVadjtensor} from \secref{homVadjtensor} for
$X=\mathpzc{P}(Q,R)\otimes \mathpzc{P}(P,Q)$ and $L=AP$ and $M=AQ$, we
consider the equivalent equations
\begin{align}\label{equivalgebraaxiom}
   M_{\mathpzc{L}}(a_{QR},(\mathpzc{P}(Q,R)\otimes a_{PQ})\alpha)=
  M_{\mathpzc{L}}(a_{PR}, M_{\mathpzc{P}}\otimes 1_{AP})
\puncteq{.}
\end{align}
Applying Lemma \ref{asabovebynatofML} from \secref{asabovebynatofML} to the left hand side shows that
its transform is given by\footnote{In fact, we do not need
  the full strength of Lemma \ref{asabovebynatofML} here, and one could
   do with more elementary considerations if one was merely 
   concerned with the identification of algebras.}
\begin{align*}
  M_{\mathpzc{L}}(A_{QR}\otimes A_{PQ})
\puncteq{.}
\end{align*}
On the other hand, the image of the right hand side under
\eqref{nintermsof} from \secref{nintermsof} is determined by the following elementary
transformations \footnote{We will be short on such routine
  transformations below. The computation here should serve as an
  example for basic naturality transformations of the same kind. One
  could subsume this into another more elementary lemma.}. 
\begin{align*}
 &([\eta^{AP}_{\mathpzc{P}(Q,R)\otimes
   \mathpzc{P}(P,Q)},1]\mathpzc{L}(L,-)_{(\mathpzc{P}(Q,R)\otimes
   \mathpzc{P}(P,Q))\otimes
   AP,AR}M_{\mathpzc{L}})(a_{PR},M_{\mathpzc{P}}\otimes 1_{AP}) \\
& = [\eta^{AP}_{\mathpzc{P}(Q,R)\otimes
   \mathpzc{P}(P,Q)},1](M_{\mathpzc{V}}( \mathpzc{L}(AP,-)_{\mathpzc{P}(P,R)\otimes
  AP,AR}(a_{PR}),  \\
& \quad \quad \mathpzc{L}(AP,-)_{(\mathpzc{P}(Q,R)\otimes
  \mathpzc{P}(P,Q))\otimes AP, \mathpzc{P}(P,R)\otimes
  AR}(M_{\mathpzc{P}}\otimes 1_{AP}) ))\\
& \quad\mbox{(by the functor axiom for $\mathpzc{L}(AP,-)$)}\\
& =M_{\mathpzc{V}}( \mathpzc{L}(AP,-)_{\mathpzc{P}(P,R)\otimes
  AP,AR}(a_{PR}), \\
& \quad \quad  [\eta^{AP}_{\mathpzc{P}(Q,R)\otimes
   \mathpzc{P}(P,Q)},1]( \mathpzc{L}(AP,-)_{(\mathpzc{P}(Q,R)\otimes
  \mathpzc{P}(P,Q))\otimes AP, \mathpzc{P}(P,R)\otimes
  AR}(M_{\mathpzc{P}}\otimes 1_{AP})) ) \\
& \quad\mbox{(by ordinary $\mathpzc{V}$-naturality of
  $M_{\mathpzc{V}}$)}\\
& = M_{\mathpzc{V}}( \mathpzc{L}(AP,-)_{\mathpzc{P}(P,R)\otimes
  AP,AR}(a_{PR}), [1,\eta^{AP}_{\mathpzc{P}(P,R)}]( M_{\mathpzc{P}}))
\\
& \quad\mbox{(by ordinary $\mathpzc{V}$-naturality of
  $\eta$)}\\
& = M_{\mathpzc{V}}( ([\eta^{AP}_{\mathpzc{P}(P,R)},1]\mathpzc{L}(AP,-)_{\mathpzc{P}(P,R)\otimes
  AP,AR})(a_{PR}), M_{\mathpzc{P}})
\\
& \quad\mbox{(by extraordinary $\mathpzc{V}$-naturality of
  $M_{\mathpzc{V}}$)}\\
& = M_{\mathpzc{V}}(A_{PR},M_{\mathpzc{P}})=A_{PR}M_{\mathpzc{P}}
\end{align*}

Hence, the algebra axiom is equivalent to the equation
\begin{align}\label{functoraxiom}
  M_{\mathpzc{L}}(A_{QR}\otimes A_{PQ})=A_{PQ}M_{\mathpzc{P}}
\puncteq{,}
\end{align}
and this is exactly one of the two axioms for a
$\mathpzc{V}$-functor.

Now we want to determine the transform of the other axiom of a $T$-algebra:
\begin{align*}
  1_A=M_{[\mathrm{ob}\mathpzc{P},\mathpzc{L}]}(a,\eta_A)
\puncteq{.}
\end{align*}
 First note that this equation is equivalent
to the equations
\begin{align*}
  1_{AP}=M_{\mathpzc{L}}(a_{PP},(j_P\otimes 1)\lambda^{-1}_{AP})
\end{align*}
on objects in $\mathpzc{L}(AP,AP)$ where $P$ runs through the objects
of $\mathpzc{P}$.
In turn, these are equivalent to the equations
\begin{align*}
  \lambda_{AP}=M_{\mathpzc{L}}(a_{PP},j_P\otimes 1)
\puncteq{.}
\end{align*}
By definition of the unitor for the tensor product, the transform of
the left hand side is given by the unit $j_{AP}\co I\rightarrow
\mathpzc{L}(AP,AP)$ of the $\mathpzc{V}$-category at $AP$ (under the
identification of elements of the internal hom and morphisms in $\mathpzc{V}$). On the
other hand, it is routine to identify the transform of the right hand
side as $A_{PP}j_{P}$. Thus the second axiom of a $T$-algebra
is equivalent to
\begin{align}\label{unitfunctoraxiom}
 j_{AP}= A_{PP}j_{P}
\puncteq{,}
\end{align}
and this is exactly the other axiom of a $\mathpzc{V}$-functor.
\end{proof}

Recall that the hom object of
$[\mathrm{ob}\mathpzc{P},\mathpzc{L}]^T$  at algebras $(A,a)$
and $(B,b)$ is given by
 the equalizer
\begin{equation*}\label{equalizerhomeilenbergmooreT}
\begin{tikzpicture}
\matrix (a) [matrix of math nodes, row sep=3em, column sep=3.9em, text
height=1.5ex, text depth=0.25ex]{
  \lbrack\mathrm{ob}\mathpzc{P},\mathpzc{L} \rbrack^T(A,B) & \lbrack
  \mathrm{ob}\mathpzc{P},\mathpzc{L} \rbrack(A,B) & 
\lbrack \mathrm{ob}\mathpzc{P},\mathpzc{L} \rbrack(TA,B) \\ };
\draw[->]
($ (a-1-1.east) + (0,.0) $)
-- node [above] {}($ (a-1-2.west) + (0,.0) $);

\draw[->]
($ (a-1-2.east) + (0,.075) $)
-- node [above] {$\scriptstyle\lbrack \mathrm{ob}\mathpzc{P},\mathpzc{L}\rbrack(a,1)$}($ (a-1-3.west) + (0,.075) $);
\draw[->]
($ (a-1-2.east) + (0,-.075) $)
-- node [below] {$\scriptstyle \lbrack \mathrm{ob}\mathpzc{P},\mathpzc{L}\rbrack(1,b)T_{A,B}$}($ (a-1-3.west) + (0,-.075) $);	
\puncttikz[a-1-3]{.}		
\end{tikzpicture}
\end{equation*}
Spelling out the hom objects of
$[\mathrm{ob}\mathpzc{P},\mathpzc{L}]$ as in \ref{subsectionmonad},
this is the same as the following equalizer
\begin{equation*}
\begin{tikzpicture}
\matrix (a) [matrix of math nodes, row sep=3em, column sep=4em, text
height=1.5ex, text depth=0.25ex]{
  \lbrack\mathrm{ob}\mathpzc{P},\mathpzc{L} \rbrack^T(A,B) & \Pi_{P\in\mathrm{ob}\mathpzc{P}}\mathpzc{L}(AP,BP) & 
  \Pi_{\mathpzc{P}\in\mathrm{ob}\mathpzc{P}}\mathpzc{L}(\int^{R\in\mathrm{ob}\mathpzc{P}}\mathpzc{P}(R,P)\otimes
  AR,BP) \\ };
\draw[->]
($ (a-1-1.east) + (0,.0) $)
-- node [above] {}($ (a-1-2.west) + (0,.0) $);

\draw[->]
($ (a-1-2.east) + (0,.075) $)
-- node [above] {$\scriptstyle\Pi \mathpzc{L}(a_P,1)$}($ (a-1-3.west) + (0,.075) $);
\draw[->]
($ (a-1-2.east) + (0,-.075) $)
-- node [below] {$\scriptstyle \Pi  \mathpzc{L}(1,b_P)(T_{A,B})_{P}$}($ (a-1-3.west) + (0,-.075) $);			
\end{tikzpicture}
\end{equation*}
where  $E_PT_{A,B}= (T_{A,B})_PE_P$ for a unique $\mathpzc{V}$-morphism
$(T_{A,B})_P\co\mathpzc{L}(AP,BP)\rightarrow \mathpzc{L}(TAP,TBP)$.
By \eqref{repcoend} and the universal property of the end,
this equalizer is the same as the equalizer of the compositions with
$\Pi_{P\in\mathrm{ob}\mathpzc{P}}\mathpzc{L}(\kappa_{RP},1)$ (by
definition \eqref{unitcoend} of $\kappa$ and by Yoneda),
and since 
\begin{align*}
  n\co \mathpzc{L}(\mathpzc{P}(R,P)\otimes
  AR,BP)\cong [\mathpzc{P}(R,P),\mathpzc{L}(AR,BP)]
\puncteq{,}
\end{align*}
this is in turn the same as the following equalizer
\begin{equation*}\label{equalizerhomeilenbergmoore}
\begin{tikzpicture}
\matrix (a) [matrix of math nodes, row sep=3em, column sep=3.9em, text
height=1.5ex, text depth=0.25ex]{
  \lbrack\mathrm{ob}\mathpzc{P},\mathpzc{L} \rbrack^T(A,B)) & \Pi_{P\in\mathrm{ob}\mathpzc{P}}\mathpzc{L}(AP,BP) & 
   \Pi_{R,P\in\mathrm{ob}\mathpzc{P}}\lbrack \mathpzc{P}(R,P),\mathpzc{L}(
  AR,BP)\rbrack \\ };
\draw[->]
($ (a-1-1.east) + (0,.0) $)
-- node [above] {}($ (a-1-2.west) + (0,.0) $);

\draw[->]
($ (a-1-2.east) + (0,.075) $)
-- node [above] {$\scriptstyle \Pi n\mathpzc{L}(a_{RP},1)$}($ (a-1-3.west) + (0,.075) $);
\draw[->]
($ (a-1-2.east) + (0,-.075) $)
-- node [below] {$\scriptstyle \Pi  n \mathpzc{L}(1,b_{RP})
  (\mathpzc{P}(R,P)\otimes -)_{AR,BR}$}($ (a-1-3.west) + (0,-.075) $);	
\puncttikz[a-1-3]{.}		
\end{tikzpicture}
\end{equation*}
(where we have used that $a_P\kappa_{R,P}=a_{RP}$ for the first
morphism of the equalizer, and where we
have used that
$\mathpzc{L}(\kappa_{R,P}^A,1)(E_P)_{TA,TB}T_{A,B}=\mathpzc{L}(1,\kappa_{R,P}^B)(\mathpzc{P}(R,P)\otimes
-)_{AR,BR} (E_R)_{A,B} $ by ordinary \mbox{$\mathpzc{V}$-naturality} of
$\kappa^A_{R,P}$ in $A$ 
and that $b_P\kappa_{R,P}=b_{RP}$ for the second
morphism of the equalizer).

One can now use the Yoneda lemma to show that this is exactly the
equalizer  \eqref{functorcatequalizer} from \secref{functorcatequalizer} which defines the hom object of the functor
$\mathpzc{V}$-category: one checks that both the first  morphism given here
and the transform of  $\mathpzc{L}(A-,AR)_{RP} $
 map the identity at $AP$ to
$A_{RP}$, and both the second here and the transform of $\mathpzc{L}(AP,B-)_{PR} $ 
map the identity at $BR$ to $B_{RP}$.

Finally, note that the composition law of
$[\mathrm{ob}\mathpzc{P},\mathpzc{L}]^T$ is  induced from the
composition law of the functor category $[\mathrm{ob}\mathpzc{P},\mathpzc{L}]$, which in
turn, is induced from the composition law of $\mathpzc{L}$.

Likewise, the composition law of the functor category
$[\mathpzc{P},\mathpzc{L}]$ is induced via the evaluation functors
from the composition law in $\mathpzc{L}$, and since the evaluation 
functors $E_P\co [\mathpzc{P},\mathpzc{L}]\rightarrow \mathpzc{L}$
where $P\in\mathrm{ob}\mathpzc{P}$ factorize through $[H,1]$ and
$E_P\co [\mathrm{ob}\mathpzc{P},\mathpzc{L}]\rightarrow \mathpzc{L}$,
this means that $[H,1]\co \mathrm{ob}[\mathpzc{P},\mathpzc{L}] \rightarrow
\mathrm{ob}[\mathrm{ob}\mathpzc{P},\mathpzc{L}]^T$ and the
$\mathpzc{V}$-isomorphisms on the hom objects induced from the
comparison of equalizers above 
satisfy the $\mathpzc{V}$-functor
axiom.

Similarly, this data is shown to satisfy the unit axiom for a
$\mathpzc{V}$-functor.
This proves the following theorem. 
\begin{theorem}\label{strictmonadicitytheorem}
Given a small $\mathpzc{V}$-category $\mathpzc{P}$ and a cocomplete
$\mathpzc{V}$-category $\mathpzc{L}$, then the functor 
$[H,1]\co [\mathpzc{P},\mathpzc{L}]\rightarrow
[\mathrm{ob}\mathpzc{P},\mathpzc{L}]$ induced by the inclusion $H\co
\mathrm{ob}\mathpzc{P}\rightarrow \mathpzc{L}$ is strictly monadic for
the $\mathpzc{V}$-monad  $T=[H,1]\mathrm{Lan}_H$ given by the Kan
adjunction $\mathrm{Lan}_H\dashv [H,1]$.
In particular, the functor $\mathpzc{V}$-category
$[\mathpzc{P},\mathpzc{L}]$ is isomorphic as a $\mathpzc{V}$-category
to the Eilenberg-Moore object $[\mathrm{ob}\mathpzc{P},\mathpzc{L}]^T$
in $\mathpzc{V}\text{-}\mathpzc{CAT}$.   \qed

\end{theorem}
\begin{remark}
  An alternative strategy to achieve this result is to use an enriched
  version of Beck's monadicity theorem (see for example
  \cite[Th. II.2.1]{dubuc}). The Kan adjunction meets the
  conditions of such a theorem because  $[H,1]$ creates pointwise
  colimits, cf. Lemma \ref{H1creates} from \secref{H1creates}.

Yet another strategy for specific $\mathpzc{V}$ where an
identification of the functor $\mathpzc{V}$-category is known, is to be completely
explicit: For example, if $\mathpzc{V}=\mathpzc{Gray}$, the functor
category can be explicitly identified
(cf. \cite[Prop. 12.2]{gurskicoherencein}). With the help of the two
Transformation Lemmata
\ref{asabovebynatofML} and \ref{noticethesymmetry} from \secref{noticethesymmetry} it is then
straightforward to identify the algebra $1$-cells, algebra $2$-cells, and
algebra $3$-cells from Proposition \ref{propstrictalgebras} from
\secref{propstrictalgebras} explicitly as we have done it for algebras above.
\end{remark}
\section{Identification of  the $\mathpzc{Gray}$-category
  $\mathrm{Ps}\text{-}T\text{-}\mathrm{Alg}$ of pseudo $T$-algebras}\label{sectionGrayhom}

The $\mathpzc{Gray}$-category $\mathpzc{Tricat}(\mathpzc{P},\mathpzc{L})$ of
trihomomorphisms $\mathpzc{P}\rightarrow \mathpzc{L}$,
tritransformations, trimodifications, and perturbations
has been described by Gurski \cite[Th. 9.4]{gurskicoherencein}. The basic definitions of the
objects and the $2$-globular data of the local hom $2$-categories, i.e. trihomomorphisms,
tritransformations, trimodifications, and perturbations may be found
in \cite[4.]{gurskicoherencein}. These are of course definitions for
the general case that domain and codomain are honest
tricategories. In our case, they simplify considerably because domain
and codomain are always
$\mathpzc{Gray}$-categories.

Let again $T=[H,1]\mathrm{Lan}_H$ be the $\mathpzc{Gray}$-monad on
$[\mathrm{ob}\mathpzc{P},\mathpzc{L}]$  from
\ref{subsectionmonad} corresponding to the Kan
adjunction $\mathrm{Lan}_H\dashv [H,1]$, where $\mathpzc{P}$ is small and $\mathpzc{L}$ is cocomplete.
The aim of this section is to prove  Theorem \ref{graygleichls} below, which
states that
$\mathrm{Ps}\text{-}T\text{-}\mathrm{Alg}$  is isomorphic to the
  full sub-$\mathpzc{Gray}$-category of
  $\mathpzc{Tricat}(\mathpzc{P},\mathpzc{L})$ determined by the
  locally strict trihomomorphisms.

The general idea of the proof is to identify how the pseudo data and
axioms transform under the adjunction \eqref{homVadjtensor} from \secref{homVadjtensor} of the tensor product. The main technical
tools employed are the two Transformation Lemmata \ref{asabovebynatofML} and
\ref{noticethesymmetry} from \secref{noticethesymmetry}, and elementary identities
involving the associators and unitors $a,l,r,\alpha,\lambda,$ and $\rho$, which are implied by the pentagon and
triangle identity as presented in \ref{someproperties}.

\subsection{Homomorphisms of
  $\mathpzc{Gray}$-categories}\label{homomorphismsofgraycategories}
To characterize how $\mathrm{Ps}\text{-}T\text{-}\mathrm{Alg}$
transforms under the adjunction of the tensor product, we now
introduce the notions of Gray homomorphisms between
$\mathpzc{Gray}$-categories, say $\mathpzc{P}$ and $\mathpzc{L}$, Gray transformations, Gray modifications,
and Gray perturbations in Definitions
\ref{Grayhomomorphism}-\ref{Grayperturbation}. In fact, by reference
to $\mathrm{Ps}\text{-}T\text{-}\mathrm{Alg}$, we show that these form
a $\mathpzc{Gray}$-category
$\mathrm{Gray}(\mathpzc{P},\mathpzc{L})$. On the other hand, we
maintain that this is in fact the natural notion of a (locally
strict) trihomomorphism when domain and target are
$\mathpzc{Gray}$-categories, and when the definitions are to be given
on Gray products and
in terms of their composition laws, say $M_{\mathpzc{P}}$ and
$M_{\mathpzc{L}}$, rather than on cartesian products and in terms of the
corresponding cubical composition functors. Thus, the definitions
below are easily seen to be mild context-related modifications of the
definitions of (locally strict) trihomomorphisms between
$\mathpzc{Gray}$-categories, tritransformations, trimodifications, and
perturbations cf. \cite[4.]{gurskicoherencein}.

\begin{definition}\label{Grayhomomorphism}
  A Gray homomorphism $A\co \mathpzc{P}\rightarrow
  \mathpzc{L}$ consists of
  \begin{itemize}
  \item a function on the objects $P\mapsto AP$ denoted by the same
    letter as the trihomomorphism itself;
  \item for objects $P,Q\in\mathpzc{P}$, a strict functor $A_{PQ}\co
    \mathpzc{P}(P,Q)\rightarrow \mathpzc{L}(AP,AQ)$;
  \item for objects $P,Q,R\in\mathpzc{P}$, an adjoint equivalence
  \begin{align*}
    (\chi,\chi^\bullet)\co M_{\mathpzc{L}}(A_{QR}\otimes
    A_{PQ})\Rightarrow A_{PR}M_{\mathpzc{P}}\co
    \mathpzc{P}(Q,R)\otimes \mathpzc{P}(P,Q) \rightarrow \mathpzc{L}(AP,AR)
\puncteq{;}
  \end{align*}
\item for each object $P\in\mathpzc{P}$, an adjoint equivalence
  \begin{align*}
    (\iota,\iota^\bullet)\co j_{AP}\Rightarrow A_{PP}j_{P}\co
    I\rightarrow \mathpzc{L}(AP,AP)
\puncteq{;}
  \end{align*}
\item and three families of invertible modifications  \textbf{(GHM1)}-\textbf{(GHM3)} which are
    subject to two axioms \textbf{(GHA1)}-\textbf{(GHA2)}:
  \end{itemize}

\noindent\textbf{(GHM1)}\; For objects $P,Q,R,S\in\mathpzc{P}$, an invertible modification
  \begin{align}
    \omega_{PQRS}\co \quad\quad & [(M_{\mathpzc{P}}\otimes
    1)a^{-1},1](\chi_{PQS}) \ast [a^{-1},M_{\mathpzc{L}}]
    (\mathrm{Ten}
    (\chi_{QRS},1_{A_{PQ}})) \nonumber \\
    \Rrightarrow \quad & [1\otimes M_{\mathpzc{P}},1](\chi_{PRS}) \ast
    [1,M_{\mathpzc{L}}](\mathrm{Ten}(1_{A_{RS}},\chi_{PQR})) \nonumber
    \puncteq{}
  \end{align}
  of pseudonatural transformations
  \begin{align*}
    M_{\mathpzc{L}}(1\otimes M_{\mathpzc{L}})(A_{RS}\otimes
    (A_{QR}\otimes A_{PQ}))\Rightarrow A_{PS}(1\otimes
    M_{\mathpzc{P}}(1\otimes M_{\mathpzc{P}}))
  \end{align*}
  of strict functors
  \begin{align*}
    \mathpzc{P}(R,S)\otimes(\mathpzc{P}(Q,R)\otimes
    \mathpzc{P}(P,Q))\rightarrow \mathpzc{L}(AP,AS) \puncteq{.}
  \end{align*}
 
\noindent\textbf{(GHM2)}\; For objects $P,Q\in \mathpzc{P}$, an invertible modification
  \begin{align*}
    \gamma_{PQ}\co & [(j_Q\otimes
    1)l^{-1}_{\mathpzc{P}(P,Q)},1](\chi_{PQQ}) \ast
    [l^{-1}_{\mathpzc{P}(P,Q)},M_{\mathpzc{L}}](\mathrm{Ten}(\iota_Q,1_{A_{PQ}}))
    \Rrightarrow 1_{A_{PQ}}
  \end{align*}
of pseudonatural transformations $A_{PQ}\Rightarrow A_{PQ}\co
  \mathpzc{P}(P,Q)\rightarrow \mathpzc{L}(AP,AQ)$ .

 \noindent\textbf{(GHM3)}\;  For objects $P,Q\in \mathpzc{P}$, an invertible modification
  \begin{align*}
    \delta_{PQ}\co  & 1_{A_{PQ}}\Rrightarrow [(1\otimes
    j_P)r^{-1}_{\mathpzc{P}(P,Q)},1](\chi_{PPQ}) \ast
    [r^{-1}_{\mathpzc{P}(P,Q)},M_{\mathpzc{L}}](\mathrm{Ten}(1_{A_{PQ}},\iota_P))
  \end{align*}
  of pseudonatural transformations $A_{PQ}\Rightarrow A_{PQ}\co
  \mathpzc{P}(P,Q)\rightarrow \mathpzc{L}(AP,AQ)$ .

\noindent\textbf{(GHA1)}\; For objects $P,Q,R,S,T\in \mathpzc{P}$, the
following equation of vertical composites of whiskered modifications is required:
  \begin{align*}
    & [1\otimes (1\otimes M_{\mathpzc{P}}),1](\omega)\ast
    1_{[M_{\mathpzc{L}}(1\otimes M_{\mathpzc{L}})](\mathrm{Ten}(1,\mathrm{Ten}(1,\chi)))} \\
    \diamond\quad & 1_{[(M_{\mathpzc{P}}\otimes
    M_{\mathpzc{P}})a^{-1},1](\chi)}\ast
    [a^{-1},M_{\mathpzc{L}}](\mathrm{Ten}(\Sigma_{\chi,\chi}))\\
    \diamond\quad& [(M_{\mathpzc{P}}\otimes 1) a^{-1},1]
    (\omega)\ast 1_{[a^{-1}a^{-1},M_{\mathpzc{L}}(M_{\mathpzc{L}}\otimes 1)](\mathrm{Ten}(\mathrm{Ten}(\chi,1),1))} \\
    = \quad \quad& 1_{[1\otimes ( M_{\mathpzc{P}}(1\otimes M_{\mathpzc{P}})),1](\chi)}\ast [1,M_{\mathpzc{L}}]
    (\mathrm{Ten}(1,\omega)) \\
    \diamond\quad&[1\otimes ((M_{\mathpzc{P}}\otimes 1)a^{-1}),1]
    (\omega)\ast 1_{[a^{-1}(1\otimes
      a^{-1}),M_{\mathpzc{L}}(M_{\mathpzc{L}}\otimes 1)](\mathrm{Ten}(\mathrm{Ten}(1,\chi),1))} \\
    \diamond\quad& 1_{[(M_{\mathpzc{P}}(1\otimes
    M_{\mathpzc{P}})a)\otimes 1,1](\chi)}\ast [(a\otimes
    1)a^{-1}a^{-1},M_{\mathpzc{L}}]
    (\mathrm{Ten}(\omega,1)) 
  \end{align*}
  To save space, we here employed the notation that vertical
  composition binds less strictly than horizontal composition $\ast$,
  which is also indicated by a line break. Also $\mathrm{Ten}$ always
  denotes the corresponding strict hom functor of the
  $\mathpzc{Gray}$-functor $\mathrm{Ten}\co \mathpzc{Gray}\otimes
  \mathpzc{Gray}\rightarrow \mathpzc{Gray}$ cf. eq. \eqref{Tenhom}
  from \secref{Tenhom} above.
It is to be noted that in each vertical factor there appears only one
nontrivial horizontal factor, and this applies generally to the
following definitions.

 The axiom is an equation of $2$-cells i.e. modifications
  \begin{align*}
    & [((M_{\mathpzc{P}}(M_{\mathpzc{P}}\otimes 1))\otimes
    1)a^{-1}a^{-1},1](\chi_{PQT}) \\
    \ast \quad &[((M_{\mathpzc{P}}\otimes 1)\otimes
    1)a^{-1}a^{-1},M_{\mathpzc{L}}](\mathrm{Ten}(\chi_{QRT},1)) \\
    \ast \quad &
    [a^{-1}a^{-1},M_{\mathpzc{L}}(M_{\mathpzc{L}}\otimes 1)] (\mathrm{Ten}(\mathrm{Ten}(\chi_{RST},1),1)) \\
    \Rrightarrow \quad\quad  & [1\otimes ( M_{\mathpzc{P}}(1\otimes M_{\mathpzc{P}})),1] (\chi_{PST})\\
    \ast \quad & [1\otimes (1\otimes M_{\mathpzc{P}}),M_{\mathpzc{L}}]
    (\mathrm{Ten}(1,\chi_{PRS})) \\
    \ast \quad & [1,M_{\mathpzc{L}}(1\otimes M_{\mathpzc{L}})](\mathrm{Ten}(1,\mathrm{Ten}(1,\chi_{PQR})))
  \end{align*}
  between $1$-cells i.e. pseudonatural transformations
  \begin{align*}
    M_{\mathpzc{L}}(1\otimes (M_{\mathpzc{L}}(1\otimes
    M_{\mathpzc{L}}))) (A_{ST} \otimes(A_{RS} \otimes (A_{QR}\otimes
    A_{PQ})))\Rightarrow A_{PT} M_{\mathpzc{P}}(1\otimes
    (M_{\mathpzc{P}}(1\otimes M_{\mathpzc{P}})))
  \end{align*}
  of strict functors
  \begin{align*}
    \mathpzc{P}(S,T)\otimes
    (\mathpzc{P}(R,S)\otimes(\mathpzc{P}(Q,R)\otimes
    \mathpzc{P}(P,Q)))\rightarrow \mathpzc{L}(AP,AT) \puncteq{.}
  \end{align*}
  We remark that we chose another bracketing than in the
  definition of a trihomomorphism in the references
  \cite{gurskicoherencein} and \cite{gordonpowerstreet}. The
  difference is of course not substantive.

\noindent\textbf{(GHA2)}\; For objects $P,Q,R\in \mathpzc{P}$, the
following equation of modifications is required:
\begin{align*}
  & 1_{\chi}\ast
  [1,M_{\mathpzc{L}}](\mathrm{Ten}(1_{1_{A_{QR}}},\gamma_{PQ}))
\\
  \diamond \quad & [1\otimes
  ((j_Q\otimes 1)l^{-1}_{\mathpzc{P}(P,Q)}),1] (\omega_{PQQR}) \ast
  1_{[r^{-1}_{\mathpzc{P}(Q,R)}\otimes
  1,M_{\mathpzc{L}}(M_{\mathpzc{L}}\otimes 1)](\mathrm{Ten}(\mathrm{Ten}(1,\iota),1))} 
\\
  = \quad \quad & 1_{[((1\otimes j)r^{-1})\otimes 1](\chi)}\ast
  [1,M_{\mathpzc{L}}](\mathrm{Ten}(\delta^{-1}_{QR},1_{1_{A_{PQ}}}))
\end{align*}
This is an equation of $2$-cells i.e. modifications
\begin{align*}
  \chi_{PQR}\Rrightarrow \chi_{PQR}\co M_{\mathpzc{L}}(A_{QR}\otimes
  A_{PQ})\Rightarrow A_{PR}M_{\mathpzc{L}}\co \mathpzc{P}(Q,R)\otimes
  \mathpzc{P}(P,Q)\rightarrow \mathpzc{L}(AP,AR ) \puncteq{.}
\end{align*}
\end{definition}

\begin{definition}\label{Graytransformation}
  Let $A,B\co \mathpzc{P}\rightarrow \mathpzc{L}$ be homomorphisms of
  $\mathpzc{Gray}$-categories. A Gray transformation $f\co
  A\Rightarrow B$ consists of 
  \begin{itemize}
  \item a family $(f_p)_{P\in\mathrm{ob}\mathpzc{P}}$ of objects $f_P\co AP\rightarrow BP$ in
    $\mathpzc{L}(AP,BP)$ ;
  \item for objects $P,Q\in\mathpzc{P}$, an adjoint equivalence
    \begin{align*}
      (f_{PQ},f_{PQ}^\bullet)\co \mathpzc{L}(AP,f_Q)A_{PQ}\Rightarrow
      \mathpzc{L}(f_P,BQ)B_{PQ}\co \mathpzc{P}(P,Q)\rightarrow
      \mathpzc{L}(AP,BQ) 
\puncteq{;}
    \end{align*}
\item and two families of invertible modifications
  \textbf{(GTM1)}-\textbf{(GTM2)} which are subject to three axioms \textbf{(GTA1)}-\textbf{(GTA3)}:
\end{itemize}

\noindent\textbf{(GTM1)}\; For objects $P,Q,R\in \mathpzc{P}$, an invertible modification
    \begin{align*}
      \Pi_{PQR}\co \quad\quad\quad & [1
      ,\mathpzc{L}(f_P,BR)](\chi^B_{PQR}) \ast [1
      ,M_{\mathpzc{L}}](\mathrm{Ten}(1_{B_{QR}}, f_{PQ}))\ast  [1
      ,M_{\mathpzc{L}}](\mathrm{Ten}(f_{QR},1_{A_{PQ}})) \\
\Rrightarrow \quad\quad &   [M_{\mathpzc{P}},1](f_{PR})\ast [1 ,\mathpzc{L}(AP,f_R)](\chi^A_{PQR} )
    \end{align*}
of pseudonatural transformations
$M_{\mathpzc{L}}(( \mathpzc{L}(AQ,f_R)A_{QR})\otimes A_{PQ})\Rightarrow
\mathpzc{L}(f_P,BR) B_{PR}M_{\mathpzc{P}}$ of strict functors $
\mathpzc{P}(Q,R)\otimes \mathpzc{P}(P,Q)\rightarrow
\mathpzc{L}(AP,BR)$ ;

\noindent\textbf{(GTM1)}\; For each object $P\in\mathpzc{P}$, an invertible modification
    \begin{align*}
M_{P}\co  [j_P,1](f_{PP})\ast  [1,\mathpzc{L}(AP,f_P)](\iota^A_P) 
      \Rrightarrow  [1,\mathpzc{L}(f_P,BP)](\iota_P^B)
    \end{align*}
of pseudonatural transformations $f_P\Rightarrow
  \mathpzc{L}(f_p,BP)B_{PP}j_P\co I\rightarrow \mathpzc{L}(AP,BP)$ .

\noindent\textbf{(GTA1)}\; For objects $P,Q,R,S\in\mathpzc{P}$, the
following equation of vertical composites of whiskered  modifications is required:
\begin{align*}
 & 1_{[M_{\mathpzc{P}}(1\otimes M_{\mathpzc{P}}),1](f_{PS})}\ast [1,\mathpzc{L}(AP,f_S)](\omega^A) \\
\diamond \quad & [(M_{\mathpzc{P}}\otimes 1)a^{-1},1](\Pi) \ast
1_{[1\otimes a^{-1},M_{\mathpzc{L}}(1\otimes M_{\mathpzc{L}}
  )](\mathrm{Ten}(\mathrm{Ten}(1,\chi^A
),1))}\\
\diamond \quad & 1_{[(M_{\mathpzc{P}}\otimes
  1)a^{-1},\mathpzc{L}(f_P,BS)](\chi^B)}\ast
1_{[(M_{\mathpzc{P}}\otimes 1)a^{-1},M_{\mathpzc{L}}]\mathrm{Ten}(1,f_{PQ})}\ast[a^{-1},M_{\mathpzc{L}}]
(\mathrm{Ten}(\Pi,1_{1_{A_{PQ}}})) \\
\diamond \quad & 1_{[(M_{\mathpzc{P}}\otimes
  1)a^{-1},\mathpzc{L}(f_P,BS)](\chi^B)}\ast
[a^{-1},M_{\mathpzc{L}}](\mathrm{Ten}(\Sigma_{\chi^B,f_{PQ}})) \\
& \quad \ast 1_{[a^{-1},M_{\mathpzc{L}}(M_{\mathpzc{L}}\otimes 1)]\mathrm{Ten}(\mathrm{Ten}(1,f_{QR}),1)}\ast 1_{[a^{-1},M_{\mathpzc{L}}(M_{\mathpzc{L}}\otimes 1)]\mathrm{Ten}(\mathrm{Ten}(f_{RS},1),1)} \\
= \quad \quad &   [1\otimes M_{\mathpzc{P}},1](\Pi) \ast 1_{[1,
\mathpzc{L}(AP,f_{S})M_{\mathpzc{L}}]\mathrm{Ten}(1,\chi^A
)} \\
\diamond \quad &  1_{[1\otimes M_{\mathpzc{P}},\mathpzc{L}(f_P,BS)](\chi^B)} \ast 1_{[1\otimes M_{P},M_{\mathpzc{L}}]\mathrm{Ten}(1,f_{PR})}\ast
[1,M_{\mathpzc{L}}](\mathrm{Ten}(\Sigma^{-1}_{f_{RS},\chi^A})) \\
\diamond \quad &  1_{[1\otimes M_{\mathpzc{P}},\mathpzc{L}(f_P,BS)](\chi^B)} \ast
[1,M_{\mathpzc{L}}](\mathrm{Ten}(1_{1_{B_{RS}}},\Pi))  \ast 1_{[a^{-1},M_{\mathpzc{L}}(M_{\mathpzc{L}}\otimes
1)]\mathrm{Ten}(\mathrm{Ten}(f_{RS},1),1)}\\
\diamond \quad & [1,\mathpzc{L}(f_P,BS)] (\omega^B )\ast 1_{[1,M_{\mathpzc{L}}(M_{\mathpzc{L}}\otimes 1)]
(\mathrm{Ten}(1,\mathrm{Ten}(1,f_{PQ})))} \\
& \quad  \ast
1_{[a^{-1},M_{\mathpzc{L}}(M_{\mathpzc{L}}\otimes 1)]\mathrm{Ten}(\mathrm{Ten}(1,f_{QR}),1)} \ast
1_{[a^{-1},M_{\mathpzc{L}}(M_{\mathpzc{L}}\otimes
1)]\mathrm{Ten}(\mathrm{Ten}(f_{RS},1),1)}
\end{align*}
This is an equation of modifications
\begin{align*}
\ast \quad & [(M_{\mathpzc{P}}\otimes
  1)a^{-1},\mathpzc{L}(f_P,BS)](\chi^B)\\
 \ast \quad & [a^{-1},\mathpzc{L}(f_P,BS)M_{\mathpzc{L}}](\mathrm{Ten}(\chi^B,1)) \\
 \ast \quad & [1,M_{\mathpzc{L}}(M_{\mathpzc{L}}\otimes 1)]
(\mathrm{Ten}(1,\mathrm{Ten}(1,f_{PQ}))) \\
 \ast \quad & [a^{-1},M_{\mathpzc{L}}(M_{\mathpzc{L}}\otimes
 1)]\mathrm{Ten}(\mathrm{Ten}(1,f_{QR}),1) \\
 \ast \quad &  [a^{-1},M_{\mathpzc{L}}(M_{\mathpzc{L}}\otimes
1)]\mathrm{Ten}(\mathrm{Ten}(f_{RS},1),1) \\
\Rrightarrow \quad\quad  & [M_{\mathpzc{P}}(1\otimes
M_{\mathpzc{P}}),1]f_{PS} \\
 \ast \quad & [1\otimes M_{\mathpzc{P}},
\mathpzc{L}(1,f_S)](\chi
^A)\\
 \ast \quad & [1,
\mathpzc{L}(1,f_S)M_{\mathpzc{L}}]\mathrm{Ten}(1,\chi
^A)
\end{align*}
of pseudonatural transformations
\begin{align*}
\mathpzc{L}(AR,f_S)M_{\mathpzc{L}}(1\otimes M_{\mathpzc{L}})(A_{RS}\otimes (A_{QR}\otimes A_{PQ}))\Rightarrow \mathpzc{L}(f_P,BS)B_{PS} M_{\mathpzc{P}}  (1\otimes M_{\mathpzc{P}})
\end{align*}
of strict functors
\begin{align*}
  \mathpzc{P}(R,S)\otimes (\mathpzc{P}(Q,R)\otimes
  \mathpzc{P}(P,Q))\rightarrow \mathpzc{L}(AP,BS)
\puncteq{.}
\end{align*}

\noindent\textbf{(GTA2)}\; For objects $P,Q\in\mathpzc{P}$, the
following equation of vertical composites of whiskered  modifications is required:
\begin{align*}
& 1_{f_{PQ}}\ast[1,\mathpzc{L}(AP,f_Q)] (\gamma^A_{PQ} )\\
\diamond \quad & [(j_{Q}\otimes 1)l^{-1}_{\mathpzc{P}(P,Q)},1]
(\Pi_{PQQ}) \ast
1_{[l^{-1}_{\mathpzc{P}(P,Q)},M_{\mathpzc{L}}](\mathrm{Ten}([1,\mathpzc{L}(1,f_Q)](\iota^A_Q),1_{A_{PQ}}))}\\
= \quad \quad  & [1,\mathpzc{L}(f_P,BQ)] (\gamma^B_{PQ})\ast
1_{f_{PQ}} \\
\diamond \quad & 1_{[(j_Q\otimes
  1)l^{-1}_{\mathpzc{P}(P,Q)},\mathpzc{L}(f_P,1)](\chi^B_{PQQ})}\ast
[l^{-1}_{\mathpzc{P}(P,Q)},M_{\mathpzc{L}}](\mathrm{Ten}(\Sigma
^{-1}_{\iota^B_Q,f_{PQ}})) \\
\diamond \quad &
1_{[(j_Q\otimes
  1)l^{-1}_{\mathpzc{P}(P,Q)},\mathpzc{L}(f_P,1)](\chi^B_{PQQ})}\ast
1_{[((B_{QQ}j_{Q})\otimes 1)l^{-1}_{\mathpzc{P}(P,Q)},M_{\mathpzc{L}}](\mathrm{Ten}(1,f_{PQ}))}\ast[l^{-1}_{\mathpzc{P}(P,Q)},M_{\mathpzc{L}}](\mathrm{Ten}(M_Q,1_{1_{A_{PQ}}}))
\end{align*}
This is an equation of modifications
\begin{align*}
&   [(j_{Q}\otimes 1)l^{-1}_{\mathpzc{P}(P,Q)},\mathpzc{L}(f_P,BQ)](\chi_{PQQ})\ast
   [((B_{QQ}j_{Q})\otimes
   1)l^{-1}_{\mathpzc{P}(P,Q)},M_{\mathpzc{L}}](\mathrm{Ten}(1,f_{PQ}))\\
\ast \quad & 
   [((j_{Q}\otimes
   1)l^{-1}_{\mathpzc{P}(P,Q)},M_{\mathpzc{L}}](\mathrm{Ten}(f_{QQ},1))
\ast
[l^{-1}_{\mathpzc{P}(P,Q)},\mathpzc{L}(AP,f_Q)](\iota_Q) \\
\Rrightarrow \quad\quad & f_{PQ}
\puncteq{.}
\end{align*}

\noindent\textbf{(GTA3)}\; For objects $P,Q\in\mathpzc{P}$, the
following equation of vertical composites of whiskered  modifications is required:
\begin{align*}
 & [1,\mathpzc{L}(f_P,BQ)]( (\delta^B)^{-1}_{PQ})\ast 1_{f_{PQ}} \\
\diamond \quad & 1_{[(1\otimes j_P)r^{-1}_{\mathpzc{P}(P,Q)},\mathpzc{L}(f_P,BQ)](\chi^B_{PPQ})}\ast
[r^{-1}_{\mathpzc{P}(P,Q)},M_{\mathpzc{L}}](\mathrm{Ten}(1_{1_{B_{PQ}}},M_P))\ast
1_{f_{PQ}}\\
  = \quad \quad & 1_{f_{PQ}}\ast [1,\mathpzc{L}(AP,f_Q)]( (\delta^A)^{-1}_{PQ}) \\
\diamond \quad & [(1\otimes j_{P})r^{-1}_{\mathpzc{P}(P,Q)},1]
(\Pi_{PPQ})\ast 1_{[r^{-1}_{\mathpzc{P}(P,Q)},M_{\mathpzc{L}}](\mathrm{Ten}(1_{\mathpzc{L}(AP,f_Q)A_{PQ}},\iota^A_{P}))} \\
\diamond \quad & 1_{[(1\otimes j_P)r^{-1}_{\mathpzc{P}(P,Q)},\mathpzc{L}(f_P,BQ)](\chi^B_{PPQ})}\ast
1_{[(1\otimes j_P)r^{-1}_{\mathpzc{P}(P,Q)},M_{\mathpzc{L}}](\mathrm{Ten}(1_{B_{PQ}},f_{PP}))}\ast
[r^{-1}_{\mathpzc{P}(P,Q)},M_{\mathpzc{L}}](\mathrm{Ten}(\Sigma^{-1}_{f_{PQ},\iota^A_P}))
\end{align*}
This is an equation of modifications
\begin{align*}
& [(1\otimes j_P)r^{-1}_{\mathpzc{P}(P,Q)},\mathpzc{L}(f_P,BQ)](\chi^B_{PPQ})\ast [(1\otimes j_P)r^{-1},M_{\mathpzc{L}}](\mathrm{Ten}(1_{B_{PQ}},f_{PP}))\\
 \ast\quad & [r^{-1}_{\mathpzc{P}(P,Q)},M_{\mathpzc{L}}](\mathrm{Ten}(1_{B_{PQ}},[1,\mathpzc{L}(AP,f_P)](\iota_P)))\ast
 f_{PQ}  \\
\Rrightarrow \quad\quad & f_{PQ}
\puncteq{.}
\end{align*}
\end{definition}

\begin{definition}\label{Graymodification}
  Let $f,g\co A\Rightarrow B\co \mathpzc{P}\rightarrow \mathpzc{L}$ be
  Gray transformations. A Gray modification $\alpha\co f\Rrightarrow g $ consists of
  \begin{itemize}
  \item a family $(\alpha_P)_{P\in\mathrm{ob}\mathpzc{P}}$ of
    $1$-cells $\alpha_P\co f_P\rightarrow g_P$ in $\mathpzc{L}(AP,BP)$
    ;
  \item and one family of invertible modifications \textbf{(GMM1)} which is
    subject to two axioms \textbf{(GMA1)}-\textbf{(GMA2)}:
  \end{itemize}
\noindent\textbf{(GMM1)}\; 
For objects $P,Q\in\mathpzc{P}$, an invertible modification
    \begin{align*}
      \alpha_{PQ}\co   [B_{PQ},1](\mathpzc{L}(\alpha_P,BQ))\ast f_{PQ}
      \Rrightarrow g_{PQ}\ast [A_{PQ},1](\mathpzc{L}(AP,\alpha_Q))
\puncteq{}
    \end{align*}
of pseudonatural transformations
\begin{align*}
 \mathpzc{L}(AP,f_Q)A_{PQ}\Rightarrow \mathpzc{L}(g_P,BQ)B_{PQ}\co
  \mathpzc{P}(P,Q)\rightarrow \mathpzc{L}(AP,BQ)
\puncteq{.}
\end{align*}

\noindent\textbf{(GMA1)}\; For objects $P,Q,R\in\mathpzc{P}$, the
following equation of vertical composites of whiskered  modifications is required:
\begin{align*}
& \Pi^g\ast 1_{[1,M_{\mathpzc{L}}](\mathrm{Ten}([A_{QR},1](\mathpzc{L}(AQ,\alpha_R)),1))} \\
\diamond \quad &
1_{M_{\mathpzc{Gray}}(1,\chi^B_{PQR})}\ast 1_{[1,M_{\mathpzc{L}}](\mathrm{Ten}(1,g_{PQ}))}\ast[1,M_{\mathpzc{L}}](\mathrm{Ten}(\alpha_{QR},1_{1_{A_{PQ}}})) \\
\diamond \quad &
1_{M_{\mathpzc{Gray}}(1,\chi^B_{PQR})}\ast[1,M_{\mathpzc{L}}](\mathrm{Ten}(1_{1_{B_{QR}}},\alpha_{PQ}))\ast  1_{[1,M_{\mathpzc{L}}](\mathrm{Ten}(f_{QR},1_{A_{PQ}}))} \\
\diamond \quad &
M_{\mathpzc{Gray}}(\Sigma_{\mathpzc{L}(\alpha_P,BR),\chi^B_{PQR}})\ast
1_{[1,M_{\mathpzc{L}}](\mathrm{Ten}(1_{B_{QR}},f_{PQ}))}\ast 1_{[1,M_{\mathpzc{L}}](\mathrm{Ten}(f_{QR},1_{A_{PQ}}))}\\
  =\quad\quad & 1_{[M_{\mathpzc{P}},1](\alpha_{PR})}\ast
  M_{\mathpzc{Gray}}(\Sigma_{\mathpzc{L}(AP,\alpha_P),\chi^A_{PQR}}) \\
\diamond \quad & [M_{\mathpzc{P}},1](\alpha_{PR})\ast 1_{M_{\mathpzc{Gray}}(1,\chi^A_{PQR})}\\
\diamond \quad &1_{[B_{PR}M_{\mathpzc{P}},1](\mathpzc{L}(\alpha_{P},BR))}\ast \Pi^f 
\end{align*}
This is an equation of modifications
\begin{align*}
& [B_{PR}M_{\mathpzc{P}},1](\mathpzc{L}(\alpha_{P},BR)) \\ 
\ast \quad &1_{M_{\mathpzc{Gray}}(1,\chi^B_{PQR})}\\ 
\ast \quad & [1,M_{\mathpzc{L}}](\mathrm{Ten}(1_{B_{QR}},f_{PQ}))\\ 
\ast \quad &
[1,M_{\mathpzc{L}}](\mathrm{Ten}(f_{QR},1_{A_{PQ}}))\\
\Rrightarrow \quad\quad & [M_{\mathpzc{P}},1](g_{PR}) \\ 
\ast \quad & M_{\mathpzc{Gray}}(1,\chi^A_{PQR})\\ 
\ast \quad &[1,M_{\mathpzc{L}}](\mathrm{Ten}([A_{QR},1](\mathpzc{L}(AQ,\alpha_R)),1))
\end{align*}
of pseudonatural transformations
\begin{align*}
  M_{\mathpzc{L}}((\mathpzc{L}(AQ,f_R)A_{QR})\otimes
  A_{PQ})\Rightarrow \mathpzc{L}(g_P,BR)B_{PR}M_{\mathpzc{P}} 
\puncteq{.}
\end{align*}

\noindent\textbf{(GMA2)}\; For an  object $P\in\mathpzc{P}$, the
following equation of vertical composites of whiskered  modifications is required:
\begin{align*}
&   M^g\ast 1_{\alpha_P}\\
  \diamond \quad & 1_{[j_P,1](g_{PP})}\ast M_{\mathpzc{Gray}}(\Sigma_{\mathpzc{L}(AP,\alpha_P),\iota^A_P}) \\
\diamond \quad & [j_{P},1](\alpha_{PP})\ast 1_{[1,\mathpzc{L}(AP,f_P)](\iota^A_{P})}\\
= \quad\quad &  M_{\mathpzc{Gray}}(\Sigma_{\mathpzc{L}(\alpha_P,BP),\iota^B_P}) \\
\diamond \quad  & 1_{[l^{-1}_I,M_{\mathpzc{L}}](\mathrm{Ten}(1_{G_{PP}j_P},\alpha_P))}\ast M^f
\end{align*}
This is an equation of modifications
\begin{align*}
[l_I^{-1},M_{\mathpzc{L}}](\mathrm{Ten}(1_{G_{PP}j_P},\alpha_P))\ast f_{PP}\ast[l^{-1}_I,M_{\mathpzc{L}}](\mathrm{Ten}(1_{f_P},\iota^A_P))
\Rrightarrow  [l_I^{-1},M_{\mathpzc{L}}](\mathrm{Ten}(\iota^G_P,1_{g_P}))\ast\alpha_P
\end{align*}
of pseudonatural transformations
\begin{align*}
  f_P\Rightarrow M_{\mathpzc{L}}((G_{PP}j_P)\otimes g_P)l_I^{-1}\co
  I\rightarrow \mathpzc{L}(AP,BP)
\puncteq{.}
\end{align*}
\end{definition}

\begin{definition}\label{Grayperturbation}
Let $\alpha,\beta\co f\Rrightarrow g\co A\Rightarrow B\co
\mathpzc{P}\rightarrow \mathpzc{L}$ be  Gray modifications. A Gray
perturbation $\Gamma\co \alpha\Rrrightarrow \beta$ consists of
\begin{itemize}
\item a family of $2$-cells $\Gamma_P\co
  \alpha_P\Rightarrow \beta_P$ in $\mathpzc{L}(AP,BP)$;
\item subject to one axiom \textbf{(GPA1)}:
\end{itemize}

\noindent\textbf{(GPA1)}\; For objects $P,Q\in\mathpzc{P}$ the
following equation of vertical composites of whiskered modifications is required:
 \begin{align*}
    \alpha_{PQ}\diamond
    ([B_{PQ},1](\mathpzc{L}(\Gamma_P,BP))\ast 1_{f_{PQ}}) = (1_{g_{PQ}}\ast
    [A_{PQ},1](\mathpzc{L}(AP,\Gamma_Q)))\diamond \beta_{PQ}
\puncteq{}
  \end{align*}
This is an equation of modifications
\begin{align*}
[B_{PQ},1](\mathpzc{L}(\alpha_P,BP))\ast  f_{PQ}\Rightarrow 
[A_{PQ},1](\mathpzc{L}(AP,\beta_Q))
\puncteq{.}
\end{align*}
\end{definition}
\subsection{The correspondence of Gray homomorphisms and pseudo
  algebras}
The following theorem is one of the main results, and forms the first
part of the promised correspondence of pseudo algebras and locally
strict trihomomorphisms.
\begin{theorem}\label{graygleichtransform}
 Let  $\mathpzc{P}$ be a small $\mathpzc{Gray}$-category and
  $\mathpzc{L}$ be a cocomplete $\mathpzc{Gray}$-category, and let $T$
  be the monad corresponding to the Kan adjunction.
  Then the notions of Gray homomorphism, Gray transformation, Gray
  modification, and Gray perturbation are precisely the transforms of
  the notions of a pseudo algebra, a pseudo functor, a pseudo
  transformation, and a pseudo modification respectively for the monad
  $T=[H,1]\mathrm{Lan}_{H}$ on $[\mathrm{ob}\mathpzc{P},\mathpzc{L}]$.\qed
\end{theorem}

 We only present parts of the proof explicitly. The
 proof involves the determination of transforms under the
 hom $\mathpzc{Gray}$-adjunction \eqref{homVadjtensor} from
 \secref{homVadjtensor}. These determinations involve
 the pentagon identity \eqref{pentagonidentity} and triangle identity
 \eqref{triangleidentity} from \secref{pentagonidentity} for associators
 and unitors both of the tensor
 products and of the monoidal category $\mathpzc{Gray}$. We also need naturality of these associators and
 unitors as presented in the same paragraph. On the other hand, there are elementary
 identities due to 
 naturality, which are  similar to the one we displayed for the transform of the right hand side
 of the algebra axiom \eqref{equivalgebraaxiom} above in
 \ref{sectionstrict}., and then there is heavy use
 of the two technical Transformation Lemmata \ref{asabovebynatofML} and
 \ref{noticethesymmetry} from \secref{noticethesymmetry}. In pursuing the proof, one quickly
 notices that many of the determinations of transforms are similar to
 each other. While we cannot display all of the computations, it is
 our aim to at least characterize the arguments needed for these different classes of
 transforms. Thus, in the lemmata below we provide examples that should
 serve as a complete guideline for the rest of the proof.

 \begin{lemma}
  Taking transforms under the adjunction \eqref{homVadjtensor} of the tensor
  product from
  \secref{homVadjtensor}  induces a one-to-one correspondence of Gray homomorphisms $\mathpzc{P}\rightarrow
   \mathpzc{L}$ and pseudo algebras for the monad
   $T=[H,1]\mathrm{Lan}_H$ on $[\mathrm{ob}\mathpzc{P},\mathpzc{L}]$.
 \end{lemma}

 \begin{proof}

Let $(A,a,m,i,\pi,\lambda,\rho)$ be a pseudo $T$-algebra. From the
identification of algebras we know that the components $a_{PQ}$
of the $1$-cell
$a\co TA\rightarrow A$ transform into strict functors $A_{PQ}\co
\mathpzc{P}(P,Q)\rightarrow \mathpzc{L}(AP,AQ)$ under the
hom $\mathpzc{Gray}$-adjunction \eqref{homVadjtensor} from \secref{homVadjtensor} of the tensor
product.

Since the adjoint equivalences $m$ and $i$ replace the two algebra
axioms, we define adjoint equivalences
$(\chi_{PQR},\chi^\bullet_{PQR})\coloneqq
(n\mathpzc{L}(\alpha,1))(m_{PQR},m_{PQR}^\bullet)$ for objects $P,Q,R\in\mathpzc{P}$
and $(\iota_{P},\iota_{P}^\bullet)\coloneqq
(n\mathpzc{L}(\lambda,1))(i_P,i_P^\bullet)$ for an object
$P\in\mathpzc{P}$. We have already determined domain and codomain of
these transforms in the
identification of the algebra axiom, and the adjoint
equivalences in Definition \ref{Grayhomomorphism} from \secref{Grayhomomorphism} do indeed replace
the axioms of a $\mathpzc{Gray}$-functor cf. \eqref{functoraxiom} and
\eqref{unitfunctoraxiom} in \ref{sectionstrict}. above.

Next we have to show that the transforms of the components of the
invertible $3$-cells $\pi,\lambda$, and $\rho$ correspond to the
invertible modifications $\omega_{PQRS}$, $\gamma_{PQ}$, and
$\delta_{PQ}$ in the definition of a Gray homomorphism.

The components of $\pi$ at objects $P,Q,R,S\in\mathpzc{P}$ are
invertible $3$-cells in $\mathpzc{L}$. We apply the invertible strict functor $\mathpzc{L}((1\otimes
\alpha)\alpha,1)$ to bring them into a form where we can apply the
hom $\mathpzc{Gray}$-adjunction \eqref{homVadjtensor} from \secref{homVadjtensor}. We then obtain
invertible $3$-cells
$\mathpzc{L}((1\otimes \alpha)\alpha,1)(\pi_{PQRS} )$ of the form
\begin{align*}
&
M_{\mathpzc{L}}(m_{PQS}, 1_{\alpha(((M_{\mathpzc{P}}\otimes 1)a^{-1})\otimes 1)})
\ast M_{\mathpzc{L}}(m_{QRS},1_{(1_{\mathpzc{P}(R,S)\otimes
      \mathpzc{P}(Q,R)}\otimes a_{PQ})\alpha (a^{-1}\otimes 1)}) \\
  \Rrightarrow\quad & M_{\mathpzc{L}}(m_{PRS},1_{(\mathpzc{P}(R,S)\otimes (
  (M_{\mathpzc{P}}\otimes 1_{AP})\alpha^{-1}))})\ast
  M_{\mathpzc{L}}(1_{a_{RS}}, \mathpzc{L}(\alpha,1)
  (\mathpzc{P}(R,S)\otimes m_{PQR}))
\puncteq{,}
\end{align*}
where we have already used the pentagon identity
\eqref{pentagonidentity} from \secref{pentagonidentity} for $\alpha$ and $a$.

The computations below then determine the transforms of the horizontal
factors and show that these coincide precisely with the horizontal
factors in the domain and codomain of the invertible modification
$\omega_{PQRS}$ in Definition \ref{Grayhomomorphism} from \secref{Grayhomomorphism}. For brevity we
have suppressed  many indices e.g. those of hom morphisms where we leave a comma as
a subscript to indicate that they are hom morphisms.

Transform of the right hand factor of the domain:
\begin{align*}
 &  ([\eta,1]\mathpzc{L}(AP,-)_{,})(M_{\mathpzc{L}}(m_{QRS},1_{(1_{\mathpzc{P}(R,S)\otimes
      \mathpzc{P}(Q,R)}\otimes a_{PQ}) \alpha (a^{-1}\otimes 1)})) \\
=& ( [a^{-1},1][\eta,1]\mathpzc{L}(AP,-)_{
,
}\mathpzc{L}(\alpha,1) M_{\mathpzc{L}}(1\otimes
((\mathpzc{P}(R,S)\otimes \mathpzc{P}(Q,R))\otimes -)_{,})) (m_{QRS},1_{a_{PQ}}) \\
& \mbox{(by naturality)} \\
 =&  [a^{-1},M_{\mathpzc{L}}] 
 \mathrm{Ten}_{,
}
 (\chi_{QRS},
  1_{
    A_{PQ}})  \\
&
\mbox{(by Lemma \ref{asabovebynatofML} from \secref{asabovebynatofML})}
\end{align*}

Transform of the left hand factor of the domain:

\begin{align*}
 &
 ([\eta,1]\mathpzc{L}(AP,-)_{,})(M_{\mathpzc{L}}(m_{PQS},1_{\alpha(((M_{\mathpzc{P}}\otimes
   1)a^{-1})\otimes 1)})) \\
= & ([(M_{\mathpzc{P}}\otimes
1)a^{-1},1][\eta,1]\mathpzc{L}(AP,-)_{,}\mathpzc{L}(\alpha,1))
(m_{PQS}) \\
& \mbox{(by naturality)} \\
= & [(M_{\mathpzc{P}}\otimes 1)a^{-1},1](\chi_{PQS})
\end{align*}

Transform of the right hand factor of the codomain:
\begin{align*}
 &  ([\eta,1]\mathpzc{L}(AP,-)_{,})(M_{\mathpzc{L}}(1_{a_{RS}}, \mathpzc{L}(a,1) (\mathpzc{P}(R,S)\otimes m_{PQR}))) \\
= & [1,M_{\mathpzc{L}}]\mathrm{Ten}_{,}(1_{A_{RS}},\chi_{PQR}) \\
& \mbox{(by Lemma \ref{asabovebynatofML} from \secref{asabovebynatofML} )}
\end{align*}

Transform of the left hand factor of the codomain:
\begin{align*}
 &  ([\eta,1]\mathpzc{L}(AP,-)_{,})(M_{\mathpzc{L}}(m_{PRS}, 1_{(\mathpzc{P}(R,S)\otimes (
  (M_{\mathpzc{P}}\otimes 1_{AP})\alpha^{-1}))}))) \\
= &  [1\otimes M_{\mathpzc{P}},1](\chi_{PRS}) \\
& \mbox{(by naturality)}
\end{align*}
Thus we may define $\omega_{PQRS}$ as the transform $
(n\mathpzc{L}((1\otimes \alpha)\alpha,1)(\pi_{PQRS})$.

Similarly, it is shown that $\gamma_{PQ}$ and $\delta_{PQ}$ may be
defined as the transforms of $\lambda_{PQ}$ and $\rho_{PQ}$ (where one
has to use the first Transformation Lemma and the triangle identity).

Finally we have to show that the axioms of a Gray homomorphism are
precisely the transforms of the axioms of a pseudo algebra. Observe that
because there are only two axioms in the definition of a pseudo
algebra, 
it is crucial that by Proposition \ref{kelly2axioms} from \secref{kelly2axioms} two of the
lax algebra axioms are redundant for a pseudo algebra\footnote{We show
below that Gray homomorphisms correspond to locally strict
trihomomorphisms. The two redundant axioms correspond to two equations
for a trihomomorphism that hold generally: This can be shown because
they hold for strict trihomomorphisms by the left and right
normalization axiom of a  tricategory, and then they hold for a
general trihomomorphism by coherence.}.

Now consider the pentagon-like axiom. The corresponding Gray
homomorphism axiom and the pseudo algebra axiom are both composed out
of three vertical factors on each side of the axiom. Each of the
vertical factors is the horizontal composition of a nontrivial
$2$-cell and an identity $2$-cell.
Since the hom $\mathpzc{Gray}$-adjunction \eqref{homVadjtensor} of the
tensor product from \secref{homVadjtensor} is given by strict functors, it preserves vertical and
horizontal composition and it preserves identity $2$-cells. It follows
that we only have to show that the nontrivial $2$-cells in each
vertical factor match. In diagrammatic language this means that we
only have to compare the nontrivial subdiagrams.

In fact,  the determination of the transforms of the nontrivial
$2$-cells of the pseudo algebra axiom is perfectly straightforward
and similar to the identification of the transforms of $\pi_{PQRS}$'s
domain and codomain above.
For example, the transform of the interchange cell is:
\begin{align*}
 & ([\eta,1]\mathpzc{L}(AP,-)_{,}
M_{\mathpzc{L}}) (\Sigma_{m_{RST},\mathpzc{L}(\alpha(a^{-1}\otimes 1),1)(\mathpzc{P}(R,S)\otimes
  -)_{,}(m_{PQR})}) \\
= & ([\eta,1]\mathpzc{L}(AP,-)_{,}
M_{\mathpzc{L}} (1\otimes (\mathpzc{L}(\alpha(a^{-1}\otimes 1),1)(\mathpzc{P}(R,S)\otimes
-)_{,}))) (\Sigma_{m_{RST},m_{PQR}}) \\
& \mbox{(by equation \eqref{FGSigma} from \secref{FGSigma})} \\
= & ([a^{-1},1][\eta,1]\mathpzc{L}(AP,-)_{,}
\mathpzc{L}(\alpha,1)M_{\mathpzc{L}}(1\otimes (\mathpzc{P}(R,S)\otimes
-)_{,})) (\Sigma_{m_{RST},m_{PQR}}) \\
& \mbox{(by naturality of $M_{\mathpzc{L}}$, $\mathpzc{L}(AP,-)_{,}$,
  and $\eta$)} \\
= & [a^{-1},M_{\mathpzc{L}}]\mathrm{Ten}_{,}
(\Sigma_{\chi_{RST},\chi_{PQR}}) \\
& \mbox{(by Lemma \ref{asabovebynatofML} from \secref{asabovebynatofML})}
\end{align*}
This is exactly the interchange cell appearing in the pentagon-like
axiom of a Gray homomorphism.
\end{proof}

Notice that the proof above only involved the first Transformation
Lemma from \ref{someproperties}, namely Lemma \ref{asabovebynatofML}.
To show how the second Transformation Lemma i.e. Lemma
\ref{noticethesymmetry} from \secref{noticethesymmetry} enters in the
proof of Theorem \ref{graygleichtransform}, we
provide the following lemma regarding the first axiom of a
$T$-transformation.
In fact, one has to employ Lemma \ref{noticethesymmetry} already in the proof of
the correspondence for pseudo $T$-functors, but only for the axioms of a
$T$-transformation, there appears a new class of interchange cells.
Thus we skip the proof for the correspondence of pseudo $T$-functors and
Gray transformations, and for the correspondence of the data of a
$T$-transformation and the data of a Gray modification.
\begin{lemma}\label{secondaxiomGraymod}
  The transform of the first axiom \textup{\textbf{(LTA1)}} of a $T$-transformation $\alpha\co
  f\rightarrow g\co A\rightarrow B$ is precisely the second axiom \textup{\textbf{(GMA2)}} of a
  Gray modification.
\end{lemma}
\begin{proof}

First note that the $T$-transformation axiom \textbf{(LTA1)} cf. Definition
\ref{Ttransformation} from \secref{Ttransformation} is equivalent to the equations
\begin{align*}
& (\mathpzc{h}^g_P\ast 1)\diamond (1\ast
M_{\mathpzc{L}}(\Sigma_{\alpha_P,i^A_P}))\diamond
(M_{\mathpzc{L}}(A_{PP},1_{(j_P\otimes 1)\lambda^{-1}_{AP}})\ast 1)\\ &
= M_{\mathpzc{L}}(\Sigma^{-1}_{i^B_P,\alpha_P})\diamond (1\ast \mathpzc{h}^f_{P})
  \puncteq{.}
\end{align*}
where $P$ runs through the objects of $\mathpzc{P}$.
We apply the invertible strict functor $\mathpzc{L}(\lambda_{AP},1)$
to these equations, which gives the following equivalent equations:
\begin{align*}
& (\mathpzc{L}(\lambda_{AP},1)(\mathpzc{h}^g_P)\ast 1)\diamond (1\ast
M_{\mathpzc{L}}(\Sigma_{\alpha_P,\mathpzc{L}(\lambda_{AP},1)(i^A_P)}))\diamond (M_{\mathpzc{L}}(A_{PP},1_{j_P\otimes 1})\ast 1)\\ &
= M_{\mathpzc{L}}(\Sigma^{-1}_{i^B_P,\mathpzc{L}(\lambda_{AP},1)(\alpha_P)})\diamond (1\ast \mathpzc{L}(\lambda_{AP},1)(\mathpzc{h}^f_{P}))
  \puncteq{.}
\end{align*}
Here we have used naturality of $M_{\mathpzc{L}}$ and equation
\eqref{FGSigma} from \secref{FGSigma} for the manipulation of the interchange cells.
As above we only have to compare the transforms of the nontrivial
$2$-cells. The transforms of
$\mathpzc{L}(\lambda_{AP},1)(\mathpzc{h}^g_P)$ and $
\mathpzc{L}(\lambda_{AP},1)(\mathpzc{h}^f_p)$  are by definition the modifications
$M^g$ and $M^f$ of the Gray transformation corresponding to the pseudo
$T$-functors $f$ and $g$. The transform of
$M_{\mathpzc{L}}(A_{PP},1_{j_P\otimes 1})$ is by naturality
$[j_P,1](\alpha_{PP})$.

\noindent The transform of the interchange cell
$M_{\mathpzc{L}}(\Sigma_{\alpha_P,\mathpzc{L}(\lambda_{AP},1)i^A_P})$
is determined as follows:
\begin{align*}
  &  ([\eta^{AP}_{I},1]\mathpzc{L}(AP,-)_{I\otimes
    AP,BP})(M_{\mathpzc{L}}(\Sigma_{\alpha_P,\mathpzc{L}(\lambda_{AP},1)i^A_P})
  )\\
= & M_{\mathpzc{Gray}}(\mathpzc{L}(AP,-)_{
      AP,BP}\otimes ([\eta^{AP}_{I},1]\mathpzc{L}(AP,-)_{I\otimes
      AP,AP} ))(\Sigma_{
    \alpha_P, (\mathpzc{L}(\lambda_{AP},1)(i^A_P))}) \\
  &  \mbox{(by the functor axiom for $\mathpzc{L}(AP,-)$ and naturality
    of $M_{\mathpzc{Gray}}$)} \\
= & M_{\mathpzc{Gray}}(\Sigma_{
   \mathpzc{L}(AP,
    \alpha_P), [\eta^{AP}_{I},1]\mathpzc{L}(AP,-)_{I\otimes
      AP,AP}(\mathpzc{L}(\lambda_{AP},1)(i^A_P))}) \\
  &  \mbox{(by equation \eqref{FGSigma} from \secref{FGSigma})} \\
=& M_{\mathpzc{Gray}}(\Sigma_{\mathpzc{L}(AP,\alpha_P),\iota^A_P}) 
\end{align*}
The transform of the interchange cell left is:
\begin{align*}
  &  ([\eta^{AP}_{I},1]\mathpzc{L}(AP,-)_{I\otimes
    AP,BP})(M_{\mathpzc{L}}(\Sigma^{-1}_{i^B_P,\mathpzc{L}(\lambda_{AP},1)(\alpha_P)})
  )\\
= & ([\eta^{AP}_{I},1]\mathpzc{L}(AP,-)_{I\otimes
    AP,BP})((M_{\mathpzc{L}}(1\otimes (I\otimes -)_,))(\Sigma^{-1}_{\mathpzc{L}(\lambda_{AP},1)(i^B_P),
  \alpha_P}) )\\ 
& \mbox{(by naturality of $\lambda$, equation \eqref{FGSigma} from \secref{FGSigma}, and
  extraordinary naturality of $M_{\mathpzc{Gray}}$)} \\
= & (M_{\mathpzc{Gray}}((\mathpzc{L}(-,BP)_{AP,BP})\otimes
  (
  [\eta^{AP}_{I},1]\mathpzc{L}(AP,-)_{(I\otimes AP,AP}))c)(\Sigma^{-1}_{\mathpzc{L}(\lambda_{AP},1)(i^B_P),
  \alpha_P}) )\\ 
& \mbox{(by Lemma \ref{noticethesymmetry} from \secref{noticethesymmetry})} \\
= & (M_{\mathpzc{Gray}}((\mathpzc{L}(-,BP)_{AP,BP})\otimes
  (
  [\eta^{AP}_{I},1]\mathpzc{L}(AP,-)_{(I\otimes AP,AP})))(\Sigma_{\alpha_P,\mathpzc{L}(\lambda_{AP},1)(i^B_P)
  }) )\\ 
  & \mbox{(by equation \eqref{cSigma} from \secref{cSigma})} \\
=& M_{\mathpzc{Gray}}(\Sigma_{\mathpzc{L}(\alpha_P,BP),\iota^B_P}) \\
& \mbox{(by equation \eqref{FGSigma} from \secref{FGSigma})} 
\end{align*}

\end{proof}
This finishes our exhibition of the proof of Theorem
\ref{graygleichtransform}, and we end this paragraph with the following trivial
corollary of Theorem \ref{graygleichtransform}:
\begin{corollary}\label{graygleichps}
  Given $\mathpzc{Gray}$-categories $\mathpzc{P}$ and $\mathpzc{L}$,
  there is a $\mathpzc{Gray}$-category $\mathrm{Gray}(\mathpzc{P},\mathpzc{L})$ with objects Gray
  homomorphisms, $1$-cells Gray transformations, $2$-cells Gray
  modifications, and $3$-cells Gray perturbations. If $\mathpzc{P}$ is
  small and $\mathpzc{L}$ is cocomplete,
  $\mathrm{Gray}(\mathpzc{P},\mathpzc{L})$ is uniquely characterized 
  by the requirement that the correspondence from Theorem
  \ref{graygleichtransform} induces
  an isomorphism
  \begin{align*}
   \mathrm{Ps}\text{-}T\text{-}\mathrm{Alg} \cong  \mathrm{Gray}(\mathpzc{P},\mathpzc{L})
  \end{align*}
 of $\mathpzc{Gray}$-categories 
  for
  \mbox{$T=[H,1]\mathrm{Lan}_H\co[\mathrm{ob}\mathpzc{P},\mathpzc{L}]\rightarrow
  [\mathrm{ob}\mathpzc{P},\mathpzc{L}]$}. \qed
\end{corollary}

\begin{remark}
  Strictly speaking, $\mathrm{Gray}(\mathpzc{P},\mathpzc{L})$ can of
  course only
  inherit the $\mathpzc{Gray}$-category structure from
  $\mathrm{Ps}\text{-}T\text{-}\mathrm{Alg}$ in the situation that the
  left Kan extension $\mathrm{Lan}_H$ along $H\co
  \mathrm{ob}\mathpzc{P}\rightarrow \mathpzc{P}$ exists and has the
  explicit description from \ref{subsectionexplicit}, e.g. if
  $\mathpzc{P}$ is small and $\mathpzc{L}$ is cocomplete, but in fact the
  prescriptions obtained in this case for the local $2$-category
  structure and the $\mathpzc{Gray}$-category structure of
  $\mathrm{Gray}(\mathpzc{P},\mathpzc{L})$ are also valid if we are
  not in this situation i.e. if there are no restrictions on $\mathpzc{P}$
  and $\mathpzc{L}$ cf. Theorem \ref{graygleichls} below.
\end{remark}

\subsection{The correspondence with locally strict
  trihomomorphisms}\label{subsectionls}
The next theorem forms the second part of the promised correspondence of pseudo
algebras and locally strict trihomomorphisms, which is then proved in
Theorem \ref{psgleichls} below.
\begin{theorem}\label{graygleichls}
Given $\mathpzc{Gray}$-categories $\mathpzc{P}$ and $\mathpzc{L}$, the $\mathpzc{Gray}$-category
  $\mathrm{Gray}(\mathpzc{P},\mathpzc{L})$ is isomorphic as a
  $\mathpzc{Gray}$-category to the full sub-$\mathpzc{Gray}$-category
  $\mathpzc{Tricat}_{\mathrm{ls}}(\mathpzc{P},\mathpzc{L})$ of
$\mathpzc{Tricat}(\mathpzc{P},\mathpzc{L})$ determined by the locally
strict trihomomorphisms.\qed
\end{theorem}
Again, we just indicate how the proof works,
but we want to stress that the steps of the proof not displayed
have been explicitly checked and they 
are indeed entirely analogous to the situations we discuss in the lemmata below.

\begin{lemma}
  Given $\mathpzc{Gray}$-categories $\mathpzc{P}$ and $\mathpzc{L}$,
  there is a one-to-one correspondence between locally strict
  trihomomorphisms $\mathpzc{P}\rightarrow \mathpzc{L}$ and Gray
  homomorphisms $\mathpzc{P}\rightarrow \mathpzc{L}$.
\end{lemma}
\begin{proof}
  Comparing the definitions, the first thing to be noticed is that
  Definitions \ref{Grayhomomorphism}-\ref{Grayperturbation} from
  \secref{Grayperturbation} involve
  considerably less cell data than the tricategorical definitions
  cf. \cite[4.3]{gurskicoherencein}. Since $\mathpzc{P}$ and
  $\mathpzc{L}$ are $\mathpzc{Gray}$-categories, these supernumerary
  cells are all trivial. 

Consider, for example, a locally strict trihomomorphism $A\co
\mathpzc{P}\rightarrow \mathpzc{L}$.
Recall that this is given by
\begin{enumerate*}
 [label=(\roman*),ref= Def. \ref{defcreate}~(\roman*)]
\item a function on the objects $P\mapsto AP$;
\item for objects $P,Q\in \mathpzc{P}$, a strict
 functor  $A_{PQ}\co \mathpzc{P}(P,Q)\rightarrow
  \mathpzc{L}(AP,AQ)$;
\item  for objects $P,Q,R\in \mathpzc{P}$, an adjoint equivalence
\end{enumerate*}
  \begin{align*}
    (\chi_{PQR},\chi^\bullet_{PQR})\co M_{\mathpzc{L}}C(A_{QR}\times
    A_{PQ})\Rightarrow A_{PR}M_{\mathpzc{P}}C
\puncteq{,}
  \end{align*}
  where $C$ is again the universal cubical functor,
  and an adjoint
  equivalence
  \begin{align*}
    (\iota_P,\iota_P^\bullet)\co j_{AP}\Rightarrow A_{PP}j_{P}
  \end{align*}
  if $P=Q=R$;
 \begin{enumerate*}
 [label=(\roman*),ref= Def. \ref{defcreate}~(\roman*),resume]
\item and three families 
  $\omega,\gamma,\delta$ of invertible modifications subject to two axioms.
\end{enumerate*}
Up to this point, this looks very similar to Definition
\ref{Grayhomomorphism} from \secref{Grayhomomorphism}, the difference being in the form of domain and
codomain of the adjoint equivalence $(\chi_{PQR},\chi^\bullet_{PQR})$. However, observing
that $C(A_{QR}\times A_{PQ})=(A_{QR}\otimes A_{PQ})C$ by naturality of
$C$, it is clear
from Proposition \ref{internal2isobicatc} from \secref{internal2isobicatc} that this corresponds to an
adjoint equivalence $(\hat\chi_{PQR},\hat\chi_{PQR}^\bullet)$ as in
the definition of a Gray homomorphism such that
$C^\ast(\hat\chi_{PQR},\hat\chi_{PQR}^\bullet)=(\chi_{PQR},\chi_{PQR}^\bullet)$.

Next, given objects $P,Q,R,S\in\mathpzc{P}$, the modification
$\omega_{PQRS}$ has the form\footnote{We again remark that we here and
  in fact always use   a different bracketing than the one employed in
  \cite{gurskicoherencein}. Thus, specifying a modification as the one
  displayed is equivalent to  specifying a modification as in \cite{gurskicoherencein}.}
\begin{align*}
  \omega_{PQRS}\co \quad\quad &(((M_{\mathpzc{P}}C)\times 1)a^{-1}_{\times})^\ast(\chi_{PQS})\ast
  (a^{-1}_\times)^\ast (M_{\mathpzc{L}}C)_\ast( \chi_{QRS} \times 1_{ A_{PQ}}) \\ 
\Rrightarrow \quad & 
(1\times (M_{\mathpzc{P}}C))^{\ast}(\chi_{PRS}) \ast
(M_{\mathpzc{L}}C)_\ast (1_{ A_{RS}}\times \chi_{PQR})
\puncteq{}
\end{align*}
where we used strictness of the local functors, and where we made the
monoidal structure of the cartesian product explicit, e.g. $a_\times$
denotes the corresponding associator.





This is the same as
\begin{align*}
  \omega_{PQRS}\co \quad\quad & (C(1\times C))^\ast([(M_{\mathpzc{P}}\otimes 1)a^{-1},1](\hat\chi_{PQS})\ast
  [a^{-1},M_{\mathpzc{L}}]( \mathrm{Ten}( \hat\chi_{QRS}, 1_{ A_{PQ}}))) \\ 
\Rrightarrow \quad & 
(C(1\times C))^\ast([1\otimes M_{\mathpzc{P}},1](\hat\chi_{PRS}) \ast
[1,M_{\mathpzc{L}}](\mathrm{Ten}(1_{ A_{RS}}, \hat\chi_{PQR})))
\puncteq{.}
\end{align*}
Here we have used that $aC(C\times 1)a_{\times}^{-1}=C(1\times C)$ on the
left hand side, that $C$ commutes with the hom functors of $\mathrm{Ten}$ and
$\times$ i.e.
\begin{align*}
 C_\ast
  \times_{(X,Y),(X',Y')}=C^\ast\mathrm{Ten}_{(X,Y),(X',Y')}C\co
  [X,X']\times [Y,Y']\rightarrow [X\times Y,X'\otimes Y']
\end{align*}
where
$\times_{(X,Y),(X',Y')}\co [X,X']\times [Y,Y']\rightarrow [X\times
Y,X'\times Y']$ (on objects, this is naturality of $C$), that $(FG)_\ast=F_\ast G_\ast$ and $(FG)^\ast=G^\ast
F^\ast$, that $(C(1\times C))^\ast$ is strict, that $(-)^\ast$ and $(-)_\ast$ coincide with the partial
functors of $[-,-]$ for strict functors, and that apart from the
cubical functor $C$, all functors are strict.

By Theorem \ref{internal2isobicatcmulti} from \secref{internal2isobicatcmulti}, $\omega_{PQRS}$ corresponds to an
invertible modification $\hat\omega_{PQRS}$ as in the definition of a Gray homomorphism
such that $(C(1\times C))^\ast\hat\omega_{PQRS}=\omega_{PQRS}$.

Given objects $P,Q\in \mathpzc{P}$, the modifications $\gamma_{PQ}$ and $\delta_{PQ}$ are of the same form
as in the definition of a Gray homomorphism:
Using strictness of the local functors, $\gamma_{PQ}$
is seen to be of
the form
\begin{align*}
  \gamma_{PQ}\co & ((j_Q\times
    1)l^{-1}_{\times})^\ast(\chi_{PQQ}) \ast
    (l^{-1}_{\times})^\ast (M_{\mathpzc{L}}C)_\ast(\iota_Q
    \times 1_{A_{PQ}})
    \Rrightarrow 1_{A_{PQ}}
\puncteq{,}
\end{align*}
and this is clearly the same as
\begin{align*}
  \gamma_{PQ}\co & [(j_Q\otimes
    1)l^{-1}_{\mathpzc{P}(P,Q)},1](\hat\chi_{PQQ}) \ast
    [l^{-1}_{\mathpzc{P}(P,Q)},M_{\mathpzc{L}}](\mathrm{Ten}(\iota_Q,1_{A_{PQ}}))
    \Rrightarrow 1_{A_{PQ}}
\puncteq{.}
\end{align*}
Similarly, it is shown that $\delta_{PQ}$ is of the form required in
the definition of a Gray homomorphism.

Finally, we have to compare the axioms. 
By Theorem \ref{internal2isobicatcmulti} from \secref{internal2isobicatcmulti}, the axioms of a
 trihomomorphism correspond to equations involving the components of the modifications
 $\hat\omega_{PQRS}$, $\gamma_{PQ}$,  and $\delta_{PQ}$. On the other
 hand, the axioms of a Gray homomorphism are equations
for the modifications $\hat\omega_{PQRS}$, $\gamma_{PQ}$,  and
$\delta_{PQ}$ themselves (involving an interchange modification in the
case of the pentagon-like axiom).
In fact, apart from the interchange cell in the pentagon-like axiom, it is obvious that the
components of the nontrivial modifications in the Gray homomorphism axioms are
precisely the nontrivial $2$-cells in the axioms of the corresponding
trihomomorphism axioms. Note here that the correspondence of Theorem
\ref{internal2isobicatcmulti} from \secref{internal2isobicatcmulti} is
trivial on components.

Recall that the interchange cell in the pentagon-like axiom of a
Gray-homomorphism is given by 
$[a^{-1},M_{\mathpzc{L}}]\mathrm{Ten}_{,}
(\Sigma_{\chi_{RST},\chi_{PQR}})$.
We maintain that at the object
\begin{align*}
  (g,(h,(i,j)))\in\mathpzc{P}(S,T)\otimes(\mathpzc{P}(R,S)\otimes
  (\mathpzc{P}(Q,R)\otimes \mathpzc{P}(P,Q)))
\puncteq{,}
\end{align*}
the component $([a^{-1},M_{\mathpzc{L}}]\mathrm{Ten}_{,}
(\Sigma_{\chi_{RST},\chi_{PQR}}))_{ghij}$
is given by $M_{\mathpzc{L}}(\Sigma^{-1}_{\chi_{gh},\chi_{ij}})$.
This is because evaluation in $\mathpzc{Gray}$ is in this case given by taking
components and now equation
\eqref{Tenhom} from \secref{Tenhom} for the strict hom functor
$\mathrm{Ten}_,$ implies that
\begin{align*}
  ([a^{-1},M_{\mathpzc{L}}]\mathrm{Ten}_{,}
(\Sigma_{\chi_{RST},\chi_{PQR}}))_{g(h(ij))}=M_{\mathpzc{L}}(\mathrm{Ten}_{,}(\Sigma_{\chi_{RST},\chi_{PQR}}))_{(gh)(ij)}=M_{\mathpzc{L}}(\Sigma_{\chi_{gh},\chi_{ij}})
\puncteq{.}
\end{align*}
This is exactly the interchange cell on the right hand side of the
corresponding axiom for a trihomomorphism, cf. \cite[p. 68]{gurskicoherencein}.
\end{proof}

As noted above, in the proof of Theorem \ref{graygleichls}, there appear
additional classes of interchange cells of which the components have
to be compared to the interchange cells appearing in the definitions
of the data and the $\mathpzc{Gray}$-category structure of
$\mathpzc{Tricat}(\mathpzc{P},\mathpzc{L})$.
To give examples for these classes, we skip the proof for the
correspondence of Gray transformations and tritransformations and for
the correspondence of the data of Gray modifications and
trimodifications, and we 
 come back  to the second axiom of a Gray modification
cf. Lemma \ref{secondaxiomGraymod} from \secref{secondaxiomGraymod}:

\begin{lemma}
  The components of the interchange cells in the second axiom of a
  Gray modification correspond precisely to the interchange cells
  appearing in the second axiom of a trimodification.
\end{lemma}
\begin{proof}
  The interchange cell appearing on the left hand side of the Gray
  modification axiom \textbf{(GMA2)} is
$M_{\mathpzc{Gray}}(\Sigma_{\mathpzc{L}(AP,\alpha_P),\iota^A_P})$.
  Note that $\Sigma_{\mathpzc{L}(AP,\alpha_P),\iota^A_P}$ is an
  interchange $2$-cell in the Gray product
  \begin{align*}
    [\mathpzc{L}(AP,AP),\mathpzc{L}(AP,BP)]\otimes
    [I,\mathpzc{L}(AP,AP)] \puncteq{,}
  \end{align*}
  and now we have to determine how $M_{\mathpzc{Gray}}$ acts on such
  an interchange cell. Recall that $M_{\mathpzc{Gray}}\co [Y,Z]\otimes
  [X,Y]\rightarrow [X,Z]$ is defined by
  \begin{align*}
    e_Z^X(M_{\mathpzc{Gray}}\otimes 1_X)= e^Y_Z(1\otimes e^X_Y)a^{-1}
    \puncteq{.}
  \end{align*}
  For the component of
  $M_{\mathpzc{Gray}}(\Sigma_{\mathpzc{L}(AP,\alpha_P),\iota^A_P})$ at
  the single object $\ast\in I$ this implies that
  \begin{align*}
    (M_{\mathpzc{Gray}}\Sigma_{\mathpzc{L}(AP,\alpha_P),\iota^A_P})_\ast
    & = (e^{\mathpzc{L}(AP,AP)}_{\mathpzc{L}(AP,BP)}
    (\mathpzc{L}(AP,-)_{AP,BP}\otimes 1)) (\Sigma_{\alpha_P,
      (\iota^A_P)_\ast}) \\
    & \quad\mbox{(by equation \eqref{FGSigma} from \secref{FGSigma})}\\
    & = M_{\mathpzc{L}} (\Sigma_{\alpha_P,
      (\iota^A_P)_\ast}) \\
    & \quad\mbox{(by definition of $\mathpzc{L}(AP,-)$ see
      \eqref{covapartialhom} from \secref{covapartialhom})}
  \end{align*}
 In fact, this is exactly the interchange cell for the left hand side
  of the axiom in \cite[p. 77 ]{gurskicoherencein}.

  Similarly, the component of the interchange $2$-cell
  $M_{\mathpzc{Gray}}(\Sigma_{\mathpzc{L}(\alpha_P,BP),\iota^B_P})$ at
  the single object $\ast\in I$, is given by
  \begin{align*}
    (M_{\mathpzc{Gray}}(\Sigma_{\mathpzc{L}(\alpha_P,BP),\iota^B_P}))_\ast
    & =
    e^{\mathpzc{L}(BP,BP)}_{\mathpzc{L}(AP,BP)}(\mathpzc{L}(-,BP)_{AP,BP}\otimes
    1)(\Sigma_{\alpha_P,(\iota^B_P)_\ast})\\
    & =  M_{\mathpzc{L}}c(\Sigma_{\alpha_P,(\iota^B_P)_\ast}) \\
    & \quad\mbox{(by definition of $\mathpzc{L}(-BP)$ see
      \eqref{contrapartialhom} from \secref{contrapartialhom})} \\
    & =  M_{\mathpzc{L}}(\Sigma^{-1}_{(\iota^B_P)_\ast,\alpha_P}) \\
    & \quad\mbox{(by equation \eqref{cSigma} from \secref{cSigma})}
    \puncteq{.}
  \end{align*}
  In fact, this is exactly the interchange cell for the right hand
  side of the axiom in \cite[p. 77]{gurskicoherencein}.
\end{proof}

\begin{remark}
  It is entirely analogous to show that the composition laws as given
  in \cite[Th. 9.1 and 9.3]{gurskicoherencein} and Definitions
  \ref{localPSTAlg} and \ref{globalPSTAlg} from \secref{globalPSTAlg} of the two
  $\mathpzc{Gray}$-categories coincide under the correspondence.
This concludes our exhibition of the critical ingredients of the proof
of Theorem \ref{graygleichls}.
\end{remark}

Combining Corollary \ref{graygleichps} from \secref{graygleichps} and
Theorem \ref{graygleichls}, we have proved our main theorem:
\begin{theorem}\label{psgleichls}
 Let  $\mathpzc{P}$ be a small $\mathpzc{Gray}$-category and
  $\mathpzc{L}$ be a cocomplete $\mathpzc{Gray}$-category, and let $T$
  be the monad corresponding to the Kan adjunction.
  Then the $\mathpzc{Gray}$-category
  $\mathrm{Ps}\text{-}T\text{-}\mathrm{Alg}$  is isomorphic to the
  full sub-$\mathpzc{Gray}$-category of
  $\mathpzc{Tricat}(\mathpzc{P},\mathpzc{L})$ determined by the
  locally strict trihomomorphisms. \qed
\end{theorem}

The identification of the functor category $[\mathpzc{P},\mathpzc{L}]$
with $[\mathrm{ob}\mathpzc{P},\mathpzc{L}]^T$ in Theorem
  \ref{strictmonadicitytheorem} from \ref{sectionstrict}., and the
  coherence result for $\mathrm{Ps}\text{-}T\text{-}\mathrm{Alg}$
  given in Corollary \ref{corollarycoherence}
  from \secref{corollarycoherence}, then prove the following coherence
  theorem for $\mathpzc{Tricat}_{\mathrm{ls}}(\mathpzc{P},\mathpzc{L})$:
\begin{theorem}\label{theoremleftadjls} Let  $\mathpzc{P}$ be a small $\mathpzc{Gray}$-category and
  $\mathpzc{L}$ be a cocomplete $\mathpzc{Gray}$-category.
  Then the inclusion $i\co [\mathpzc{P},\mathpzc{L}]\rightarrow \mathpzc{Tricat}_{\mathrm{ls}}(\mathpzc{P},\mathpzc{L})$ of the functor $\mathpzc{Gray}$-category
  $[\mathpzc{P},\mathpzc{L}]$ into the $\mathpzc{Gray}$-category
  $\mathpzc{Tricat}_{\mathrm{ls}}(\mathpzc{P},\mathpzc{L})$ of locally
  strict trihomomorphisms has a left
  adjoint such that the components $\eta_A\co A\rightarrow  iLA$ for
   objects $A\in \mathpzc{Tricat}_{\mathrm{ls}}(\mathpzc{P},\mathpzc{L}) $ of
  the unit of this adjunction are internal biequivalences. \qed
\end{theorem}

\begin{example*}
  Let $\mathpzc{P}$ be a small $\mathpzc{Gray}$-category. Recall that
  $\mathpzc{Gray}$ considered as a $\mathpzc{Gray}$-category is
  complete and cocomplete cf. Lemma \ref{graycocomplete} from \secref{graycocomplete}.
Thus Theorem \ref{theoremleftadjls} applies for
$\mathpzc{L}=\mathpzc{Gray}$.  As a consequence, a locally strict
trihomomorphism $\mathpzc{P}\rightarrow \mathpzc{Gray}$ which is
nothing else than a locally strict
$\mathpzc{Gray}$-valued presheaf is biequivalent to a
$\mathpzc{Gray}$-functor $\mathpzc{P}\rightarrow \mathpzc{L}$.

In particular, let $\mathpzc{P}$ be a category $C$ considered as a discrete
$\mathpzc{Gray}$-category. Then locally strict trihomomorphisms
$C\rightarrow \mathpzc{L}$ are the homomorphisms of interest, and we
have proved that any such homomorphism is biequivalent to a
$\mathpzc{Gray}$-functor $C\rightarrow \mathpzc{L}$.
\end{example*}

\bibliographystyle{plain}
\bibliography{references}

\def\polhk#1{\setbox0=\hbox{#1}{\ooalign{\hidewidth
  \lower1.5ex\hbox{`}\hidewidth\crcr\unhbox0}}}
\begin{thebibliography}{10}

\bibitem{benabou67}
Jean B{\'e}nabou.
\newblock Introduction to bicategories.
\newblock In {\em Reports of the {M}idwest {C}ategory {S}eminar}, pages 1--77.
  Springer, Berlin, 1967.

\bibitem{blackwell}
R.~Blackwell, G.~M. Kelly, and A.~J. Power.
\newblock Two-dimensional monad theory.
\newblock {\em J. Pure Appl. Algebra}, 59(1):1--41, 1989.

\bibitem{dubuc}
Eduardo~J. Dubuc.
\newblock {\em Kan extensions in enriched category theory}.
\newblock Lecture Notes in Mathematics, Vol. 145. Springer-Verlag, Berlin-New
  York, 1970.

\bibitem{gordonpowerstreet}
R.~Gordon, A.~J. Power, and Ross Street.
\newblock Coherence for tricategories.
\newblock {\em Mem. Amer. Math. Soc.}, 117(558):vi+81, 1995.

\bibitem{gray}
John~W. Gray.
\newblock {\em Formal category theory: adjointness for {$2$}-categories}.
\newblock Lecture Notes in Mathematics, Vol. 391. Springer-Verlag, Berlin-New
  York, 1974.

\bibitem{gurskialgebraic}
Nick Gurski.
\newblock {\em An algebraic theory of tricategories}.
\newblock ProQuest LLC, Ann Arbor, MI, 2006.
\newblock Thesis (Ph.D.)--The University of Chicago.

\bibitem{gurskicoherencein}
Nick Gurski.
\newblock {\em Coherence in three-dimensional category theory}, volume 201 of
  {\em Cambridge Tracts in Mathematics}.
\newblock Cambridge University Press, Cambridge, 2013.

\bibitem{kellyonmaclanes64}
G.~M. Kelly.
\newblock On {M}ac{L}ane's conditions for coherence of natural associativities,
  commutativities, etc.
\newblock {\em J. Algebra}, 1:397--402, 1964.

\bibitem{kelly}
G.~M. Kelly.
\newblock {\em Basic concepts of enriched category theory}, volume~64 of {\em
  London Mathematical Society Lecture Note Series}.
\newblock Cambridge University Press, Cambridge, 1982.

\bibitem{kellystreet74}
G.~M. Kelly and Ross Street.
\newblock Review of the elements of {$2$}-categories.
\newblock In {\em Category {S}eminar ({P}roc. {S}em., {S}ydney, 1972/1973)},
  pages 75--103. Lecture Notes in Math., Vol. 420. Springer, Berlin, 1974.

\bibitem{lackcodescent2002}
Stephen Lack.
\newblock Codescent objects and coherence.
\newblock {\em J. Pure Appl. Algebra}, 175(1-3):223--241, 2002.
\newblock Special volume celebrating the 70th birthday of Professor Max Kelly.

\bibitem{leinsterbasic98}
Tom Leinster.
\newblock Basic bicategories.
\newblock {\em arXiv preprint math.CT/9810017}, 589, 1998.

\bibitem{lintonrelative69}
F.~E.~J. Linton.
\newblock Relative functorial semantics: {A}djointness results.
\newblock In {\em Category {T}heory, {H}omology {T}heory and their
  {A}pplications, {III} ({B}attelle {I}nstitute {C}onference, {S}eattle,
  {W}ash., 1968, {V}ol. {T}hree)}, pages 384--418. Springer, Berlin, 1969.

\bibitem{maclane71}
Saunders MacLane.
\newblock {\em Categories for the working mathematician}.
\newblock Springer-Verlag, New York-Berlin, 1971.
\newblock Graduate Texts in Mathematics, Vol. 5.

\bibitem{powergeneral}
A.~J. Power.
\newblock A general coherence result.
\newblock {\em J. Pure Appl. Algebra}, 57(2):165--173, 1989.

\bibitem{powerthreedimensional}
John Power.
\newblock Three dimensional monad theory.
\newblock In {\em Categories in algebra, geometry and mathematical physics},
  volume 431 of {\em Contemp. Math.}, pages 405--426. Amer. Math. Soc.,
  Providence, RI, 2007.

\bibitem{streetformal72}
Ross Street.
\newblock The formal theory of monads.
\newblock {\em J. Pure Appl. Algebra}, 2(2):149--168, 1972.

\end{thebibliography}

\end{document}